\documentclass[MSNbibl,number,citesort,seceqn,dvips]{arxbj}
\usepackage{upgreek}
\aid{0}
\volume{20}
\issue{4}
\pubyear{2014}
\firstpage{2169}
\lastpage{2216}
\doi{10.3150/13-BEJ554} %kopijuoti is PTS

\makeatletter
\mathchardef\varPsi="0109
\newcommand{\rrVert}{\Vert}
\newcommand{\rrvert}{\vert}
\newcommand{\llVert}{\Vert}
\newcommand{\llvert}{\vert}
\newremark{remark}{Remark}[section]
\newtheorem{theorem}[remark]{Theorem}
\newtheorem{proposition}[remark]{Proposition}
\newtheorem{lemma}[remark]{Lemma}
\newtheorem{Hypothesis}{Hypothesis}

\newcommand{\trup}[2]{{#1}/{#2}}
\renewcommand{\epsilon}{\varepsilon}
\newcommand{\eqref}[1]{(\ref{#1})}
\newcommand{\IR}{\mathbb{R}}
\newcommand{\R}{\mathbb{R}}
\newcommand{\IN}{\mathbb{N}}
\newcommand{\Pb}{\mathbb{P}}
\newcommand{\E}{\mathbb{E}}
\newcommand{\tf}{\mathcal{F}}
\newcommand{\hac}{\mathcal{H}}
\newcommand{\hact}{\mathcal{H}_T}
\makeatother

\begin{document}
\begin{frontmatter}

\title{Approximation of a stochastic wave equation in dimension three,
with application to a support theorem in H\"older norm}
\runtitle{Approximation and support theorem}

\begin{aug}
%%%% inicialai - be tarpu
\author{\inits{F.J.}\fnms{Francisco J.} \snm{Delgado-Vences}\thanksref{e1}\ead[label=e1,mark]{javier.delgado@ub.edu}} \and
\author{\inits{M.}\fnms{Marta} \snm{Sanz-Sol\'e}\corref{}\thanksref{e2}\ead[label=e2,mark]{marta.sanz@ub.edu}}
%%\runauthor{} %% auto
\address{Facultat de Matem\`atiques, Universitat de Barcelona, Gran
Via, 585
E-08007 Barcelona, Spain.\\ \printead{e1,e2}}
\end{aug}

% HISTORY:
\received{\smonth{12} \syear{2012}}

% ABSTRACT
%
\begin{abstract}
A characterization of the support in H\"older norm of the law of the
solution to a stochastic wave equation with three-dimensional space
variable is proved. The result is a consequence of an approximation
theorem, in the convergence of probability, for a sequence of evolution
equations driven by a family of regularizations of the driving noise.
\end{abstract}

% KEYWORDS
% visi is mazosios raides ir pagal abecele
%

\end{frontmatter}

%s1 #&#
\section{Introduction}
\label{s1}

In this paper, we consider a stochastic wave equation with
three-dimensional spatial variable, and we prove a characterization of
the topological support of the law of the solution in a space of
H\"older continuous functions.

%%%%%%The model
We focus on the stochastic partial differential equation
%
%e1.1 #&#
\begin{eqnarray}
\label{s1.1} \biggl(\frac{\partial^2}{\partial t^2} - \Delta \biggr) u(t,x) &=& \sigma
\bigl(u(t,x) \bigr) \dot M(t,x) + b \bigl(u(t,x) \bigr),
\nonumber
\\[-8pt]
\\[-8pt]
u(0,x) &=& \frac{\partial}{\partial t}u(0,x) = 0,
\nonumber
\end{eqnarray}
where $\Delta$ denotes the Laplacian on $\IR^3$, $T>0$ is fixed,
$t\in\,]0,T]$ and $x\in\IR^3$. The non-linear terms are defined by
functions $\sigma, b\dvtx  \IR\rightarrow\IR$. The notation $\dot
M(t,x)$ refers to the formal derivative of a Gaussian random field
white in the time variable and with a non-trivial covariance in space.
More explicitly, on a complete probability space $(\Omega, \mathcal
{G}, \Pb)$ we consider a Gaussian process
$M=\{M(\varphi), \varphi\in\mathcal{C}_0^\infty(\IR^{1+3})\}$,
where $\mathcal{C}_0^\infty(\IR^{1+3})$ denotes the space of
infinitely differentiable functions with compact support. We assume
that $E(M(\varphi))=0$ and that the covariance function of $M$ is
given by
%
%e1.2 #&#
\begin{equation}
\label{s1.2} \E\bigl(M(\varphi)M(\psi)\bigr)=\int_{\R^+}
\mathrm{d}s\int_{\R^3} \Gamma(\mathrm{d}x) \bigl(\varphi(s,
\cdot)\star\tilde\psi(s,\cdot) \bigr) (x),
\end{equation}
where ``$\star$'' denotes the convolution operator in the spatial
argument and $\tilde\psi(t,x)=\psi(t,-x)$. We suppose that $\Gamma$
is a measure on
$\IR^3$ absolutely continuous with respect to the Lebesgue measure
with density $f$ given by
%
%e1.3 #&#
\begin{equation}
\label{s1.3} f(x)=\vert x\vert^{-\beta},\qquad x\in\IR^3
\setminus\{0\}, \beta\in\,]0,2[.
\end{equation}
Let $\mathcal{S}(\R^3)$ be the space of rapidly decreasing functions
on $\R^3$. We denote by
$\tf$ the Fourier transform operator defined by
\[
\tf\varphi(\xi)= \int_{\R^3} \varphi(x)\exp\bigl(-2\uppi \mathrm{i} (
\xi\cdot x)\bigr) \,\mathrm{d}x,
\]
where the notation ``$\cdot$'' stands for the Euclidean inner product.
The covariance function \eqref{s1.2} can also be written as
\[
\E\bigl(M(\varphi)M(\psi)\bigr)= \int_0^\infty
\mathrm{d}s \int_{\IR^3}\mu(\mathrm{d}\xi) \tf\varphi(s) (\xi)
\overline{\tf\psi(s) (\xi)},
\]
where $\mu=\tf^{-1}f$.

We introduce the Hilbert space $\hac$ defined by the completion of
$\mathcal{S}(\IR^3)$ endowed with the semi-inner product
\[
\langle\varphi,\psi\rangle_\hac= \int_{\IR^3}\mu(
\mathrm{d}\xi)\tf \varphi(\xi) \overline{\tf\psi(\xi)}.
\]

Assume that $\varphi\in\hac$ is a finite measure. Then \cite{mat}, Lemma~12.12, page 162, gives
%
%e1.4 #&#
\begin{equation}
\label{fundamental} \Vert\varphi\Vert_{\hac}^2 = C \int
_{\IR^3} \bigl \vert\tf\varphi (\xi)\bigr \vert^2 |
\xi|^{-(3-\beta)}\,\mathrm{d}\xi =C\int_{\IR^3\times\IR^3}\varphi(
\mathrm{d}x)\varphi(\mathrm{d}y)\vert x-y\vert ^{-\beta} \,\mathrm{d}x \,
\mathrm{d}y,
\end{equation}
for some finite constant $C$. This identity extends easily to signed
finite measures $\varphi\in\hac$, by using the decomposition into a
difference of positive finite measures. We will apply \eqref
{fundamental} to $\varphi(\mathrm{d}x):= G(t,\mathrm{d}x)Z(t,x)$, where $G(t,\mathrm{d}x)$ is the
fundamental solution to the wave equation (the definition is given
later) and $Z(t,x)$ is an a.s. finite random variable.

The spaces $\hac$ and $\hac_t:=L^2([0,t]; \hac)$, $t\in\,]0,T]$, will
play an important role throughout the paper. It is useful to introduce
an isometric representation of theses spaces, as follows. Consider a
complete orthonormal basis $(e_j)_{j\in\mathbb{N}}\subset\mathcal
{S}(\IR^3)$ of $\hac$. Then the mappings
\[
\mathcal{I}\dvtx \hac\rightarrow{\ell}^2,\qquad
\mathcal{I}_T\dvtx \hact \rightarrow L^2 \bigl([0,T]; {
\ell}^2 \bigr)
\]
defined by
\[
\mathcal{I}(g)= \bigl(\langle g, e_j\rangle_{\hac}
\bigr)_{j\in
\mathbb{N}}, \qquad \mathcal{I}_T(\varphi) (t)= \bigl(
\bigl\langle\varphi(t,\ast), e_j\bigr\rangle_{\hac}
\bigr)_{j\in\mathbb{N}},\qquad t\in[0,T],
\]
respectively, are isometries. This provides an identification of $\hac
$, $\hact$ with ${\ell}^2$, $L^2 ([0,T]; {\ell}^2 )$,
respectively.

In a similar vein, the Gaussian process $M$ admits a representation as
a sequence $(W_j(t), t\in[0,T])_{j\in\mathbb{N}}$ of\vadjust{\goodbreak} independent
real-valued standard Brownian motions (see, e.g.,
\cite{dalang-quer}, Proposition~2.5). Indeed, this is given by the formula
\[
W_j(t):=M (1_{[0,t]}e_j ),\qquad j\in
\mathbb{N}, t\in[0,T].
\]
We refer the reader to \cite{dalangfrangos} for a rigorous derivation
of $M (1_{[0,t]}e_j )$ from the process $M$.

Along with the probability space $(\Omega, \mathcal{G}, \Pb)$, we
will consider the filtration $(\tf_t)_{t\in[0,T]}$ generated by the
process $\{W_j(t), j\in\mathbb{N}, t\in[0,T]\}$.

%%%%%%%Rigourous formulation of the equation and properties of the
%sample paths
%%%%%%%
%%%%%%%

Let $G(t)$ be the fundamental solution to the wave equation in
dimension three. It is well-known that $G(t,\mathrm{d}x)=\frac{1}{4\uppi t}\sigma_t(\mathrm{d}x)$,
where $\sigma_t(x)$ denotes the uniform surface measure on the sphere
of radius $t$ with total mass $4\uppi t^2$ (see \cite{folland}). We interpret
\eqref{s1.1} as the evolution equation
%
%e1.5 #&#
\begin{eqnarray}
\label{s1.6}
u(t,x) &=& \int_0^t \int
_{\mathbb{R}^3} G(t-s,x-y) \sigma\bigl(u(s,y)\bigr) M(\mathrm{d}s,
\mathrm{d}y)
\nonumber
\\[-8pt]
\\[-8pt]
&&{}+ \int_0^t \bigl[G(t-s,\cdot)\star
b\bigl(u(s,\cdot)\bigr) \bigr](x) \,\mathrm{d}s,\nonumber
\end{eqnarray}
where the stochastic integral (also termed stochastic convolution) in
\eqref{s1.6} is defined as
%
%e1.6 #&#
\begin{eqnarray}
\label{s1.7}
&&\int_0^t \int
_{\mathbb{R}^3} G(t-s,x-y) \sigma\bigl(u(s,y)\bigr) M(\mathrm{d}s,
\mathrm{d}y)
\nonumber
\\[-8pt]
\\[-8pt]
&&\quad:=\sum_{j\in\mathbb{N}}\int_0^t
\bigl\langle G(t-s,x-\ast) \sigma \bigl(u(s,\ast)\bigr), e_j\bigr
\rangle_{\hac} W_j(\mathrm{d}s).\nonumber
\end{eqnarray}
The notation on the left-hand side of this identity suggests an
integration with respect to the martingale measure derived from the
Gaussian process $M$, as has been considered in \cite{dalang}, while
on the right-hand side, there is an It\^o integral with respect to the
infinite-dimensional Brownian motion $W=(W_j, j\in\mathbb{N})$. It
follows from \cite{dalang-quer}, Propositions 2.6, 2.9, that if
$Y(t,x):= \sigma(u(t,x))$, $(t,x)\in[0,T]\times\IR^3$, satisfies
$\sup_{(t,x)\in[0,T]\times\IR^3}\E(\vert Y(t,x)\vert^2)<\infty$,
then both integrals coincide.

Assume that the functions $\sigma$ and $b$ are Lipschitz continuous.
With the definition \eqref{s1.7}, Theorem~4.3 in \cite{dalang-quer}
gives the existence and uniqueness of a random field solution to
equation \eqref{s1.6} satisfying
$\sup_{(t,x)\in[0,T]\times\IR^3}\E(\vert u(t,x)\vert^p)<\infty$,
for any $p\in[1,\infty[$. This means a real-valued adapted stochastic
process such that \eqref{s1.6} holds a.s. for all $(t,x)\in
[0,T]\times\R^3$.
In Theorem~\ref{ts5.1}, we will give an extension of this result.

In \cite{dss}, equation \eqref{s1.6} has been formulated using the
stochastic integral introduced by Dalang and Mueller in \cite{DM}, and
a theorem of existence and uniqueness of a random field solution is
proved. Moreover, it is also established that the sample paths are
almost surely H\"older continuous
jointly in $(t,x)$, with degree $\rho$. For the particular covariance
density given in \eqref{s1.3}, $\rho\in\, ]0,\frac{2-\beta
}{2} [$. Appealing to \cite{dalang-quer}, Proposition~2.11, this
property holds for the solution of \eqref{s1.6} with the choice of
stochastic integral made in \eqref{s1.7}. More precisely,
fix $t_0\in[0,T]$ and a compact set $K\subset\IR^3$. For any $\rho
\in\,]0,1[$, and every real function $g$, set
\[
\Vert g\Vert_{\rho,t_0,K}:= \sup_{(t,x)\in[t_0,T]\times K} \bigl |g(t,x) \bigr |+ \mathop{
\sup_{(t,x),(\bar{t},\bar{x})\in[t_0,T]\times K}}_{t\ne
\bar{t}, x\ne\bar{x}} \frac{|g(t,x)-g(\bar{t},\bar{x})| }{(|t-\bar{t}|+|x-\bar{x}|)^\rho}.
\]
We denote by $\mathcal{C}^\rho([t_0,T]\times K)$ the space of real
functions $g$ such that
$\Vert g\Vert_{\rho,t_0,K}<\infty$.
Then \cite{dss}, Theorem~4.11, shows that, for any $\rho\in\,
]0,\frac{2-\beta}{2} [$,
$\Vert u\Vert_{\rho,t_0,K} \le c$, a.s., where $c$ is a finite random
variable, a.s.
This result tells us that the law of the solution of \eqref{s1.6},
when restricted to $[t_0,T]\times K$, is a probability on $\mathcal
{C}^\rho([t_0,T]\times K)$, with $\rho\in\, ]0,\frac{2-\beta
}{2} [$.

%%%%%%%%%On the support and method

The analysis of the topological support under different kinds of norms,
like the supremum norm, H\"older norm, weighted Sobolev norms,
has been extensively studied for diffusion processes. As a
representative sample of references, let us mention \cite{stroock,mack,ba-g-l,g-n-ss}. Inspired by \cite
{aida-kusuoka-stroock}, Millet and Sanz-Sol\'e have introduced a method
for the characterization of the support of a random vector based
exclusively on approximations. For solutions to stochastic equations,
such approximations entail regularizations of the noise.
The paper \cite{millet-ss94b} illustrates the suitability of the
method by giving a very simplified proof of an extension of Stroock's
support theorem for diffusions. Moreover, the method
has also been successfully applied to several examples of stochastic
partial differential equations, like a reduced wave equation with
$d=1$, a stochastic heat equation with $d=1$ and a stochastic wave
equation with $d=2$ (see \cite{millet-ss94a,Ba-Mi-SS} and
\cite{milletss2}, resp.).

A motivation to study the support of a stochastic evolution equation
lies in the analysis of the uniqueness of invariant measures. Recently,
R. Cont and D. Fourni\'e have proved results on functional Kolmogorov
equations in the framework of a functional It\^o calculus (see \cite
{cont-fournie}). Assumptions concerning the support of some functionals
play a crucial role in their results. This provides an additional
motivation for our work.

In this paper, we apply the approximation method of \cite
{millet-ss94a} to obtain a characterization of the topological support
of the law of
$u$ (the solution to \eqref{s1.6}) in the H\"older norm $\Vert\cdot
\Vert_{\rho,t_0,K}$. The core of the work consists of an
approximation result for a family of equations more general than
equation \eqref{s1.6} by a sequence of pathwise evolution equations
obtained by a smooth approximation of the driving process $M$. In
finite dimensions, the celebrated
Wong--Zakai approximations for diffusions in the supremum norm could be
considered as an analogue. However there are two substantial
differences, first the type of equation we consider in this paper is
much more complex, and moreover we deal with a stronger topology.

For the sake of completeness, we give a brief description of the
procedure of \cite{millet-ss94a} in the particular context of this
work, and refer the reader to \cite{millet-ss94a} for further details.

Let $(\bar\Omega,\bar{\mathcal{G}}, \bar\mu)$ be the canonical
space of a standard real-valued Brownian motion on $[0,T]$. In the
sequel, the reference probability space will be
$(\Omega,\mathcal{G}, \Pb):=(\bar\Omega^{\mathbb{N}},\bar
{\mathcal{G}}^{\otimes\mathbb{N}}, \bar\mu^{\otimes\mathbb
{N}})$. By the preceding identification of $M$ with $(W_j, j\in\mathbb
{N})$, this is the canonical probability space of $M$.

Assume that there exists a measurable mapping $\xi_1\dvtx  L^2 ([0,T];
{\ell}^2 )\rightarrow\mathcal{C}^\rho([t_0,T]\times K)$, and
a sequence $w^n\dvtx  \Omega\rightarrow L^2 ([0,T]; {\ell}^2 )$
such that for every $\epsilon>0$,\vspace*{1.5pt}
%
%e1.7 #&#
\begin{equation}
\label{s1.8} \lim_{n\to\infty}\Pb \bigl\{\bigl \Vert u-\xi_1
\bigl(w^n\bigr)\bigr \Vert_{\rho
,t_0,K}>\epsilon \bigr\}=0.
\end{equation}
Then $\operatorname{supp}(u\circ\Pb^{-1}) \subset\overline{\xi_1
(L^2 ([0,T]; {\ell}^2 ) )}$, where the closure refers
to the H\"older norm $\Vert\cdot\Vert_{\rho,t_0,K}$.

Next, we assume that there exists a mapping $\xi_2\dvtx  L^2 ([0,T];
{\ell}^2 )\rightarrow\mathcal{C}^\rho([t_0,T]\times K)$ and
for any $h\in L^2 ([0,T]; {\ell}^2 )$, we suppose\vadjust{\goodbreak} that
there exist a sequence $T_n^h\dvtx  \Omega\rightarrow\Omega$ of
measurable transformations such that, for any $n\ge1$, the probability
$\Pb\circ(T_n^h)^{-1}$ is absolutely continuous with respect to $\Pb
$ and, for any $h\in L^2 ([0,T]; {\ell}^2 )$, $\epsilon>0$,
%
%e1.8 #&#
\begin{equation}
\label{s1.9} \lim_{n\to\infty}\Pb \bigl\{\bigl \Vert u
\bigl(T_n^h\bigr)-\xi_2(h)
\bigr \Vert_{\rho
,t_0,K}<\epsilon \bigr\}>0.
\end{equation}
Then $\operatorname{supp}(u\circ\Pb^{-1})\supset\overline{\xi_2
(L^2 ([0,T]; {\ell}^2 ) )}$.

For any $h\in L^2 ([0,T]; {\ell}^2 )$ (or equivalently,
$h\in\hact$), consider the deterministic evolution equation
%
%e1.9 #&#
\begin{equation}
\label{sm.h} \Phi^h(t,x)= \bigl\langle G(t-\cdot,x-\ast)\sigma
\bigl(\Phi^h(\cdot ,\ast)\bigr), h \bigr\rangle_{\mathcal{H}_t} +\int
_0^t \mathrm{d}s \bigl[G(t-s,\cdot)\star\bigl(
\Phi^h(s,\cdot)\bigr)\bigr](x).
\end{equation}
Similarly as for $u$, the mapping $(t,x)\in[t_0,T]\times K\mapsto\Phi
^h(t,x)$ belongs to $\mathcal{C}^\rho([t_0,T]\times K)$.

Let $\xi_1(h)=\xi_2(h)=\Phi^h$, and $(w^n)_{n\ge1}$ be given by
\eqref{s3.3}.
From \eqref{s3.4} and the isometric representation of $\hact$, we see that
$w^n\dvtx  \Omega\rightarrow L^2 ([0,T]; {\ell}^2 )$.
Given $h\in L^2 ([0,T]; {\ell}^2 )$, we define
%
%e1.10 #&#
\begin{equation}
\label{sm.1} T_n^h(\omega) = \omega+h-w^n.
\end{equation}
By Girsanov's theorem, the probability $\Pb\circ(T_n^h)^{-1}$ is
absolutely continuous with respect to $\Pb$.

According to \eqref{s1.8}, \eqref{s1.9}, the final objective is to
prove that
\[
\lim_{n\to\infty}\Phi^{w^n}=u,\qquad \lim
_{n\to\infty}u\bigl(T_n^h\bigr)=
\Phi^h,
\]
in probability and with the H\"older norm $\Vert\cdot\Vert_{\rho
,t_0,T}$. Then, by the preceding discussion we infer that
the support of the law of $u$ in the H\"older norm is the closure of
the set of functions $\{\Phi^h, h\in\hact\}$ (see Theorem~\ref
{tsm.1} for the rigorous statement).
Notice that the characterization of the support does not depend on the
approximating sequence $(w^n)_{n\in\mathbb{N}}$.

The paper is structured as follows. The next Section~\ref{s3} is
devoted to a general approximation result.
This is the hard core of the work (see Theorem~\ref{ts3.1}). We
postpone for a while a more extensive description of its content.
Section~\ref{sm} is devoted to the proof of the characterization of
the support of $u$. It is a corollary of Theorem~\ref{ts3.1}.
Section~\ref{s4} is of technical character. It is devoted to establish some
auxiliary results which are needed in some proofs of Section~\ref{s3}.
In the \hyperref[s5]{Appendix},
%Section~\ref{}
a theorem on existence and uniqueness of a random
field solution for a quite general evolution equation is proved. It
provides the rigorous setting for all the stochastic partial
differential equations that appear in this paper. The section also
contains two known but fundamental results used at some crucial parts
of the proofs of Sections~\ref{s3} and~\ref{sm}.

We end this introduction with a more detailed description of
Section~\ref{s3} devoted to the proof of the approximation result (see Theorem~\ref{ts3.1}).
The method we use is similar as in \cite{milletss2}, where the case
$d=2$ was studied. Nevertheless, for $d=3$ its implementation entails
substantial differences and new difficulties. The reason for this is
that the fundamental solution of the wave equation in dimension three
is a measure and not a real-valued function, as in dimension two.

As was formulated in \cite{Ba-Mi-SS}, and further developed in \cite
{milletss2}, there are two main elements in the proof of Theorem~\ref
{ts3.1}: a control on the $L^p(\Omega)$-increments in time and in
space of the processes $X$ and $X_n$, independently of $n$, and
$L^p(\Omega)$ convergence of $X_n(t,x)$ to $X(t,x)$, for any $(t,x)$.
The precise assertions are given in Theorems~\ref{ts3.2} and~\ref
{ts3.3}, respectively.

For the sake of illustration, we sketch one of the difficulties
encountered in the proof of Theorem~\ref{ts3.2}.
Consider either stochastic or pathwise integrals with integrands of the form
\[
\bigl[G(\bar t-s, x-\mathrm{d}y)-G(t-s, \bar x-\mathrm{d}y) \bigr] Z(s,y),\qquad
0<t \le\bar t\le T, x,\bar x\in\IR^3,
\]
where $Z(s,y)$ is a stochastic process. We want estimates of some norms
of these expressions
in terms of powers of the increments $|\bar t-t|$, $|\bar x-x|$.
In dimension $d=2$, $G(t,\mathrm{d}x)= G(t,x)\,\mathrm{d}x$ and the problem is solved using
direct computations on the
function differences $G(\bar t-s, x-y)-G(t-s, \bar x-y)$. For $d=3$,
this approach fails.

In \cite{dss}, this problem was tackled by passing increments of the
measure $G$ to increments of $Z$, by means of a change of variables. We
shall apply repeatedly this idea throughout the paper.
However, there are some significant differences between the arguments
in \cite{dss} and those used here. In \cite{dss}, the formulation of
equation \eqref{s1.6} is based on Dalang--Mueller stochastic integral
-- a functional type integral in the spatial variable developed in \cite
{DM}. Hence, pointwise arguments in the space variable are excluded.
Instead they use fractional Sobolev norms and Sobolev's embedding
theorem. Moreover, in \cite{dss} a regularization of the distribution
$G$ is systematically used and final results are obtained by passing to
the limit.
With the selection of the stochastic integral given in \eqref{s1.7} it
is not necessary to appeal to Sobolev's embedding theorem. Moreover,
applying \eqref{fundamental} we avoid the regularization of $G$. There
is yet another difference that deserves to be mentioned. In \cite
{dss}, non-null initial conditions were considered, while here
$u_0=v_0=0$. As a consequence, the random fields $X_n$ and $X$ possess
the stationary property described in Remark~\ref{rs3.1}. This fact is
frequently used in the proofs.

For an It\^o's stochastic differential equation, smoothing the noise
leads to a Stratonovich (or pathwise) type integral, and the correction
term between the two kinds of integrals appears naturally in the
approximating scheme. In our setting, correction terms explode and
therefore they must be avoided. Instead, a control on the growth of the
regularized noise is used.
This method was introduced in \cite{milletss2} and successfully
applied here too. The control is achieved by introducing a localization
in $\Omega$ (see \eqref{localization}). With this method, the
convergence of the approximating sequence $X_n$ to $X$ takes place in
probability.

Let us finally remark that using the method of the proof of Theorem~\ref{ts3.2}, a different but simplified proof of \cite{dss}, Theorem~4.11, in the particular case of null initial conditions can be provided.

Throughout the paper, we shall often call different positive and finite
constants by the same notation, even if they differ from one place to another.

%%%%%%%%%
%s2 #&#
\section{Approximations of the wave equation}
\label{s3}

Consider smooth approximations of $W$ defined as follows. Fix $n\in\IN
$ and consider the partition
of $[0,T]$ determined by $\frac{iT}{2^n}$, $i=0,1,\ldots,2^n$. Denote
by $\Delta_i$ the interval $[\frac{iT}{2^n},
\frac{(i+1)T}{2^n}[$ and by $|\Delta_i|$ its length. We write
$W_j(\Delta_i)$ for the increment $W_j(\frac{(i+1)T}{2^n})-
W_j(\frac{iT}{2^n})$, $i=0,\ldots,2^n-1$, $j\in\IN$. Define
differentiable approximations of $(W_j, j\in\mathbb{N})$ as follows:
\[
W^n= \biggl(W_j^n=\int
_0^\cdot\dot{W}_j^n(s)
\,\mathrm{d}s, j\in\IN \biggr),
\]
where for $j>n$, $\dot{W}_j^n=0$, and for $1\le j\le n$,
\[
\dot{W}_j^n(t)= %
\cases{ \displaystyle\sum
_{i=0}^{2^n-2} 2^{n}T^{-1}W_j(
\Delta_i)1_{\Delta
_{i+1}}(t) &\quad $\mbox{if }t\in
\bigl[2^{-n}T,T\bigr]$,
\cr
&
\cr
0 &\quad$\mbox{if }t
\in\bigl [0,2^{-n}T\bigr [$. } %
\]
Set
%
%e2.1 #&#
\begin{equation}
\label{s3.3} w^n(t,x)=\sum_{j\in\IN}
\dot{W}_j^n(t) e_j(x).
\end{equation}
It is easy to check that, for any $p\in[2,\infty[$,
%
%e2.2 #&#
\begin{equation}
\label{s3.4} \bigl \|w^n\bigr \|_{L^p(\Omega,\mathcal{H}_T)}\le C n^{{\trup{1}{2}}}2^{n/2}.
\end{equation}
In particular, from \eqref{s3.4} it follows that $w^n$ belongs to
$\mathcal{H}_T$ a.s.

In this section, we shall consider the equations
%
%e2.3 #&#
\begin{eqnarray}
\label{s3.7}X(t,x) &=&\int_0^t \int
_{\R^3} G(t-s,x-y) (A+B) \bigl(X(s,y)\bigr) M(\mathrm{d}s,
\mathrm{d}y)
\nonumber
\\
&&{}+\bigl\langle G(t-\cdot,x-\ast)D\bigl(X(\cdot,\ast)\bigr),h\bigr
\rangle_{\mathcal{H}_t}
\\
&&{}+\int_0^t \bigl[G(t-s,
\cdot)\star b\bigl(X(s,\cdot)\bigr) \bigr](x)\, \mathrm{d}s,
\nonumber
\\
\label{s3.6}X_n(t,x) &=&\int_0^t
\int_{\R^3} G(t-s,x-y) A\bigl(X_n(s,y)\bigr) M(
\mathrm{d}s,\mathrm{d}y)
\nonumber
\\
&&{} +\bigl\langle G(t-\cdot,x-\ast)B\bigl(X_n(\cdot,\ast)
\bigr),w^n\bigr\rangle_{\mathcal{H}_t}
\nonumber
\\[-8pt]
\\[-8pt]
&&{}+\bigl\langle G(t-\cdot,x-
\ast)D\bigl(X_n(\cdot,\ast)\bigr),h \bigr\rangle_{\mathcal{H}_t}\nonumber
\\
&&{} +\int_0^t \bigl[G(t-s,\cdot)\star b
\bigl(X_n(s,\cdot)\bigr) \bigr](x) \,\mathrm{d}s,
\nonumber
\end{eqnarray}
where $n\in\IN$, $h\in\mathcal{H}_T$, $w^n$ defined as in (\ref{s3.3}) and $A, B, D, b\dvtx \R\to\R$.

Moreover, we also need the slight modification of these equations
defined by
%
%e2.5 #&#
\begin{eqnarray}
\label{s3.8.2} X_n^-(t,x) &=&\int_0^{t_n}
\int_{\R^3} G(t-s,x-y) A\bigl(X_n(s,y)\bigr) M(
\mathrm{d}s,\mathrm{d}y)
\nonumber
\\
&&{} +\bigl\langle G(t-\cdot,x-\ast)B\bigl(X_n(\cdot,\ast)
\bigr)1_{[0,t_n]}(\cdot ),w^n\bigr\rangle_{\mathcal{H}_t}
\nonumber
\\[-8pt]
\\[-8pt]
&&{} +\bigl\langle G(t-\cdot,x-\ast)D\bigl(X_n(\cdot,\ast)
\bigr)1_{[0,t_n]}(\cdot ),h\bigr\rangle_{\mathcal{H}_t}
\nonumber
\\
&&{} +\int_0^{t_n} \bigl[G(t-s,\cdot)\star b
\bigl(X_n(s,\cdot)\bigr) \bigr](x)\, \mathrm{d}s,
\nonumber
\\
\label{s3.8.3} X(t,t_n,x) &=&\int_0^{t_n}
\int_{\R^3} G(t-s,x-y) (A+B) \bigl(X(s,y)\bigr) M(\mathrm{d}s,
\mathrm{d}y)
\nonumber
\\
&&{} +\bigl\langle G(t-\cdot,x-\ast)D\bigl(X(\cdot,\ast)\bigr)1_{[0,t_n]}(
\cdot ),h\bigr\rangle_{\mathcal{H}_t}
\\
&&{} +\int_0^{t_n} \bigl[G(t-s,\cdot)\star b
\bigl(X(s,\cdot)\bigr) \bigr](x) \,\mathrm{d}s,
\nonumber
\end{eqnarray}
where for any $n\in\IN$, $t\in[0,T]$, $t_n= \max\{\underline{t}_n-
2^{-n}T, 0 \}$, with
%
%e2.7 #&#
\begin{equation}
\label{s3.8.1} \underline{t}_n = \max \bigl\{k2^{-n}T,
k=1,\ldots,2^n-1\dvtx k2^{-n}T\le t \bigr\}.
\end{equation}

We will consider the following assumption.
\renewcommand{\theHypothesis}{(B)}
\begin{Hypothesis}\label{HypB}
The coefficients $A,B,D,b\dvtx \R\mapsto\R$ are
globally Lipschitz continuous.
\end{Hypothesis}

Notice that equation \eqref{s3.6} is more general than \eqref{s3.7}
and \eqref{s1.6}. In Theorem~\ref{ts5.1}, we prove a result on
existence and uniqueness of a random field solution to
a class of SPDEs which applies to equation \eqref{s3.6}.

%re2.1 #&#
\begin{remark}
\label{rs3.1}
As a consequence of Remark~\ref{r5.1}, we have the following
translation invariance of moments:
%
%e2.8 #&#
\begin{eqnarray}
\label{s3.est} \E \bigl(\bigl \vert X(t,x-y-z)-X(t,y-z)\bigr \vert^p \bigr)&=&
\E \bigl(\bigl \vert X(t,x-y)-X(t,y)\bigr \vert^p \bigr),
\nonumber
\\[-8pt]
\\[-8pt]
\E \bigl(\bigl \vert X_n(t,x-y-z)-X_n(t,y-z)
\bigr \vert^p \bigr)&=& \E \bigl(\bigl \vert X_n(t,x-y)-X_n(t,y)
\bigr \vert^p \bigr),
\nonumber
\end{eqnarray}
for any $x,y,z\in\IR^3$ and any $p\in[1,\infty[$. Consequently, a
similar property also holds for $X_n^-(t,\ast)$ and $X_n(t,t_n,\ast)$
defined in \eqref{s3.8.2}, \eqref{s3.8.3}, respectively
\end{remark}

The aim of this section is to prove the following theorem.
%%%%%%%%%MAIN THEOREM
%%%%%%%%%

%th2.2 #&#
\begin{theorem}
\label{ts3.1}
We assume Hypothesis~\textup{\ref{HypB}}. Fix $t_0>0$ and a compact set $K\subset\R
^3$. Then for any
$\rho\in\, ]0,\frac{2-\beta}{2} [$
and $\lambda>0$,
%
%e2.9 #&#
\begin{equation}
\label{s3.9} \lim_{n\to\infty} \Pb \bigl(\|X_n-X
\|_{\rho,t_0,K}> \lambda \bigr)=0.
\end{equation}
\end{theorem}

The convergence \eqref{s3.9} will be proved through several steps.
The main ingredients are local $L^p$ estimates of increments of $X_n$
and $X$, in time and in space, and a local $L^p$
convergence of the sequence $X_n(t,x)$ to $X(t,x)$.

%%%%%%
%%%%%%Localization
Let us describe the \textit{localization} procedure (see \cite
{milletss2}). Fix $\alpha>0$. For any integer $n\ge1$ and $t\in
[0,T]$, define
%
%e2.10 #&#
\begin{equation}
\label{localization} L_{n}(t)= \Bigl\{\sup_{1\le j\le n} \sup
_{0\le i\le[2^nt
T^{-1}-1]^+} \bigl |W_j(\Delta_i)\bigr |\le\alpha
n^{\trup{1}{2}}2^{-\trup
{n}{2}} \Bigr\},
\end{equation}
where $\alpha>(2\ln2)^{\trup{1}{2}}$. Notice that the sets $L_n(t)$
decrease with $t\ge0$. Moreover, in \cite{milletss2}, Lemma~2.1, it is
proved that $\lim_{n\to\infty} \Pb(L_n(t)^c)=0$.

It is easy to check that
%
%e2.11 #&#
\begin{equation}
\label{s3.101} \bigl \Vert w^n(t,\ast)1_{L_n(t)}
\bigr \Vert_{\mathcal{H}}\le C n^{\trup{3}{2}} 2^{\trup{n}{2}}.
\end{equation}
Moreover, for any $0\le t\le t'\le T$
\[
\bigl\llVert w^{n}1_{L_n(t')}1_{[t,t']}\bigr\rrVert
_{\mathcal{H}_T}\le Cn^{\trup{3}{2}}2^{\trup{n}{2}}\bigl |t'-t\bigr |^{\trup{1}{2}}.
\]
In particular, if $[t,t']\subset\Delta_i$ for some $i=0,\ldots
,2^n-1$, then
%
%e2.12 #&#
\begin{equation}
\label{s3.14} \bigl\llVert w^{n}1_{L_n(t')}1_{[t,t']}
\bigr\rrVert _{\mathcal{H}_T}\le Cn^{\trup{3}{2}}.
\end{equation}

%%%%%%%%%%%
%%%%%%%%%%% ENd localization
%%%%%%%%%%

As has been announced in the \hyperref[s1]{Introduction}, the proof of Theorem~\ref
{ts3.1} will follow from Theorems~\ref{ts3.2} and~\ref{ts3.3} below.
We denote by $\Vert\cdot\Vert_p$ the
$L^p(\Omega)$ norm.
%
%th2.3 #&#
\begin{theorem}
\label{ts3.2}
We assume Hypothesis~\textup{\ref{HypB}}. Fix $t_0\in\,]0, T[$ and a compact subset
$K\subset\IR^3$. Let $t_0\le t\le\bar{t}\le T$, $x,\bar{x}\in K$.
Then, for any
$p\in[1,\infty)$ and any $\rho\in\,]0,\frac{2-\beta}{2}
[$, there exists a positive constant $C$ such that
%
%e2.13 #&#
\begin{equation}
\label{s3.15} \sup_{n\ge1} \bigl\llVert \bigl[X_n(t,x)-X_n(
\bar{t},\bar{x}) \bigr] 1_{L_n(\bar{t})} \bigr\rrVert _p \le C \bigl( |
\bar{t}- t |+ |\bar{x}-x | \bigr)^{\rho}.
\end{equation}
\end{theorem}

%th2.4 #&#
\begin{theorem}
\label{ts3.3}
The assumptions are the same as in Theorem~\ref{ts3.2}.
Fix $t\in[t_0,T]$, $x\in\IR^3$. Then, for any
$p\in[1,\infty)$
%
%e2.14 #&#
\begin{equation}
\label{s3.16} \lim_{n\to\infty} \mathop{\sup_{{t\in[t_0,T]}}}_{{x\in K(t)}}\bigl \|
\bigl(X_n(t,x)-X(t,x) \bigr) 1_{L_n(t)}\bigr \|_p =0,
\end{equation}
where for $t\in[0,T]$,
\[
K(t)= \bigl\{x\in\R^3\dvtx d(x,K)\le T-t\bigr\},
\]
and $d$ denotes the Euclidean distance.
\end{theorem}

The proof of Theorem~\ref{ts3.2} is carried out through two steps.
First, we shall consider $t=\bar t$ and obtain \eqref{s3.15},
uniformly in $t\in[t_0,T]$. Using this, we will consider $x=\bar x$
and establish \eqref{s3.15}, uniformly in $x\in K$. We devote the next
two subsections to the proof of these results.

%%%%%%%%%%%
%%%%%%%%%%%
%%%%%%%%%%
%s2.1 #&#
\subsection{Increments in space}
\label{ss3.1}

Throughout this section, we fix $t_0\in\,]0,T[$ and a compact set
$K\subset\IR^3$. The objective is to prove the following proposition.
%
%pr2.5 #&#
\begin{proposition}
\label{pss3.1.1}
Suppose that Hypothesis~\textup{\ref{HypB}} holds. Fix $t\in[t_0,T]$ and
$x,\bar{x}\in K$. Then, for any $p\in[1,\infty)$ and $\rho\in\,
]0,\frac{2-\beta}{2} [$, there exists a finite constant $C$ such that
%
%e2.15 #&#
\begin{equation}
\label{s3.17} \sup_{n\ge0} \sup_{t\in[t_0,T]}\bigl \|
\bigl(X_n(t,x)-X_n(t,\bar{x}) \bigr) 1_{L_n(t)}
\bigr \|_p \le C |x-\bar{x}|^{\rho}.
\end{equation}
\end{proposition}

In the next lemma, we give an abstract result that will be used
throughout the proofs. We start by introducing some notation.

For a function $f\dvtx \R^3 \to\R$, we set
\begin{eqnarray*}
Df(u,x) &=&f(u+x)-f(u),
\\
\bar{D}^2f(u,x,y) &=&f(u+x+y)-f(u+x)-f(u+y)+f(u),
\\
D^2f(u,x) &=&\bar{D}^2f(u-x,x,x) =f(u-x)-2f(u)+f(u+x).
\end{eqnarray*}

%le2.6 #&#
\begin{lemma}
\label{lss3.1.1}
Consider a sequence of predictable stochastic processes $\{Z_n(t,x),
(t,x)\in[0,T]\times\R^3\}$, $n\in\mathbb{N}$, such that, for any
$p\in[2,\infty[$,
%
%e2.16 #&#
\begin{equation}
\label{s3.21} \sup_n \sup_{(t,x)\in[0,T]\times\R^3} \E
\bigl(\bigl\llvert Z_n(t,x)\bigr\rrvert ^p \bigr)< C,
\end{equation}
for some finite constant $C$.
For any $t\in[0,T]$, $x,\bar x\in\R^3$, we define
\[
I_n(t,x,\bar{x}) := \int_0^t
\mathrm{d}s \bigl \Vert Z_n(s,\ast) \bigl[G(t-s,x-\ast )-G(t-s,\bar{x}-\ast)
\bigr]\bigr  \Vert^2_{\mathcal{H}}.
\]\eject\noindent
Then, for any $p\in[2,\infty[$,
%
%e2.17 #&#
\begin{eqnarray}
\label{s3.220} &&\E \bigl( \bigl |I_n(t,x,\bar{x}) \bigr |^{{\trup{p}{2}}} \bigr)
\nonumber
\\
&&\quad \le C \biggl\{|x-\bar{x}|^{\alpha_2 {\trup{p}{2}}}+\int_0^t
\mathrm{d}s \Bigl[\sup_{y\in\R^3} \E \bigl( \bigl |Z_n(s,x-y)-Z_n(s,
\bar {x}-y)\bigr  |^p \bigr) \Bigr]
\\
&&\phantom{\quad \le C \biggl\{} {} +|x-\bar{x}|^{\alpha_1 {\trup{p}{2}}}\int
_0^t \mathrm{d}s \Bigl[\sup
_{y\in\R^3} \E \bigl( \bigl |Z_n(s,x-y)-Z_n(s,\bar
{x}-y)\bigr |^p \bigr) \Bigr]^{\trup{1}{2}} \biggr\},
\nonumber
\end{eqnarray}
where $\alpha_1 \in\,]0,(2-\beta)\wedge1[ $ and $\alpha_2 \in\,]0,(2-\beta)[ $.
\end{lemma}
\begin{pf}
First, we notice that $I_n(t,x,\bar x)$ is the second order moment of
the stochastic integral
\[
\int_0^t \int_{\R^3}
Z_n(s,y)\bigl[G(t-s,x-y)-G(t-s,\bar x-y)\bigr] M(\mathrm{d}s,
\mathrm{d}y).
\]

We write $I_n(t,x,\bar x)$ using \eqref{fundamental}. This yields
\begin{eqnarray*}
I_n(t,x,\bar x)&=& C\int_0^t
\mathrm{d}s \int_{\R^3} \int_{\R^3}
Z_n(s,u)Z_n(s,v) \bigl[G(t-s,x-\mathrm{d}u)- G(t-s,\bar
x-\mathrm{d}u)\bigr]
\\
&&{} \times\bigl[G(t-s,x-\mathrm{d}v)-G(t-s,\bar x-\mathrm{d}v)
\bigr]|u-v|^{-\beta}.
\end{eqnarray*}
Then, as in \cite{dss}, pages 19--20, we see that, by decomposing this
expression into the sum of four integrals, by applying a change of
variables and rearranging terms, we have
\[
I_n(t,x,\bar x)= C\sum_{i=1}^4
J_i^t(x,\bar{x}),
\]
where, for $i=1,\ldots,4$,
\[
J_i^t(x,\bar{x})=\int_0^t
\mathrm{d}s \int_{\R^3}\int_{\R^3} G(s,
\mathrm{d}u)G(s,\mathrm{d}v) h_i(x,\bar{x};t,s,u,v)
\]
with
\begin{eqnarray*}
h_1(x,\bar{x} ;t,s,u,v) &=&f(\bar{x}-x+v-u) \bigl[Z_n(t-s,x-u)-Z_n(t-s,
\bar {x}-u)\bigr]
\\
&&{} \times\bigl[Z_n(t-s,x-v)-Z_n(t-s,\bar{x}-v)\bigr],
\\
h_2(x,\bar{x} ;t,s,u,v) &=&Df(v-u, x-\bar{x}) Z_n(t-s,x-u)
\\
&&{} \times\bigl[Z_n(t-s,x-v)-Z_n(t-s,\bar{x}-v)\bigr],
\\
h_3(x,\bar{x} ;t,s,u,v) &=&Df(v-u,\bar{x}-x) Z_n(t-s,
\bar{x}-v)
\\
&&{} \times\bigl[Z_n(t-s,x-u)-Z_n(t-s,\bar{x}-u)\bigr],
\\
h_4(x,\bar{x} ;t,s,u,v) &=&-D^2f(v-u,x-\bar{x})
Z_n(t-s,x-u)Z_n(t-s,x-v).
\end{eqnarray*}
Fix $p\in[2,\infty[$. It holds that
%
%e2.18 #&#
\begin{equation}
\label{s3.23} \E\bigl(\bigl |I_n(t,x,\bar{x})\bigr |^{{\trup{p}{2}}}\bigr) \le C
\sum_{i=1}^4 \E \bigl(\bigl |J_i^t(x,
\bar{x}) \bigr |^{{\trup{p}{2}}}\bigr).
\end{equation}
%
%%%%%%
%%%%%%
The next purpose is to obtain estimates for each term on the right
hand-side of \eqref{s3.23}. Let
\[
\mu_1(x,\bar{x})= \sup_{s\in[0,T]}\int
_{\R^3}\int_{\R^3} G(s,\mathrm{d}u)G(s,
\mathrm{d}v)f(\bar{x}-x+v-u).
\]
We recall that the inverse Fourier transform of $f(x)=|x|^\beta$ is
given by $\mu(\mathrm{d}\xi)=|\xi|^{-(3-\beta)}\, \mathrm{d}\xi$, and that $\tf
G(t,\ast)(\xi)=\frac{\sin(2\uppi t|\xi|)}{2\uppi|\xi|}$. Hence,
\begin{eqnarray*}
\int_{\R^3}\int_{\R^3} G(s,
\mathrm{d}u)G(s,\mathrm{d}v)f(\bar{x}-x+v-u)&\le&\int_{\R^3}
\bigl \vert\tf G(s,\ast) (\xi)\bigr \vert^2 \mu(\mathrm{d}\xi)
\\
&=&\int_{\R^3} \frac{\sin^2(2\uppi s|\xi|)}{4\uppi^2|\xi|^{5-\beta
}}\,\mathrm{d}\xi.
\end{eqnarray*}
Consequently, for any $\beta\in\,]0,2[$, $\sup_{x,\bar{x}}\mu
_1(x,\bar{x})<\infty$
(see \cite{dss} for a similar result).

Hence using firstly H\"older's inequality and then Cauchy--Schwarz's
inequality, we see that
%
%e2.19 #&#
\begin{eqnarray}
\label{s3.24} &&\E \bigl(\bigl |J_1^t(x,
\bar{x})\bigr |^{{\trup{p}{2}}} \bigr)
\nonumber
\\
&&\quad \le \biggl(\int_0^t \mathrm{d}s\int
_{\R^3}\int_{\R^3} G(s,\mathrm{d}u)G(s,
\mathrm{d}v) f(\bar{x}-x+v-u) \biggr)^{(\trup{p}{2})-1}
\nonumber
\\
&&\qquad{} \times\int_0^t \mathrm{d}s\int
_{\R^3}\int_{\R^3} G(s,\mathrm{d}u)G(s,
\mathrm{d}v)f(\bar {x}-x+v-u)
\nonumber
\\
&&\qquad{} \times \E \bigl( \bigl | \bigl[Z_n(t-s,x-u)-Z_n(t-s,
\bar{x}-u)\bigr]
\nonumber
\\[-8pt]
\\[-8pt]
&&\qquad{}\times\bigl[Z_n(t-s,x-v)-Z_n(t-s,\bar{x}-v)
\bigr] \bigr |^{{\trup{p}{2}}} \bigr)\nonumber
\\
&&\quad \le C \sup_{x,\bar x} \mu_1(x,\bar
x)^{\trup{p}{2}}
\nonumber
\\
&&\qquad{} \times\int_0^t \mathrm{d}s \sup
_{y\in\R^3} \E \bigl( \bigl | Z_n(t-s,x-y)-Z_n(t-s,
\bar{x}-y) \bigr |^p \bigr)
\nonumber
\\
&&\quad \le C \int_0^t \mathrm{d}s \sup
_{y\in\R^3} \E \bigl( \bigl | Z_n(s,x-y)-Z_n(s,
\bar{x}-y) \bigr |^p \bigr).
\nonumber
\end{eqnarray}
Set
%
%e2.20 #&#
\begin{equation}
\label{s3.241} \mu_2(x,\bar{x})= \sup_{s\in[0,T]}\int
_{\R^3}\int_{\R^3} G(s,\mathrm{d}u)G(s,
\mathrm{d}v)\bigl |Df(v-u,x-\bar{x})\bigr |.
\end{equation}
The following property holds: there exists a positive finite constant
$C$ such that
\[
\mu_2(x,\bar{x})\le C|x-\bar{x}|^{\alpha_1}, \qquad
\alpha_1 \in\,\bigl ]0,(2-\beta)\wedge1\bigr [.
\]
Indeed, this follows from a slight modification of the proof of
\cite{dss}, Lemma~6.1.

Using H\"older's and Cauchy--Schwarz's inequalities, along with \eqref
{s3.21}, we have
%
%e2.21 #&#
\begin{eqnarray}
\label{s3.25} \E\bigl(\bigl |J_2^t(x,\bar{x})\bigr |^{{\trup{p}{2}}}
\bigr) &\le& C \biggl(\int_0^t \mathrm{d}s\int
_{\R^3}\int_{\R^3} G(s,\mathrm{d}u)G(s,
\mathrm{d}v)\bigl |Df(v-u,x-\bar{x})\bigr | \biggr)^{(\trup{p}{2})-1}
\nonumber
\\
&&{} \times\int_0^t \mathrm{d}s\int
_{\R^3}\int_{\R^3} G(s,\mathrm{d}u)G(s,
\mathrm{d}v)\bigl |Df(v-u,x-\bar{x})\bigr |
\nonumber
\\[-8pt]
\\[-8pt]
&&{} \times\E \bigl( \bigl | Z_n(t-s,x-u)\bigl[Z_n(t-s,x-v)-Z_n(t-s,
\bar {x}-v)\bigr] \bigr |^{{\trup{p}{2}}} \bigr)\qquad\quad
\nonumber
\\
&\le& C |x-\bar{x}|^{{\alpha_1}{\trup{p}{2}}}\int_0^t
\mathrm{d}s \Bigl[ \sup_{y\in\R^3} \E \bigl( \bigl | Z_n(s,x-y)-Z_n(s,
\bar{x}-y) \bigr |^p \bigr) \Bigr]^{{\trup{1}{2}}},
\nonumber
\end{eqnarray}
with $\alpha_1 \in\,]0,(2-\beta)\wedge1[$.

Similarly,
%
%e2.22 #&#
\begin{equation}
\label{s3.26} \E\bigl(\bigl |J_3^t(x,\bar{x})\bigr |^{{\trup{p}{2}}}
\bigr) \le C |x-\bar{x}|^{{\alpha
_1}{\trup{p}{2}}} \int_0^t
\mathrm{d}s \Bigl[\sup_{y\in\R^3} \E \bigl( \bigl | Z_n(s,x-y)-Z_n(s,
\bar{x}-y) \bigr |^p \bigr) \Bigr]^{{\trup{1}{2}}},\ \ \
\end{equation}
with $\alpha_1 \in\,]0,(2-\beta)\wedge1[$.

Let
\[
\mu_4(x,\bar{x})= \sup_{s\in[0,T]}\int
_{\R^3}\int_{\R^3} G(s,\mathrm{d}u)G(s,
\mathrm{d}v)\bigl |D^2f(v-u,x-\bar{x})\bigr |.
\]
Following the arguments of the proof of Lemma~6.2 in \cite{dss}, we
see that, for any $\alpha_2 \in\,]0,(2-\beta)[$,
\[
\mu_4(x,\bar{x})\le C|x-\bar{x}|^{\alpha_2}.
\]
Then, H\"older's and Cauchy--Schwarz's inequalities, along with \eqref
{s3.21}, imply
%
%e2.23 #&#
\begin{eqnarray}
\label{s3.27} \E\bigl(\bigl |J_4^t(x,\bar{x})\bigr |^{{\trup{p}{2}}}
\bigr) &\le& C \biggl(\int_0^t \mathrm{d}s\int
_{\R^3}\int_{\R^3} G(s,\mathrm{d}u)G(s,
\mathrm{d}v) \bigl |D^2f(v-u,x-\bar{x})\bigr | \biggr)^{(\trup{p}{2})-1}\quad
\nonumber
\\
&&{} \times\int_0^t \mathrm{d}s\int
_{\R^3}\int_{\R^3} G(s,\mathrm{d}u)G(s,
\mathrm{d}v)\bigl |D^2f(v-u,x-\bar{x})\bigr |
\nonumber
\\
&&{} \times\E \bigl( \bigl | Z_n(t-s,x-u)Z_n(t-s,\bar{x}-v)
\bigr |^{{\trup
{p}{2}}} \bigr)
\\
&\le& C |x-\bar{x}|^{{\alpha_2}\trup{p}{2}} \int_0^t
\mathrm{d}s \sup_{y\in
\R^3} \E \bigl( \bigl |Z_n(t-s,y)
\bigr |^p \bigr)
\nonumber
\\
&\le& C |x-\bar{x}|^{\trup{\alpha_2 p}{2} }.
\nonumber
\end{eqnarray}

From \eqref{s3.23}, \eqref{s3.24}, \eqref{s3.25}, \eqref{s3.26} and
\eqref{s3.27}, we obtain \eqref{s3.220}.
\end{pf}
%
%%%%%%%%%%end of technical lemma

For any $t\in[t_0,T]$, $x,\bar x\in K$, $p\in[1,\infty[$, we set
\begin{eqnarray*}
\varphi_{n,p}^0(t,x,\bar{x}) &=& \E \bigl( \bigl |
X_n(t,x)-X_n(t,\bar {x}) \bigr |^p
1_{L_n(t)} \bigr),
\\
\varphi_{n,p}^-(t,x,\bar{x}) &=& \E \bigl( \bigl | X_n^-(t,x)-X_n^-(t,
\bar{x}) \bigr |^p 1_{L_n(t)} \bigr),
\\
\varphi_{n,p}(t,x,\bar{x}) &=& \varphi_n^0(t,x,
\bar{x})+\varphi _n^-(t,x,\bar{x}).
\end{eqnarray*}

Proposition~\ref{pss3.1.1} is a consequence of the following assertion.
%
%pr2.7 #&#
\begin{proposition}\label{pss3.1.2}
The hypotheses are the same as in Proposition~\ref{pss3.1.1}. Fix
$t\in[t_0,T]$, $x,\bar x\in K$. Then, for any $p\in[1,\infty[$,
$\rho\in\, ]0, \frac{2-\beta}{2} [$,
%
%e2.24 #&#
\begin{equation}
\label{s3.28} \sup_{n\ge0}\varphi_{n,p}(t,x,\bar{x})
\le C|x-\bar x|^{\rho p}.
\end{equation}
\end{proposition}

The proof of this proposition relies on the next lemma and a version of
Gronwall's lemma quoted in Lemma~\ref{ls5.1}.

%%%%%%Lemma
%
%le2.8 #&#
\begin{lemma}\label{lss3.1.2}
We assume the same hypotheses as in Proposition~\ref{pss3.1.1}. For
any $n\ge1$, $t\in[t_0,T]$, $x,\bar x\in K$, $p\in[2,\infty[$,
there exists a finite constant $C$ (not depending on $n$) such that
%
%e2.25 #&#
\begin{eqnarray}
\label{s3.29} \varphi_{n,p}(t,x,\bar{x}) &\le& C
\biggl[f_n+|x-\bar x|^{\trup{\alpha
_2p}{2}} + \int_0^t
\mathrm{d}s \bigl(\varphi_{n,p}(s,x,\bar{x})\bigr)
\nonumber
\\[-8pt]
\\[-8pt]
&&\phantom{C
\biggl[}{}+   |x-\bar{x}|^{\alpha_1 \trup{p}{2}} \int_0^t
\mathrm{d}s \bigl( \bigl[\varphi_{n,p}^0(s,x,\bar{x})
\bigr]^{{\trup{1}{2}}}+ \bigl[\varphi_{n,p}^-(s,x,\bar{x})
\bigr]^{{\trup{1}{2}}} \bigr) \biggr],\quad\
\nonumber
\end{eqnarray}
where $(f_n, n\ge1)$ is a sequence of real numbers which converges to
zero as $n\to\infty$, $\alpha_1\in[0, (2-\beta\wedge1)[$, $\alpha
_2\in\,]0, 2-\beta[$.
\end{lemma}

We postpone the proof of this lemma to the end of this section.

\begin{pf*}{Proof of Proposition~\ref{pss3.1.2}}
Fix $t\in[t_0,T]$, $x,\bar x \in K$, $p\in[2,\infty[$.
From Lemma~\ref{lss3.1.2} along with Jensen's inequality, we have
\begin{eqnarray*}
\varphi_{n,p}(t,x,\bar{x})^2 &\le& C \biggl\{f_n^2+|x-
\bar {x}|^{\alpha_2 p} + \int_0^t \mathrm{d}s
\bigl(\varphi_{n,p}(s,x,\bar{x}) \bigr)^2
\\
&&\phantom{C\biggl\{} {}  + |x-\bar{x}|^{{\alpha_1}p} \int
_0^t \mathrm{d}s \bigl( \bigl[\varphi
_{n,p}^0(s,x,\bar{x}) \bigr]^{{\trup{1}{2}}}+ \bigl[
\varphi_{n,p}^-(s,x,\bar{x}) \bigr]^{{\trup{1}{2}}} \bigr)^2
\biggr\}
\\
&\le& C \biggl\{f_n^2 +|x-\bar{x}|^{\alpha_2 p} + \int
_0^t \mathrm{d}s \bigl(\varphi_{n,p}(s,x,
\bar{x}) \bigr)^2
\\
&&\phantom{C\biggl\{} {}  + |x-\bar{x}|^{{\alpha_1}p} \int
_0^t \mathrm{d}s \bigl(\varphi
_{n,p}(s,x,\bar{x}) \bigr) \biggr\}.
\end{eqnarray*}
Since the sequence $(f_n, n\ge1)$ is bounded, there exists a
constant $C_0$ satisfying
\[
\sup_n{f_n^2} \le
C_0t_0\le C_0t\le C\int
_0^t \mathrm{d}s \bigl[1+\bigl(\varphi
_{n,p}(s,x,\bar x)\bigr)^2\bigr]
\]
for any $t\in[t_0,T]$.
Thus, for some positive constant $C$,
\begin{eqnarray*}
1+\varphi_{n,p}(t,x,\bar{x})^2 &\le& C \biggl\{|x-
\bar{x}|^{\alpha_2 p} + \int_0^t \mathrm{d}s
\bigl[1+ \bigl(\varphi_{n,p}(s,x,\bar{x}) \bigr)^2 \bigr]
\\
&&\phantom{C\biggl\{} {}  + |x-\bar{x}|^{{\alpha_1}p} \int
_0^t \mathrm{d}s \bigl[1+ \varphi
_{n,p}(s,x,\bar{x})^2 \bigr]^{{\trup{1}{2}}} \biggr\}.
\end{eqnarray*}

%%%%%%%
%%%%%%%Application of Gronwall's lemma

We apply Lemma~\ref{ls5.1} in the following particular situation:
$u(t)= \varphi_{n,p}(t,x,\bar{x})^2 +1$,
$a= C|x-\bar{x}|^{\alpha_2 p}$, $b(s)\equiv C$, $k(s)\equiv C |x-\bar
{x}|^{\alpha_1 p}$,
$\bar p=\bar q=\frac{1}{2}$, $\alpha=0$, $\beta=T$.
This yields
\[
\varphi_{n,p}(t,x,\bar{x})^2 +1 \le C \bigl[|x-
\bar{x}|^{2\alpha_1
p }+|x-\bar{x}|^{\alpha_2 p} \bigr],
\]
which trivially implies\vspace*{-2pt}
%
%e2.26 #&#
\begin{equation}
\label{s3.300} \varphi_{n,p}(t,x,\bar{x}) \le C \bigl[|x-
\bar{x}|^{\alpha_1 p
}+|x-\bar{x}|^{\trup{\alpha_2 p}{2}} \bigr].
\end{equation}

We recall that $\alpha_1 \in\,]0,(2-\beta)\wedge1[$ and $\alpha_2
\in\,]0,(2-\beta)[$. Therefore, \eqref{s3.300} implies \eqref{s3.28}.
This ends the proof of Proposition~\ref{pss3.1.2}.
\end{pf*}

%%%%%%
%%%%%% Proof of auxiliary lemmas (parts 1+2)
%%%%%%
%%%%%%
%%%%%%

\begin{pf*}{Proof of Lemma~\ref{lss3.1.2}}
Fix $p\in[2,\infty[$. From \eqref{s3.6}, we have the following:\vspace*{-2pt}
\[
\varphi_{n,p}^0(t,x,\bar x):=\E \bigl(\bigl |X_n(t,x)-X_n(t,
\bar{x}) \bigr |^p 1_{L_n(t)} \bigr) \le C \sum
_{i=1}^6 R_n^i(t,x,\bar
x),
\]
with\vspace*{-2pt}
\begin{eqnarray*}
R_n^1(t,x,\bar x) &=&\E \biggl(\biggl\llvert \int
_0^t \int_{\R^3}
\bigl[G(t-s,x-y)-G(t-s,\bar {x}-y)\bigr] A\bigl(X_n(s,y)\bigr) M(
\mathrm{d}s,\mathrm{d}y) \biggr\rrvert ^p1_{L_n(t)} \biggr),
\\
R_n^2(t,x,\bar x) &=&\E {\bigl(} \bigl |\bigl\langle\bigl[G(t-
\cdot,x-\ast)-G(t-\cdot,\bar{x}-\ast )\bigr] B\bigl(X_n^-(\cdot,\ast)
\bigr), w^n\bigr\rangle_{\mathcal{H}_t} \bigr |^p1_{L_n(t)}
{\bigr)},
\\
R_n^3(t,x,\bar x) &=&\E \bigl(\bigl\llvert \bigl\langle
\bigl[G(t-\cdot,x-\ast)-G(t-\cdot,\bar {x}-\ast)\bigr] \bigl[B(X_n)-B
\bigl(X_n^-\bigr)\bigr] (\cdot,\ast),w^n \bigr
\rangle_{\mathcal{H}_t}\bigr\rrvert ^p1_{L_n(t)} \bigr),
\\
R_n^4(t,x,\bar x) &=&\E \bigl(\bigl\llvert \bigl\langle
\bigl[G(t-\cdot,x-\ast)-G(t-\cdot,\bar {x}-\ast)\bigr] D\bigl(X_n(
\cdot,\ast)\bigr),h \bigr\rangle_{\mathcal{H}_t}\bigr\rrvert ^p1_{L_n(t)}
\bigr),
\\
R_n^5(t,x,\bar x) &=&\E \biggl(\biggl\llvert \int
_0^t \int_{\R^3}
\bigl[G(t-s,x-\mathrm{d}y)-G(t-s,\bar{x}-\mathrm{d}y)\bigr] b\bigl(X_n(s,y)
\bigr) \,\mathrm{d}s \biggr\rrvert ^p 1_{L_n(t)} \biggr).
\end{eqnarray*}

Using Burkholder's inequality and then Plancherel's identity, we have\vspace*{-2pt}
%
%e2.27 #&#
\begin{eqnarray}
\label{s3.32} R_n^1(t,x,\bar{x})&=&\E \biggl(\biggl
\llvert \int_0^t \int_{\R^3}
\bigl[G(t-s,x-y)-G(t-s,\bar{x}-y) \bigr] A\bigl(X_n(s,y)\bigr)M(
\mathrm{d}s,\mathrm{d}y) \biggr\rrvert ^p1_{L_n(t)} \biggr)
\nonumber
\\
&=&\E \biggl(\biggl\llvert \sum_{j\in\mathbb{N}}\int
_0^t \bigl\langle \bigl[G(t-s,x-\ast)\nonumber
\\[-1pt]
&&\phantom{\E \biggl(\biggl\llvert\sum_{j\in\mathbb{N}}\int
_0^t \bigl\langle \bigl[}{}-G(t-s,
\bar{x}-\ast) \bigr] A\bigl(X_n(s,\ast)\bigr), e_k(\ast)
\bigr\rangle_{\mathcal{H}}\, \mathrm{d}W_j(s) \biggr\rrvert
^p1_{L_n(t)} \biggr)
\nonumber
\\[-8pt]
\\[-8pt]
&\le& C\E \biggl( \biggl[\int_0^t \mathrm{d}s
\sum_{j\in\mathbb{N}} \bigl\llvert \bigl\langle \bigl[G(t-s,x-
\ast)\nonumber
\\[-1pt]
&&\phantom{C\E \biggl( \biggl[\int_0^t \mathrm{d}s
\sum_{j\in\mathbb{N}} \bigl\llvert \bigl\langle \bigl[}{}-G(t-s,\bar{x}-\ast) \bigr] A\bigl(X_n(s,\ast )\bigr),
e_k(\ast) \bigr\rangle_{\mathcal{H}} \bigr\rrvert
^21_{L_n(s)} \biggr] \biggr)^{{\trup{p}{2}}}
\nonumber
\\[-1pt]
&=& C\E \biggl(\biggl\llvert \int_0^t
\mathrm{d}s \bigl \| \bigl[G(t-s,x-\ast )-G(t-s,\bar{x}-\ast) \bigr] A
\bigl(X_n(s,\ast)\bigr)\bigr  \|_{\mathcal{H}}^2 \biggr\rrvert
1_{L_n(s)} \biggr)^{{\trup{p}{2}}}.
\nonumber
\end{eqnarray}
The process $\{Z_n(t,x):=A(X_n(t,x))1_{L_n(t)}, (t,x)\in[0,T]\times
\IR^3\}$ satisfies the assumption \eqref{s3.21}. Indeed, this is a
consequence of the linear growth of $A$ and \eqref{s4.18}. Then,\vadjust{\goodbreak} by
applying Lemma~\ref{lss3.1.1} and using the Lipschitz continuity of
$A$, we obtain
%
%e2.28 #&#
\begin{eqnarray}
\label{s3.33} R_n^1(t,x,\bar{x}) &\le& C \biggl\{ |x-
\bar{x}|^{\alpha_2 \trup{p}{2}} \nonumber
\\
&&\phantom{C\biggl\{} {} +\int_0^t \mathrm{d}s
\Bigl[ \sup_{y\in\R^3} \E \bigl( \bigl | X_n(s,x-y)-X_n(s,
\bar{x}-y) \bigr |^p 1_{L_n(s)} \bigr) \Bigr]
\\
&&\phantom{C\biggl\{} {}  + |x-\bar{x}|^{{\alpha_1}{\trup{p}{2}}}\int
_0^t \mathrm{d}s \Bigl[ \sup
_{y\in\R^3} \E \bigl( \bigl | X_n(s,x-y)-X_n(s,
\bar{x}-y) \bigr |^p 1_{L_n(s)} \bigr) \Bigr]^{{\trup
{1}{2}}} \biggr
\},
\nonumber
\end{eqnarray}
with $\alpha_1 \in\,]0,(2-\beta)\wedge1[$ and $\alpha_2 \in\,]0,(2-\beta)[$.

For a given function $\rho\dvtx [0,T]\times\IR^3\to\IR$ and $t\in
[0,T]$, let $\tau_n$ be the operator defined by
%
%e2.29 #&#
\begin{equation}
\label{s3.34} \tau_n(\rho)=\rho \bigl(\bigl(s+2^{-n}
\bigr)\wedge t, x \bigr).
\end{equation}
Let $\mathcal{E}_n$ be the closed subspace of $\mathcal{H}_T$
generated by the orthonormal system of functions
\[
2^nT^{-1}1_{\Delta_i}(\cdot)\otimes
e_j(\ast),\qquad i=0,\ldots,2^n-1,\ j=1,\ldots,n,
\]
and denote by $\pi_n$ the orthogonal projection on $\mathcal{E}_n$.
Notice that $\pi_n\circ\tau_n$ is
a bounded operator on $\mathcal{H}_T$, uniformly in $n$.

Since $X_n^-(s,\ast)$ is $\mathcal{F}_{s_n}$-measurable, by using the
definition of $w^n$ we easily see that
\begin{eqnarray*}
R_n^2(t,x,\bar{x}) &=&\E \biggl(\biggl\llvert \int
_0^t \int_{\R^3}(
\pi_n\circ\tau_n) \bigl(\bigl[G(t-\cdot,x-\ast)-G(t-
\cdot,\bar{x}-\ast)\bigr]
\\
&&\phantom{\E\bigl(|} {}  \times B\bigl(X_n^-(
\cdot,\ast)\bigr) \bigr) (s,y) M(\mathrm{d}s,\mathrm{d}y) \biggr\rrvert
^p1_{L_n(s)} \biggr).
\end{eqnarray*}
By Burkholder's inequality and the properties of the operator $\pi_n\circ\tau_n$, this last expression is bounded up to a constant by
\[
\E \biggl(\int_0^t \mathrm{d}s \bigl\llVert
\bigl(\bigl[G(t-s,x-\ast)-G(t-s,\bar {x}-\ast)\bigr] B\bigl(X_n^-(s,
\ast)\bigr) \bigr) \bigr\rrVert _{\mathcal{H}}^2 1_{L_n(s)}
\biggr)^{{\trup{p}{2}}}.
\]
The properties of the function $B$ along with \eqref{s4.18} imply that
the process $\{Z_n(t,x):=B(X_n^-(t,x))1_{L_n(t)}, (t,x)\in[0,T]\times
\IR^3\}$ satisfies the hypotheses of Lemma~\ref{lss3.1.1}. This yields
%
%e2.30 #&#
\begin{eqnarray}
\label{s3.35} R_n^2(t,x,\bar{x}) &\le& C \biggl\{ |x-
\bar{x}|^{\alpha_2 {\trup{p}{2}}} \nonumber
\\
&&\phantom{C \biggl\{}{}+ \int_0^t \mathrm{d}s
\Bigl[ \sup_{y\in\R^3} \E \bigl( \bigl | X_n^-(s,x-y)-X_n^-(s,
\bar{x}-y) \bigr |^p 1_{L_n(s)} \bigr) \Bigr]
\\
&&\phantom{C \biggl\{}{}  + |x-\bar{x}|^{\trup{\alpha_1 p}{2} }\int_0^t
\mathrm{d}s \Bigl[ \sup_{y\in\R^3} \E \bigl( \bigl | X_n^-(s,x-y)-X_n^-(s,
\bar{x}-y)\bigr  |^p 1_{L_n(s)} \bigr) \Bigr]^{{\trup{1}{2}}} \biggr
\},
\nonumber
\end{eqnarray}
where as before, $\alpha_1 \in\,]0,(2-\beta)\wedge1[ $ and $ \alpha
_2 \in\,]0,(2-\beta)[ $.\vadjust{\goodbreak}
%%%%%$R_n^4$

Cauchy--Schwarz's inequality along with \eqref{s3.101} yield
\begin{eqnarray*}
R_n^3(t,x,\bar{x}) &\le& C n^{\trup{3p}{2}}2^{n{\trup{p}{2}}}
\\
&&{} \times\E \biggl( \int_0^t \mathrm{d}s \bigl\|
\bigl[G(t-s,x-\ast)-G(t-s,\bar {x}-\ast)\bigr]
\\
&&\phantom{\times\E \bigl( \int_0^t \mathrm{d}s \bigl\|} {}   \times\bigl[B(X_n)-B
\bigl(X_n^-\bigr)\bigr](s,\ast)1_{L_n(s)} \bigr\|_{\mathcal
{H}}^2
\biggr)^{{\trup{p}{2}}}.
\end{eqnarray*}
Notice that an upper bound for the second factor on the right-hand side
of the preceding inequality could be obtained using
Lemma~\ref{lss3.1.1} with $Z_n(t,x):=[B(X_n(t,x))-B(X_n^-(t,x))]
1_{L_n}(t)$. However, this would not be a good strategy to compensate the
first factor (which explodes when $n\to\infty$). Instead, we will try
to quantify the discrepancy between $B(X_n(t,x))$ and $B(X_n^-(t,x))$.
This can be achieved by transferring again the increments of the Green
function to increments of the process
%
%e2.31 #&#
\begin{equation}
\label{s3.351} \hat{B}\bigl(X_n(t,x)\bigr)=\bigl[B
\bigl(X_n(t,x)\bigr)-B\bigl(X_n^-(t,x)\bigr)\bigr],
\end{equation}
in the same manner as we did in the proof of Lemma~\ref{lss3.1.1} (see
\cite{dss}, pages 19--20).

Indeed, similarly as in \eqref{s3.23}, we obtain
%
%e2.32 #&#
\begin{equation}
\label{s3.36} R_n^3(t,x,\bar{x}) \le C
n^{\trup{3p}{2}}2^{n{\trup{p}{2}}} \sum_{i=1}^4
\E \bigl(\bigl |K_i^t(x,\bar{x}) \bigr |^{{\trup{p}{2}}}1_{L_n(t)}
\bigr),
\end{equation}
where for any $i=1,\ldots,4$, $K_i^t(x,\bar x)$ is given by
$J_i^t(x,\bar x)$ of Lemma~\ref{lss3.1.1} with $Z_n$ replaced by $\hat
{B}(X_n)$.
%%%%%

Using Remark~\ref{rs3.1}, we have
%
%e2.33 #&#
\begin{equation}
\label{s3.361} \E \bigl(\bigl |X_n(t,x-y)-X_n(t,\bar{x}-y)
\bigr |^p 1_{L_n(t)} \bigr) = \E \bigl( \bigl | X_n(t,x)-X_n(t,
\bar{x}) \bigr |^p 1_{L_n(t)} \bigr).
\end{equation}

With this property and the definition of $\hat{B}(X_n)$ given in
\eqref{s3.351}, we easily get
%
%e2.34 #&#
\begin{eqnarray}
\label{s3.37} && \E \bigl( \bigl | \hat{B}\bigl(X_n(s,x-y)\bigr)-\hat{B}
\bigl(X_n(s,\bar{x}-y)\bigr) \bigr |^p 1_{L_n(s)} \bigr)
\nonumber
\\
&&\quad \le C \bigl[\E \bigl( \bigl | X_n(s,x-y)-X_n^-(s,x-y)
\bigr |^p 1_{L_n(s)} \bigr)
\nonumber
\\[-8pt]
\\[-8pt]
&&\phantom{\quad \le C\bigl[} {}  +\E \bigl( \bigl | X_n(s,
\bar{x}-y)-X_n^-(s,\bar{x}-y) \bigr |^p 1_{L_n(s)} \bigr)
\bigr]
\nonumber
\\
&&\quad \le C n^{\trup{3p}{2}}2^{-np{\trup{(3-\beta)}{2}}},
\nonumber
\end{eqnarray}
uniformly in $(s,x,y)\in[0,T]\times\IR^3\times\IR^3$, where the
last bound is obtained by using \eqref{s4.19}.
This estimate will be applied to the study of the right-hand side of
\eqref{s3.36}.

%$K_1$
For $i=1$, \eqref{s3.24} with $Z_n(s,y):=\hat B(X_n(s,y))1_{L_n(s)}$,
along with \eqref{s3.37} yields
%
%e2.35 #&#
\begin{equation}
\label{s3.38} \E\bigl(\bigl |K_1^t(x,\bar{x})\bigr |^{{\trup{p}{2}}}1_{L_n(t)}
\bigr) \le C n^{\trup
{3p}{2}}2^{-np{\trup{(3-\beta)}{2}}}.
\end{equation}

Let $\mu_2(x,\bar{x})$ be as in \eqref{s3.241}. Since $x,\bar{x}
\in K$, and $K$ is bounded,
\[
\sup_{x,\bar{x}\in K}\mu_2(x,\bar{x})\le C,
\]
for some finite constant $C>0$. Hence, \eqref{s3.25}, \eqref{s3.26}
(with the same choice of $Z_n$ as before) together with \eqref{s3.37} gives
%
%e2.36 #&#
\begin{equation}
\label{s3.39} \E\bigl(\bigl |K_2^t(x,\bar{x})\bigr |^{{\trup{p}{2}}}1_{L_n(t)}
\bigr) + \E \bigl(\bigl |K_3^t(x,\bar{x})\bigr |^{{\trup{p}{2}}}1_{L_n(t)}
\bigr) \le C n^{\trup{3p}{2}}2^{-np {\trup{(3-\beta)}{2}} }.
\end{equation}

Proceeding as in \eqref{s3.27}, but replacing $Z_n(s,y)$ by $\hat
B(X_n(s,y))1_{L_n(s)}$, we obtain
\[
\E\bigl(\bigl |K_4^t(x,\bar{x})\bigr |^{{\trup{p}{2}}}
1_{L_n(t)}\bigr) \le C |x-\bar{x}|^{{\alpha_2}\trup{p}{2}} \int
_0^t \mathrm{d}s\sup_{y\in\R^3}
\E \bigl( \bigl | \hat{B}\bigl(X_n(s,y)\bigr) \bigr |^p1_{L_n(s)}
\bigr).
\]
By the definition of $\hat B(X_n)$, and applying \eqref{s4.19}, we have
\[
\sup_{(s,y)\in[0,T]\times\R^3} \E \bigl( \bigl | \hat {B}\bigl(X_n(s,y)
\bigr) \bigr |^p1_{L_n(s)} \bigr) \le C n^{\trup{3p}{2}}
2^{-np{\trup{(3-\beta)}{2}}}.
\]
Thus,
%
%e2.37 #&#
\begin{equation}
\label{s3.40} \E\bigl(\bigl |K_4^t(x,\bar{x})\bigr |^{{\trup{p}{2}}}
1_{L_n(t)}\bigr)\le C n^{\trup
{3p}{2}}2^{-np{\trup{(3-\beta)}{2}}}.
\end{equation}
Putting together \eqref{s3.36} and \eqref{s3.38}--\eqref{s3.40} yields
%
%e2.38 #&#
\begin{equation}
\label{s3.41} R_n^3(t,x,\bar{x}) \le C f_n,
\end{equation}
where $f_n=n^{3p} 2^{-np [{\trup{(3-\beta)}{2}}-{\trup
{1}{2}} ]}$. Since $\beta\in\,]0,2[$, $\lim_{n\to\infty}f_n=0$.

The last part of the proof consists of getting estimates for the term
$R_n^4(t,x,\bar{x})$. This is done using first
Cauchy--Schwarz's inequality and then, applying Lemma~\ref{lss3.1.1}
with $Z_n$ replaced by $D(X_n) 1_{L_n}$.
The Lipschitz continuity of $D$ along with the estimate \eqref{s4.18}
ensure that assumption \eqref{s3.21} is satisfied. We obtain
%
%e2.39 #&#
\begin{eqnarray}
\label{s3.42} R_n^4(t,x,\bar{x}) &\le&\| h
\|_{\mathcal{H}_t}^p \E \bigl( \bigl |\bigl \|\bigl[G(t-\cdot,x-\ast)-G(t-
\cdot,\bar{x}-\ast)\bigr] D\bigl(X_n(\cdot,\ast)\bigr)
1_{L_n(t)} \bigr \|_{\mathcal{H}_t}^2\bigr  |^{{\trup{p}{2}}} \bigr)
\nonumber
\\
&\le& C \biggl\{ |x-\bar{x}|^{\alpha_2 \trup{p}{2} } + \int_0^t
\mathrm{d}s \sup_{y\in\R^3} \E \bigl( \bigl | X_n(s,x-y)-X_n(s,
\bar {x}-y)\bigr  |^p 1_{L_n(s)} \bigr)
\\
&&\phantom{C\biggl\{} {}  +|x-\bar{x}|^{{\alpha_1}\trup{p}{2}}\int
_0^t \mathrm{d}s \Bigl[\sup
_{y\in\R^3} \E \bigl( \bigl | X_n(s,x-y)-X_n(s,
\bar{x}-y) \bigr |^p 1_{L_n(s)} \bigr) \Bigr]^{{\trup
{1}{2}}} \biggr
\},\quad
\nonumber
\end{eqnarray}
where $ \alpha_1 \in\,]0,(2-\beta)\wedge1[ $ and $ \alpha_2 \in\,]0,(2-\beta)[$.

After having applied the change of variable $u\mapsto x-\bar{x}+y$, we have
\[
R_n^5(t,x,\bar{x}) =\E \biggl(\biggl\llvert \int
_0^t \int_{\R^3}G(t-s,x-
\mathrm{d}y)\bigl[b\bigl(X_n(s,y)\bigr) -b\bigl(X_n(s,y-x+
\bar{x})\bigr) \bigr]\,\mathrm{d}s \biggr\rrvert ^p1_{L_n(t)}
\biggr).
\]
Applying H\"older's inequality, we obtain
%
%e2.40 #&#
\begin{eqnarray}
\label{s3.31}
&&R_n^5(t,x,\bar{x}) \nonumber
\\
&&\quad\le  \biggl(\int
_0^t\int_{\R^3} G(t-s,x-
\mathrm{d}y)\,\mathrm{d}s \biggr)^{p-1}
\nonumber
\\[-8pt]
\\[-8pt]
&&\qquad  {} \times\int_0^t \int_{\R^3}
G(t-s,x-\mathrm{d}y)\E \bigl( \bigl |b\bigl(X_n(s,y)\bigr)-b
\bigl(X_n(s,y-x+\bar{x})\bigr) \bigr |^p 1_{L_n(s)}
\bigr)\,\mathrm{d}s\nonumber
\\
&&\quad \le  C\int_0^t \mathrm{d}s \sup
_{y\in\R^3}\E \bigl( \bigl |X_n(s,x-y)-X_n(s,
\bar{x}-y) \bigr |^p1_{L_n(s)} \bigr).
\nonumber
\end{eqnarray}

Bringing together the inequalities \eqref{s3.33}, \eqref{s3.35},
\eqref{s3.41}, \eqref{s3.42} and \eqref{s3.31}, yields
\begin{eqnarray*}
&& \E \bigl(\bigl |X_n(t,x)-X_n(t,\bar{x}) \bigr |^p
1_{L_n(t)} \bigr)
\\
&& \quad \le C \biggl\{ f_n+ |x-\bar{x}|^{\alpha_2 \trup{p}{2} } + \int
_0^t \mathrm{d}s \Bigl[ \sup
_{y\in\R^3}\E \bigl( \bigl |X_n(s,x-y)-X_n(s,
\bar{x}-y) \bigr |^p 1_{L_n(s)} \bigr) \Bigr]
\\
&&\phantom{\quad \le C\biggl\{} {}  + |x-\bar{x}|^{{\alpha_1}\trup{p}{2}}\int
_0^t \mathrm{d}s \Bigl[\sup
_{y\in\R^3} \E \bigl( \bigl | X_n(s,x-y)-X_n(s,
\bar{x}-y)\bigr  |^p 1_{L_n(s)} \bigr) \Bigr]^{{\trup
{1}{2}}}
\\
&&\phantom{\quad \le C\biggl\{} {}  + \int_0^t
\mathrm{d}s \Bigl[ \sup_{y\in\R^3}\E \bigl( \bigl |X_n^-(s,x-y)-X_n^-(s,
\bar{x}-y) \bigr |^p 1_{L_n(s)} \bigr) \Bigr]
\\
&&\phantom{\quad \le C\biggl\{} {}  + |x-\bar{x}|^{{\alpha_1}\trup{p}{2}} \int
_0^t \mathrm{d}s \Bigl[\sup
_{y\in\R^3} \E \bigl( \bigl | X_n^-(s,x-y)-X_n^-(s,
\bar{x}-y) \bigr |^p 1_{L_n(s)} \bigr) \Bigr]^{{\trup{1}{2}}} \biggr
\}.
\end{eqnarray*}
By Remark~\ref{rs3.1}, the right-hand side of this inequality is equal
(up to a constant) to
\begin{eqnarray*}
&& f_n + |x-\bar{x}|^{\alpha_2 \trup{p}{2} } + \int_0^t
\E \bigl( \bigl |X_n(s,x)-X_n(s,\bar{x}) \bigr |^p
1_{L_n(s)} \bigr)\,\mathrm{d}s
\nonumber
\\
&&\quad{} + \int_0^t \E \bigl(
\bigl |X_n^-(s,x)-X_n^-(s,\bar{x})\bigr  |^p
1_{L_n(s)} \bigr)\,\mathrm{d}s
\nonumber
\\
&&\quad{} + |x-\bar{x}|^{{\alpha_1}\trup{p}{2}} \int_0^t
\bigl[ \E \bigl( \bigl | X_n(s,x)-X_n(s,\bar{x})
\bigr |^p 1_{L_n(s)} \bigr) \bigr]^{{\trup{1}{2}}} \,\mathrm{d}s
\\
&&\quad{} + |x-\bar{x}|^{{\alpha_1}\trup{p}{2}} \int_0^t
\bigl[\E \bigl( \bigl | X_n^-(s,x)-X_n^-(s,\bar{x})
\bigr |^p 1_{L_n(s)} \bigr) \bigr]^{{\trup{1}{2}}}\, \mathrm{d}s.
\end{eqnarray*}
With this, we see that $\varphi_{n,p}^0(t,x,\bar x)$ is bounded by the
right-hand side of \eqref{s3.29}.

Finally, we prove that the same bound holds for $\varphi
_{n,p}^-(t,x,\bar x)$ too. Indeed,
For every $i=1,\ldots,5$, we consider the terms $R_n^i(t,x,\bar x)$
defined in the first part of the proof,
and we replace the domain of integration of the time variable $s$
($[0,t]$) by $[0,t_n]$. We denote the corresponding new expressions by
$S_n^i(t,x,\bar x)$.
From \eqref{s3.8.2}, we obtain the following
\[
\varphi_{n,p}^-(t,x,\bar x) \le C \sum_{i=1}^5
S_n^i(t,x,\bar x).
\]
Since $t_n\le t$, it can be checked that, similarly as for
$R_n^i(t,x,\bar x)$,
$S_n^i(t,x,\bar x)$, $i=1,\ldots,5$, are bounded by \eqref{s3.33},
\eqref{s3.35}, \eqref{s3.41}, \eqref{s3.42}, \eqref{s3.31},
respectively. This ends the
proof of the lemma.
\end{pf*}

%%%%%%%%%%%% Increments in time
%%%%%%%%%%
%%%%%%%%%%%
%s2.2 #&#
\subsection{Increments in time}
\label{ss3.2}

Throughout this section, we fix $t_0\in\,]0,T]$, and a compact set
$K\subset\IR^3$. We shall prove the following proposition.
%
%pr2.9 #&#
\begin{proposition}
\label{pss3.2.1}
Assume that Hypothesis~\textup{\ref{HypB}} holds. Fix $t, \bar t\in[t_0,T]$.
Then for any $p\in[1,\infty)$ and $\rho\in\, ]0,\frac{2-\beta
}{2}[$, there exists a finite constant $C$ such that
%
%e2.41 #&#
\begin{equation}
\label{s3.43} \sup_{n\ge1} \sup_{x\in K}\bigl \Vert
\bigl(X_n(t,x) - X_n(\bar t,x)\bigr) 1_{L_n(\bar t)}
\bigr \Vert_p \le C|t-\bar t|^\rho.
\end{equation}
\end{proposition}

The next lemma is meant to play a similar r\^ole than Lemma~\ref
{lss3.1.1} but in this case, for integrals containing increments in
time of the Green function $G(t)$.
%%%%%%Lemma auxiliar

%le2.10 #&#
\begin{lemma}
\label{lss3.2.1}
Consider a sequence of stochastic processes $\{D_n(t,x), (t,x)\in
[0,T]\times\IR^3\}$, $n\ge1$, satisfying the following conditions:

For any $p\in[2,\infty[$,
%
%e2.42 #&#
\begin{equation}
\label{s3.44} \sup_n\sup_{(t,x)\in[t_0,T]\times\R^3} \E \bigl(
\bigl |D_n(t,x)\bigr  |^p \bigr)\le C.
\end{equation}
There exists $\rho_1>0$ and for any $x,y\in K$,
%
%e2.43 #&#
\begin{equation}
\label{s3.45} \sup_n \sup_{t\in[t_0,T]} \E
\bigl( \bigl |D_n(t,x)-D_n(t,y) \bigr |^p \bigr)\le
C|x-y|^{\rho_1 p},
\end{equation}
where $C$ is a finite constant and $\rho_1>0$.

For $0\le t_0\le t\le\bar{t} \le T$ and $x\in K$, set
\[
J_n(t,\bar t,x)=\int_0^t
\mathrm{d}s \bigl \Vert D_n(x,\ast) \bigl[G(\bar t-s,x-\ast )-G(t-s,x-\ast)
\bigr]\bigr \Vert^2_{\mathcal{H}}.
\]
Then, for any $p\in[2,\infty[$ there exists a finite constant $C>0$
such that
%
%e2.44 #&#
\begin{equation}
\label{s3.47} \E \bigl(J_n(t,\bar t,x)^{\trup{p}{2}} \bigr) \le C
\bigl( |\bar{t}-t|^{\rho_1 p} + |\bar{t}-t|^{(\rho_1+\alpha
_1) {\trup{p}{2}}} + |
\bar{t}-t|^{\alpha_2 p /2} \bigr),
\end{equation}
with $\alpha_1\in\,]0,1\wedge(2-\beta)[$ and $\alpha_2\in\, ]0,
(2-\beta) [$.
\end{lemma}

\begin{pf}
First of all we notice that, as a consequence of Burkholder's
inequality, the $L^p$-moment of the stochastic integral
\[
\int_0^t \int_{\R^3}
D_n(x,y)\bigl[G(\bar t-s,x-y)-G(t-s,x-y)\bigr] M(\mathrm{d}s,
\mathrm{d}y),
\]
is bounded up to a positive constant, by $\E (J_n(t,\bar
t,x)^{\trup{p}{2}} )$.

We write $J_n(t,\bar t,x)$ using \eqref{fundamental}. This gives
\begin{eqnarray*}
J_n(t,\bar t,x)&=& C\int_0^t
\mathrm{d}s \int_{\R^3}\int_{\R^3}
D_n(x,y)\bigl[G(\bar t-s,x-y)-G(t-s,x-y)\bigr]
\\
&&{} \times D_n(x,z)\bigl[G(\bar t-s,x-z)-G(t-s,x-z)\bigr]
|y-z|^{-\beta}.
\end{eqnarray*}
Then, as in \cite{dss} page 28 (see the study of the term
$T_2^n(t,\bar t,x)$ in this reference), we have
%
%e2.45 #&#
\begin{equation}
\label{s3.48} \E \bigl(J_n(t,\bar t,x)^{\trup{p}{2}} \bigr)\le C
\sum_{k=1}^4 \E \bigl(\bigl |Q^i(t,
\bar{t},x)\bigr |^{{\trup{p}{2}}} \bigr),
\end{equation}
where for $i=1,\ldots,4$,
%
%e2.46 #&#
\begin{equation}
\label{s3.481} Q^i(t,\bar{t},x):=\int_0^t
\mathrm{d}s\int_{\R^3} \int_{\R^3} G(t-s,
\mathrm{d}u)G(t-s,\mathrm{d}v) r_i(t,\bar{t},s,x,u,v)
\end{equation}
and
\begin{eqnarray*}
r_1(t,\bar{t},s,x,u,v) &:=&\frac{\bar{t}-s}{t-s} f \biggl(v
\frac{\bar
{t}-s}{t-s}-u \biggr) \biggl[D_n \biggl(s,x-\frac{\bar{t}
-s}{t-s}u
\biggr)-D_n(s,x-u) \biggr]
\\
&&{} \times \biggl[D_n \biggl(s,x-\frac{\bar{t}-s}{t-s}v
\biggr)-D_n(s,x-v) \biggr],
\\
r_2(t,\bar{t},s,x,u,v) &:=& \biggl\{ \biggl(\frac{\bar{t}-s}{t-s}
\biggr)^2f \biggl(\frac{\bar{t}-s}{t-s}(v-u) \biggr) -\frac{\bar{t}-s}{t-s}
f \biggl(v\frac{\bar{t}-s}{t-s}-u \biggr) \biggr\}
\\
&&{} \times D_n \biggl(s,x-\frac{\bar{t}-s}{t-s}u \biggr)
\biggl[D_n \biggl(s,x-\frac{\bar{t}-s}{t-s}v \biggr)-D_n(s,x-v)
\biggr],
\\
r_3(t,\bar{t},s,x,u,v) &:=& \biggl\{ \biggl(\frac{\bar{t}-s}{t-s}
\biggr)^2f \biggl(\frac{\bar{t}-s}{t-s}(v-u) \biggr) -\frac{\bar{t}-s}{t-s}
f \biggl(v-u\frac{\bar{t}-s}{t-s} \biggr) \biggr\}
\\
&&{}\times \biggl[D_n \biggl(s,x-\frac{\bar{t}-s}{t-s}u
\biggr)-D_n(s,x-u) \biggr] D_n (s,x-v ),
\\
r_4(t,\bar{t},s,x,u,v) &:=& \biggl\{ \biggl(\frac{\bar{t}-s}{t-s}
\biggr)^2 f \biggl(\frac{\bar{t}-s}{t-s}(v-u) \biggr) -\frac{\bar{t}-s}{t-s}
f \biggl(v\frac{\bar{t}-s}{t-s}-u \biggr)
\\
&&{} -  \frac{\bar{t}-s}{t-s} f \biggl(v-u\frac{\bar
{t}-s}{t-s} \biggr)+f(v-u)
\biggr\} D_n(s,x-u)D_n(s,x-v).
\end{eqnarray*}

Let
\[
\nu_1(s,t,\bar{t}):=\int_{\R^3} \int
_{\R^3} G(t-s,\mathrm{d}u)G(t-s,\mathrm{d}v)\frac{\bar{t}-s}{t-s}
f \biggl(v\frac{\bar{t}-s}{t-s}-u \biggr).
\]
Following the arguments of the proof of Lemma 6.3 in \cite{dss} (with
$G_n$ replaced by $G$), we see that
%
%e2.47 #&#
\begin{equation}
\label{s3.49} \sup_{0\le s\le t\le\bar{t}\le T} \nu_1(s,t,\bar{t}) <
\infty.
\end{equation}
Applying H\"older's and then Cauchy--Schwarz' inequalities, along with
\eqref{s3.45} yield
%
%e2.48 #&#
\begin{eqnarray}
\label{s3.491} \E \bigl( \bigl |Q^1(t,\bar{t},x)\bigr  |^{{\trup{p}{2}}} \bigr)
&\le& \Bigl(\sup_{0\le s\le t\le\bar{t}\le T}\nu_1(s,t,\bar{t})
\Bigr)^{{\trup{p}{2}}-1}
\nonumber
\\
&&{} \times\int_0^t \mathrm{d}s\int
_{R^3} \int_{R^3} G(t-s,\mathrm{d}u)G(t-s,
\mathrm{d}v)\frac
{\bar{t}-s}{t-s} f \biggl(v\frac{\bar{t}-s}{t-s}-u \biggr)
\nonumber
\\
&&{} \times \biggl[\E \biggl(\biggl\llvert D_n \biggl(s,x-
\frac{\bar
{t}-s}{t-s}u \biggr)-D_n(s,x-u) \biggr\rrvert ^p
\biggr) \biggr]^{{\trup{1}{2}}}
\nonumber
\\[-8pt]
\\[-8pt]
&&{} \times \biggl[\E \biggl(\biggl\llvert D_n \biggl(s,x-
\frac{\bar
{t}-s}{t-s}v \biggr)-D_n(s,x-v) \biggr\rrvert ^p
\biggr) \biggr]^{{\trup{1}{2}}}
\nonumber
\\
&\le& C\int_0^t \mathrm{d}s\int
_{R^3} \int_{R^3} G(t-s,\mathrm{d}u)G(t-s,
\mathrm{d}v) \frac{\bar{t}-s}{t-s} f \biggl(v\frac{\bar{t}-s}{t-s}-u \biggr)
\nonumber
\\
&&{} \times\biggl\llvert \frac{\bar{t}-t}{t-s}u\biggr\rrvert ^{\rho_1 {\trup
{p}{2}}}\biggl
\llvert \frac{\bar{t}-t}{t-s}v \biggr\rrvert ^{\rho_1 {\trup
{p}{2}}}.
\nonumber
\end{eqnarray}
The support of the measure $G(t)$ is $\{x\in\IR^3\dvtx  |x|= t\}$. Using
this property and \eqref{s3.49}, we obtain
%
%e2.49 #&#
\begin{equation}
\label{s3.50} \E \bigl( \bigl |Q^1(t,\bar{t},x)\bigr  |^{{\trup{p}{2}}} \bigr)
\le C|t-\bar t|^{\rho_1p}.
\end{equation}

Let
\begin{eqnarray*}
\nu_2(s,t,\bar{t})&:=&\int_{\R^3} \int
_{\R^3} G(t-s,\mathrm{d}u)G(t-s,\mathrm{d}v)
\\
&&{} \times\biggl\llvert \biggl(\frac{\bar{t}-s}{t-s} \biggr)^2f
\biggl(\frac
{\bar{t}-s}{t-s}(v-u) \biggr) -\frac{\bar{t}-s}{t-s} f \biggl(
\frac{\bar{t}-s}{t-s}v-u \biggr)\biggr\rrvert .
\end{eqnarray*}
A slight modification of Lemma~6.4 in \cite{dss} (where $G_n$ is
replaced by $G$), yields
%
%e2.50 #&#
\begin{equation}
\label{s3.51} \sup_{s\le t\le\bar{t}\le T} \nu_2(s,t,\bar{t}) \le
C |t-\bar {t}|^{\alpha_1},
\end{equation}
with $\alpha_1\in\,]0,(2-\beta)\wedge1[$. Then, H\"older's and
Cauchy--Schwarz's inequalities along with \eqref{s3.44}, \eqref{s3.45}
and \eqref{s3.51} imply
%
%e2.51 #&#
%e2.52 #&#
\begin{eqnarray}
\E \bigl( \bigl |Q^2(t,\bar{t},x) \bigr |^{{\trup{p}{2}}} \bigr) &\le& \Bigl(\sup
_{0\le s\le t\le\bar{t}\le T}\nu_2(s,t,\bar {t}) \Bigr)^{{\trup{p}{2}}-1}
\int_0^t \mathrm{d}s\int_{\R^3}
\int_{\R^3} G(t-s,\mathrm{d}u)G(t-s,\mathrm{d}v)
\nonumber
\\
&&{} \times\biggl\llvert \biggl(\frac{\bar{t}-s}{t-s} \biggr)^2f
\biggl(\frac
{\bar{t}-s}{t-s}(v-u) \biggr) -\frac{\bar{t}-s}{t-s} f \biggl(
\frac{\bar{t}-s}{t-s}v-u \biggr) \biggr\rrvert
\nonumber
\\
&&{} \times \biggl[ \biggl(\E\biggl\llvert D_n \biggl(s,x-
\frac{\bar
{t}-s}{t-s}u \biggr)\biggr\rrvert ^p \biggr)
\biggr]^{{\trup{1}{2}}}
\nonumber
\\
\label{s3.511}&&{} \times \biggl[ \biggl(\E\biggl\llvert D_n
\biggl(s,x-\frac{\bar
{t}-s}{t-s}v \biggr)-D_n(s,x-v)\biggr\rrvert
^p \biggr) \biggr]^{{\trup
{1}{2}}}
\\
&\le& C\llvert \bar{t}-t\rrvert ^{\rho_1 {\trup{p}{2}}} \Bigl(\sup_{0\le s\le t\le\bar{t}\le T}
\nu_2(s,t,\bar{t}) \Bigr)^{{\trup
{p}{2}}}
\nonumber
\\
\label{s3.52} &\le& C|\bar{t}-t|^{(\rho_1+\alpha_1){\trup{p}{2}}},
\end{eqnarray}
with $\alpha_1\in\,]0,1\wedge(2-\beta)[$.

Similarly,
%
%e2.53 #&#
\begin{equation}
\label{s3.53} \E \bigl( \bigl |Q^3(t,\bar{t},x) \bigr |^{{\trup{p}{2}}} \bigr)
\le C|\bar{t}-t|^{(\rho_1+\alpha_1){\trup{p}{2}}},
\end{equation}
$\alpha_1\in\,]0,1\wedge(2-\beta)[$.

Define
\begin{eqnarray*}
\nu_4(s,t,\bar{t})&:=&\int_{\R^3} \int
_{\R^3} G(t-s,\mathrm{d}u)G(t-s,\mathrm{d}v)
\\
&&{}\times \biggl\{ \biggl(\frac{\bar{t}-s}{t-s} \biggr)^2 f \biggl(
\frac
{\bar{t}-s}{t-s}(v-u) \biggr)
\\\
&&\phantom{\times \biggl\{}{}-\frac{\bar{t}-s}{t-s} f \biggl(v\frac{\bar{t}-s}{t-s}-u
\biggr) -\frac{\bar{t}-s}{t-s} f \biggl(v-u\frac{\bar{t}-s}{t-s} \biggr)+f(v-u) \biggr
\}.
\end{eqnarray*}
Replacing $G_n$ by $G$ in \cite{dss}, Lemma~6.5 yields
%
%e2.54 #&#
\begin{equation}
\label{s3.54} \sup_{s\le t\le\bar{r}\le T} \nu_4(s,t,\bar{t})\le C
|\bar {t}-t|^{\alpha_2},
\end{equation}
where $\alpha_2\in\, ]0, (2-\beta) [$.

By applying H\"older's and Cauchy--Schwarz's inequalities along with
\eqref{s3.44}, we get
%
%e2.55 #&#
%e2.56 #&#
\begin{eqnarray}
&& \E \bigl( \bigl |Q^4(t,\bar{t},x) \bigr |^{{\trup{p}{2}}} \bigr)
\nonumber
\\
&&\quad \le \Bigl(\sup_{0\le s\le t\le\bar{t}\le T}\nu_4(s,t,\bar {t})
\Bigr)^{{\trup{p}{2}}-1} \int_0^t \mathrm{d}s\int
_{\R^3} \int_{\R^3} G(t-s,\mathrm{d}u)G(t-s,
\mathrm{d}v)
\nonumber
\\
&&\qquad{} \times \biggl\{ \biggl(\frac{\bar{t}-s}{t-s} \biggr)^2 f \biggl(
\frac{\bar{t}-s}{t-s}(v-u) \biggr) -\frac{\bar{t}-s}{t-s} f \biggl(v\frac{\bar{t}-s}{t-s}-u
\biggr)
\nonumber
\\
&&\phantom{\qquad{} \times\biggl\{} {} -  \frac{\bar{t}-s}{t-s} f \biggl(v-u
\frac{\bar
{t}-s}{t-s} \biggr)+f(v-u) \biggr\}
\nonumber
\\
\label{s3.541}&&\qquad{} \times \bigl[\E \bigl( \bigl |D_n(s,x-u)
\bigr |^p1_{L_n(s)} \bigr) \bigr]^{{\trup{1}{2}}} \bigl[\E \bigl(
\bigl |D_n(s,x-v) \bigr |^p \bigr) \bigr]^{{\trup
{1}{2}}}
\\
&&\quad \le C \Bigl(\sup_{0\le s\le t\le\bar{t}\le T}\nu_4(s,t,\bar {t})
\Bigr)^{{\trup{p}{2}}}
\nonumber
\\
\label{s3.55}&&\quad \le C |\bar{t}-t|^{\alpha_2 p /2},
\end{eqnarray}
with $\alpha_2\in\, ]0, (2-\beta) [$.

The inequalities \eqref{s3.50}, \eqref{s3.52}, \eqref{s3.53}, \eqref
{s3.55}, together with \eqref{s3.48} imply \eqref{s3.47}.
\end{pf}
%
%%%%%%%%%%%
%%%%%%%%%% Proof of Proposicion~\ref{pss3.2.1}
%%%%%%%%%%%

\begin{pf*}{Proof of Proposition~\ref{pss3.2.1}}
Fix $0\le t\le\bar{t}\le T$, $x\in K$, $p\in[2,\infty[$, and
according to \eqref{s3.6} consider the decomposition
\[
\E \bigl(\bigl |X_n(\bar{t},x)-X_n(t,x) \bigr |^p
1_{L_n(\bar{t})} \bigr) \le C \sum_{i=1}^6
R_n^i(t,\bar{t},x),
\]
where
\begin{eqnarray*}
R_n^1(t,\bar{t},x) &=&\E \biggl(\biggl\llvert \int
_0^{\bar{t}}\int_{\R^3} \bigl[G(
\bar {t}-s,x-y)-G(t-s,x-y) \bigr]
\\
&&\phantom{\E\bigl(|} {}   \times A\bigl(X_n(s,y)\bigr)
M(\mathrm{d}s,\mathrm{d}y) \biggr\rrvert ^p1_{L_n(\bar
{t})} \biggr),
\\
R_n^2(t,\bar{t},x) &=&\E {\bigl(} \bigl | \bigl\langle
\bigl[G(\bar{t}-\cdot,x-\ast )-G(t-\cdot,x-\ast) \bigr] B\bigl(X_n^-(
\cdot,\ast)\bigr),w^n \bigr\rangle _{\mathcal{H}_{\bar t}}
\bigr |^p1_{L_n(\bar{t})} {\bigr)},
\\
R_n^3(t,\bar{t},x) &=&\E \bigl( \bigl | \bigl\langle \bigl[G(\bar{t}-
\cdot,x-\ast )-G(t-\cdot,x-\ast) \bigr]
\\
&&\phantom{\E\bigl(|\bigl\langle} {}  \times\bigl[B(X_n)-B
\bigl(X_n^-\bigr)\bigr](\cdot,\ast),w^n \bigr\rangle
_{\mathcal{H}_{\bar t}} \bigr |^p1_{L_n(\bar{t})} \bigr),
\\
R_n^4(t,\bar{t},x) &=&\E \bigl( \bigl | \bigl\langle \bigl[G(
\bar{t}-\cdot,x-\ast )-G(t-\cdot,x-\ast) \bigr] D\bigl(X_n(\cdot,\ast)
\bigr),h \bigr\rangle _{\mathcal{H}_{\bar t}} \bigr |^p1_{L_n(\bar{t})} \bigr),
\\
R_n^5(t,\bar{t},x) &=&\E \biggl(\biggl\llvert \int
_0^{\bar{t}} \int_{\R^3} \bigl[G(
\bar {t}-s,x-\mathrm{d}y)-G(t-s,x-\mathrm{d}y) \bigr] b\bigl(X_n(s,y)
\bigr) \,\mathrm{d}s \biggr\rrvert ^p1_{L_n(\bar
{t})} \biggr).
\end{eqnarray*}

%%%%%%%%% Term R_n^1(t,\bar{t},x)
%%%%%%%%%
Similarly as for the term $R_n^1(t,x,\bar x)$ in the proof of Lemma~\ref{lss3.1.2} (see \eqref{s3.32}), we have
%
%e2.57 #&#
\begin{eqnarray}
\label{s3.571}
 R_n^1(t,\bar{t},x) &\le& C\E \biggl(\int
_0^{\bar t} \mathrm{d}s \bigl \| \bigl[G(\bar t-s,x-\ast
)
\nonumber
\\[-8pt]
\\[-8pt]
&&\phantom{C\E \biggl(\int
_0^{\bar t} \mathrm{d}s \bigl \| \bigl[}{}-G(t-s,x-\ast) \bigr] A\bigl(X_n(s,\ast)\bigr)\bigr  \|_{\mathcal{H}}^2
1_{L_n(s)} \biggr)^{{\trup{p}{2}}}.\nonumber
\end{eqnarray}
This is bounded up to a positive constant by $R_n^{1,1}(t,\bar
{t},x)+R_n^{1,2}(t,\bar{t},x)$, where
%
%e2.58 #&#
\begin{eqnarray}
\label{s3.58} R_n^{1,1}(t,\bar{t},x) &=&\E \biggl(\biggl
\llvert \int_t^{\bar{t}} \bigl \| G(\bar{t}-s,x-\ast) A
\bigl(X_n(s,\ast)\bigr) \bigr \|_{\mathcal{H}}^2
1_{L_n(s)}\,\mathrm{d}s \biggr\rrvert \biggr)^{{\trup{p}{2}}}
\nonumber
\\[-8pt]
\\[-8pt]
&=&\E \biggl(\biggl\llvert \int_0^{\bar{t}-t} \bigl \|
G(s,x-\ast) A\bigl(X_n(\bar t-s,\ast)\bigr) \bigr \|_{\mathcal{H}}^2
1_{L_n(s)} \,\mathrm{d}s \biggr\rrvert \biggr)^{{\trup{p}{2}}}
\nonumber
\end{eqnarray}
and
%
%e2.59 #&#
\begin{eqnarray}
\label{s3.581} R_n^{1,2}(t,\bar{t},x) &=&\E \biggl(\biggl
\llvert \int_0^t \mathrm{d}s \bigl\llVert
\bigl[G(\bar t-s,x-\ast )
\nonumber
\\[-8pt]
\\[-8pt]
&&\phantom{\E \biggl(\biggl
\llvert \int_0^t \mathrm{d}s \bigl\llVert
\bigl[}{}-G(t-s,x-\ast)\bigr]A\bigl(X_n(s,\ast)\bigr)
\bigr\rrVert ^2_{\mathcal{H}} 1_{L_n(s)}\biggr\rrvert
\biggr)^{{\trup{p}{2}}}.\nonumber
\end{eqnarray}

Set
\[
\mu_1(t,\bar{t},x):=\int_0^{\bar{t}-t}
\mathrm{d}s\int_{\R^3} \mathrm{d}\xi\bigl \vert \mathcal{F}G(s) (\xi)
\bigr \vert^2 \mu(\mathrm{d}\xi).
\]
Lemma~2.2 in \cite{dss} shows that
%
%e2.60 #&#
\begin{equation}
\label{s3.59} \mu_1(t,\bar{t},x)\le C|\bar{t}-t|^{3-\beta}.
\end{equation}
Then, using H\"older's inequality, the linear growth of $A$ and \eqref
{s4.18}, we obtain
%
%e2.61 #&#
%e2.62 #&#
\begin{eqnarray}
\label{s3.591}R_n^{1,1}(t,\bar{t},x) &\le& C \bigl(
\mu_1(t,\bar{t},x) \bigr)^{\trup{p}{2}} \Bigl(1+ \sup
_{(t,x)\in[0,T]\times\R^3}\E \bigl( \bigl | X_n(t,x) \bigr |^p1_{L_n(t)}
\bigr) \Bigr)
\\
\label{s3.60}&\le& C |\bar{t}-t|^{p{\trup{(3-\beta)}{2}}}.
\end{eqnarray}

Set $D_n(t,x)=A(X_n(t,x)) 1_{L_n(t)}$. Owing to Hypothesis~\ref{HypB}, \eqref{s4.18} and Proposition~\ref{pss3.1.1}, the conditions
\eqref{s3.44}, \eqref{s3.45} of Lemma~\ref{lss3.2.1}
are satisfied with $\rho_1\in\, ]0,\frac{2-\beta}{2} [$. Thus,
%
%e2.63 #&#
\begin{equation}
\label{s3.61} R_n^{1,2}(t,\bar{t},x)\le C \bigl( |
\bar{t}-t|^{\rho_1 p} + |\bar {t}-t|^{(\rho_1+\alpha_1){\trup{p}{2}}} + |\bar{t}-t|^{\alpha_2 p
/2}
\bigr),
\end{equation}
with $\rho_1 \in\, ]0,\frac{2-\beta}{2} [$, $\alpha_1\in\,]0,1\wedge(2-\beta)[$ and $\alpha_2\in\, ]0,(2-\beta) [$.

It is easy to check that $\frac{2-\beta}{2}+(1\wedge(2-\beta))\ge
(2-\beta)$. Hence, from \eqref{s3.61} we obtain
%
%e2.64 #&#
\begin{equation}
\label{s3.62} R_n^{1,2}(t,\bar{t},x) \le C |
\bar{t}-t|^{\rho p},\qquad \rho\in\, \biggl]0,\frac{2-\beta}{2} \biggr[.
\end{equation}

Since $\frac{3-\beta}{2}\ge\frac{2-\beta}{2}$, \eqref{s3.60} and
\eqref{s3.62} imply
%
%e2.65 #&#
\begin{equation}
\label{s3.63} R_n^1(t,\bar{t},x)\le C|
\bar{t}-t|^{\rho p },\qquad \rho\in\, \biggl]0,\frac
{2-\beta}{2}\biggr [.
\end{equation}

%%%%%
%%%%%Term R_n^2

With the same arguments as those applied in the study of the term
$R_n^2(t,x,\bar x)$ in the proof of Lemma~\ref{lss3.1.2}, we have
\[
R_n^2(t,\bar{t},x) \le C \E \biggl(\int
_0^{\bar t} \mathrm{d}s \bigl \Vert\bigl[G(\bar t-s,x+\ast
)-G(t-s,x-\ast)\bigr] B\bigl(X_n^-(s,\ast)\bigr)\bigr \Vert_{\mathcal{H}}^2
1_{L_n(s)} \biggr)^{\trup{p}{2}}.
\]
This yields $R_n^2(t,\bar{t},x)\le C(R_n^{2,1}(t,\bar
{t},x)+R_n^{2,2}(t,\bar{t},x))$,
where
\begin{eqnarray*}
R_n^{2,1}(t,\bar{t},x)& =& \E \biggl(\int
_0^t \mathrm{d}s \bigl \Vert\bigl[G(\bar t-s,x+
\ast)-G(t-s,x-\ast)\bigr] B\bigl(X_n^-(s,\ast)\bigr)
\bigr \Vert_{\mathcal{H}}^2 1_{L_n(s)} \biggr)^{\trup{p}{2}},
\\
R_n^{2,2}(t,\bar{t},x)& =& \E \biggl(\int
_0^{\bar t-t} \bigl \Vert G(s,x-\ast)B\bigl(X_n^-(s,
\ast)\bigr)\bigr \Vert_{\mathcal{H}}^2 1_{L_n(s)}
\biggr)^{\trup{p}{2}}.
\end{eqnarray*}

The term $R_n^{2,1}(t,\bar{t},x)$ is similar as $R_n^{1,2}(t,\bar
{t},x)$, with $A(X_n)$ replaced by $B(X_n^-)$. Hence both can be
studied using the same approach.
First, we see that the process $D_n(t,x):=B(X_n^-(t,x)) 1_{L_n(t)}$
satisfies the hypothesis of Lemma~\ref{lss3.2.1} with $\rho_1\in\,
]0,\frac{2-\beta}{2} [$. In fact, this is a consequence of \eqref
{s4.18} and Proposition~\ref{pss3.1.2}. Therefore, as for
$R_n^{1,2}(t,\bar{t},x)$, we have
%
%e2.66 #&#
\begin{equation}
\label{s3.633} R_n^{2,1}(t,\bar{t},x)\le C |
\bar{t}-t|^{\rho p},\qquad \rho\in\, \biggl]0,\frac{2-\beta}{2}\biggr [.
\end{equation}

As for $R_n^{2,2}(t,\bar{t},x)$, it is analogous to $R_n^{1,1}$ with
$A(X_n)$ replaced by $B(X_n^-)$. As in \eqref{s3.60}, we have
%
%e2.67 #&#
\begin{equation}
\label{s3.64} R_n^{2,2}(t,\bar{t},x) \le C |
\bar{t}-t|^{ p{\trup{(3-\beta)}{2}} }.
\end{equation}
Consequently, from \eqref{s3.633}, \eqref{s3.64}, we obtain
%
%e2.68 #&#
\begin{equation}
\label{s3.640} R_n^2(t,\bar{t},x) \le C |
\bar{t}-t|^{\rho p },\qquad \rho\in\, \biggl]0,\frac{2-\beta}{2} \biggr[.
\end{equation}

%%%%%%%%%%%%%
%%%%%%%%%%%%% R_n^3

Let $\hat{B}(X_n(\cdot,\ast))$ be defined by \eqref{s3.351}. Using
Cauchy--Schwarz's inequality and \eqref{s3.101} we have
%
%e2.69 #&#
\begin{equation}
\label{s3.641} R_n^3(t,\bar{t},x)\le C
n^{\trup{3p}{2}}2^{n{\trup{p}{2}}} \bigl[R_n^{3,1}(t,
\bar{t},x)+R_n^{3,2}(t,\bar{t},x) \bigr],
\end{equation}
where
\begin{eqnarray*}
R_n^{3,1}(t,\bar{t},x) &=&\E \biggl(\biggl\llvert \int
_0^t \mathrm{d}s \bigl\llVert \bigl[G(
\bar{t}-s,x-\ast )-G(t-s,x-\ast)\bigr] \hat{B}\bigl(X_n(s,\ast)\bigr)
\bigr\rrVert ^2_{\mathcal
{H}}1_{L_n(s)}\biggr\rrvert
\biggr)^{\trup{p}{2}},
\\
R_n^{3,2}(t,\bar{t},x)&=&\E \biggl( \biggl\llvert \int
_0^{\bar{t}-t} \mathrm{d}s\bigl\llVert G(s,x-\ast)
\hat{B}\bigl(X_n(\bar t-s,\ast)\bigr)\bigr\rrVert ^2_{\mathcal{H}}1_{L_n(s)}
\biggr\rrvert \biggr)^{\trup{p}{2}}.
\end{eqnarray*}

From \eqref{s4.19}, it follows that
%
%e2.70 #&#
\begin{equation}
\label{s3.65} \sup_{(t,x)\in[0,T\times\IR^3]} \E \bigl( \bigl |\hat {B}
\bigl(X_n(t,x)\bigr)\bigr  |^p1_{L_n(t)} \bigr)\le C
n^{\trup
{3p}{2}}2^{-np{\trup{(3-\beta)}{2}}}.
\end{equation}

Let us study $R_n^{3,2}(t,\bar{t},x)$. This term is similar to
$R_n^{1,1}(t,\bar{t},x)$ with $A(X_n)$ replaced here by $\hat
{B}(X_n)$. Hence, as in \eqref{s3.591} we have
%
%e2.71 #&#
\begin{eqnarray}
\label{s3.66} R_n^{3,2}(t,\bar{t},x)&\le& \bigl(
\mu_1(t,\bar{t},x) \bigr)^{\trup
{p}{2}} \Bigl(\sup
_{(t,x)\in[0,T]\times\IR^3}\E \bigl( \bigl |\hat {B}\bigl(X_n(t,x)\bigr)
\bigr |^p1_{L_n(t)} \bigr) \Bigr)
\nonumber
\\[-8pt]
\\[-8pt]
&\le& C |\bar{t}-t|^{p{\trup{(3-\beta)}{2}}} n^{\trup
{3p}{2}}2^{-np{\trup{(3-\beta)}{2}}},
\nonumber
\end{eqnarray}
where in the last inequality we have applied \eqref{s3.65}.

%%%%%%%%
%%%%%%%% R_n^{3,1}
%%%%%%%%
%%%%%%%
The analysis of $R_n^{3,1}$ relies on a variant of Lemma~\ref
{lss3.2.1} where the process $D_n$ is replaced by $\hat{B}(X_n)$. By
\eqref{s3.65}, this process satisfies a stronger assumption than
\eqref{s3.44}. This fact is expected to compensate the factor
$n^{\trup{3p}{2}}2^{n{\trup{p}{2}}}$ in \eqref{s3.641}.

As in the proof of Lemma~\ref{lss3.2.1} (see also \cite{dss}, page
28), we consider the decomposition
\[
R_n^{3,1}(t,\bar{t},x)\le\sum
_{k=1}^4 \E \bigl(\bigl |Q^i(t,\bar
{t},x)\bigr |^{{\trup{p}{2}}}1_{L_n(\bar{t})} \bigr),
\]
where $Q^i(t,\bar{t},x)$, $i=1,\ldots,4$, are defined in \eqref
{s3.481} with $D_n:=\hat{B}(X_n) 1_{L_n}$.

From \eqref{s3.65} and the triangular inequality, we obtain
%
%e2.72 #&#
\begin{equation}
\label{s3.68} \E \biggl(\biggl\llvert \hat{B} \biggl(X_n
\biggl(s,x-\frac{\bar
{t}-s}{t-s}u \biggr) \biggr)-\hat{B}\bigl(X_n(s,x-u)
\bigr) \biggr\rrvert ^p 1_{L_n(s)} \biggr) \le C
n^{\trup{3p}{2}} 2^{-np (3-\beta)/ 2}.
\end{equation}
Consider the expression \eqref{s3.491} with $D_n=\hat
{B}(X_n)1_{L_n}$. The above estimate \eqref{s3.68} yields
\begin{eqnarray*}
&& \E \bigl( \bigl |Q^1(t,\bar{t},x)\bigr  |^{{\trup{p}{2}}}1_{L_n(\bar
{t})}
\bigr)
\\
&&\quad \le C n^{\trup{3p}{2}}2^{-np{\trup{(3-\beta)}{2}}} \int_0^t
\mathrm{d}s\int_{R^3} \int_{R^3} G(t-s,
\mathrm{d}u)G(t-s,\mathrm{d}v)\frac{\bar
{t}-s}{t-s} f \biggl(v\frac{\bar{t}-s}{t-s}-u
\biggr).
\end{eqnarray*}
Along with \eqref{s3.49}, this implies
%
%e2.73 #&#
\begin{equation}
\label{s3.69} \E \bigl( \bigl |Q^1(t,\bar{t},x)\bigr  |^{{\trup{p}{2}}}1_{L_n(\bar
{t})}
\bigr)\le C n^{\trup{3p}{2}}2^{-np{\trup{(3-\beta)}{2}}}.
\end{equation}

%%%%%%%$Q^2$ and $Q^3$

Consider the expression \eqref{s3.511} with $D_n=\hat{B}(X_n)
1_{L_n}$. Using \eqref{s3.351}, the Lipschitz property of $B$ and
\eqref{s4.19}, we obtain
%
%e2.74 #&#
\begin{equation}
\label{s3.70} \E \bigl( \bigl |Q^2(t,\bar{t},x)\bigr  |^{{\trup{p}{2}}}1_{L_n(\bar
{t})}
\bigr) \le Cn^{\trup{3p}{2}}2^{-np\trup{(3-\beta)}{2}}.
\end{equation}

Similarly,
%
%e2.75 #&#
\begin{equation}
\label{s3.71} \E \bigl( \bigl |Q^3(t,\bar{t},x)\bigr  |^{{\trup{p}{2}}}1_{L_n(\bar
{t})}
\bigr) \le Cn^{\trup{3p}{2}}2^{-np\trup{(3-\beta)}{2}}.
\end{equation}

Let us now consider the expression \eqref{s3.541} with $D_n=\hat
{B}(X_n) 1_{L_n}$. Appealing to \eqref{s3.65}, we obtain
%
%e2.76 #&#
\begin{equation}
\label{s3.72} \E \bigl( \bigl |Q^4(t,\bar{t},x)\bigr  |^{{\trup{p}{2}}}1_{L_n(\bar
{t})}
\bigr) \le C n^{\trup{3p}{2}}2^{-np\trup{(3-\beta)}{2}}.
\end{equation}

%%%%%%%%%%
%%%%%%%%%%
%%%%%%%%%%
%%%%%%%%%%

From \eqref{s3.69}--\eqref{s3.72} it follows that
%
%e2.77 #&#
\begin{equation}
\label{s3.73} R_n^{3,1}(t,\bar t,x) \le
Cn^{\trup{3p}{2}}2^{-np \trup{(3-\beta
)}{2} },
\end{equation}
where $C$ is a finite constant.

Set $f_n:=n^{3p}2^{-np (\trup{(3-\beta)}{2}-\trup{1}{2} )}$.
From \eqref{s3.641}, \eqref{s3.66}, \eqref{s3.73}, it follows that\vspace*{-1pt}
%
%e2.78 #&#
\begin{equation}
\label{s3.74} R_n^3(t,\bar{t},x)\le C |\bar{t}-t
|^{\rho p}+Cf_n,\qquad \rho \in\, \biggl]0,\frac{2-\beta}{2} \biggr[.
\end{equation}
%
%%%%%%%%%%
%%%%%%%%%%
%%%%%%%%%%
%%%%%%%%%%R^4

By applying Cauchy--Schwarz's inequality, we see that\vspace*{-1pt}
\[
R_n^4(t,x,\bar{x}) \le C\E \biggl(\int
_0^{\bar t} \mathrm{d}s \bigl\llVert \bigl[G(\bar
t-s,x-\ast )-G(t-s,x-\ast)\bigr] D\bigl(X_n(s,\ast)\bigr)\bigr\rrVert
_{\mathcal
{H}}^21_{L_n(s)} \biggr)^{\trup{p}{2}}.
\]
The last expression is similar as \eqref{s3.571} with the function $A$
replaced by $D$. Therefore, as in \eqref{s3.63} we obtain\vspace*{-1pt}
%
%e2.79 #&#
\begin{equation}
\label{s3.75} R_n^4(t,\bar{t},x)\le C |
\bar{t}-t|^{\rho p },\qquad \rho\in\, \biggl]0,\frac{2-\beta}{2} \biggr[.
\end{equation}

%%%%%%%%
%%%%%%%% R_n^5

Finally, we consider $R_n^5(t,\bar{t},x)$. Clearly,\vspace*{-1pt}
\[
R_n^5(t,\bar{t},x)\le C \bigl[R_n^{5,1}(t,
\bar{t},x)+R_n^{5,2}(t,\bar {t},x) \bigr],
\]
where\vspace*{-1pt}
\begin{eqnarray*}
R_n^{5,1}(t,\bar{t},x) &:=&\E \biggl(\biggl\llvert \int
_0^t \int_{\R^3} \bigl[G(
\bar{t}-s,x-\mathrm{d}y)-G(t-s,x-\mathrm{d}y)\bigr] b\bigl(X_n(s,y)
\bigr) \,\mathrm{d}s \biggr\rrvert ^p 1_{L_n(\bar{t})} \biggr),
\\
R_n^{5,2}(t,\bar{t},x)&:=&\E \biggl(\biggl\llvert \int
_t^{\bar{t}} \int_{\R
^3} G(
\bar{t}-s,x-\mathrm{d}y) b\bigl(X_n(s,y)\bigr) \,\mathrm{d}s \biggr
\rrvert ^p1_{L_n(\bar{t})} \biggr).
\end{eqnarray*}

Applying the change of variable, $y\mapsto\frac{y-x}{\overline
{t}-s}+ x $ and $y\mapsto\frac{y-x}{t-s}+x$, we see that\vspace*{-1pt}
\[
R_n^{5,1}(t,\bar{t},x)= \E \bigl( \bigl |T_1(t,
\bar{t},x)-T_2(t,\bar {t},x) \bigr |^p 1_{L_n(\bar{t})} \bigr),
\]
where\vspace*{-1pt}
\begin{eqnarray*}
T_1(t,\bar{t},x)&=&\int_0^t (
\bar{t}-s)\int_{\R^3} G(1,x-\mathrm{d}y) b\bigl(X_n
\bigl(s,(\bar{t}-s) (y-x)+x \bigr)\bigr) \,\mathrm{d}s,
\\
T_2(t,\bar{t},x)&=&\int_0^t (t-s)
\int_{\R^3} G(1,x-\mathrm{d}y) b\bigl(X_n
\bigl(s,(t-s) (y-x)+x\bigr)\bigr)\,\mathrm{d}s.
\end{eqnarray*}
By adding and subtracting $t$ in $T_1$ we get\vspace*{-1pt}
\begin{eqnarray*}
T_1(t,\bar{t},x)&=&\int_0^t (
\bar{t}-t) \int_{\R^3} G(1,x-\mathrm{d}y) b
\bigl(X_n\bigl(s,(\bar{t}-s) (y-x)+x \bigr)\bigr) \,\mathrm{d}s
\\
&&{}+\int_0^t (t-s)\int_{\R^3}
G(1,x-\mathrm{d}y) b\bigl(X_n\bigl(s,(\bar{t}-s) (y-x)+x \bigr)\bigr)
\,\mathrm{d}s.
\end{eqnarray*}
Then, H\"older's inequality yields
\begin{eqnarray*}
&&R_n^{5,1}(t,\bar{t},x)
\\
&&\quad\le  C |\bar{t}-t|^p
\int_0^t \mathrm{d}s\int_{\R^3}
G(1,x-\mathrm{d}y)\E \bigl( \bigl |b \bigl(X_n\bigl(s,(\bar{t}-s) (y-x)+x
\bigr) \bigr) \bigr |^p1_{L_n(s)} \bigr)
\\
 &&\qquad {} + C \int_0^t |t-s|^{p}\,
\mathrm{d}s\int_{\R^3} G(1,x-\mathrm{d}y)\E \bigl( \bigl |b
\bigl(X_n\bigl(s,(\bar{t}-s) (y-x)+x \bigr) \bigr)
\\
&&\phantom{\qquad {} + C \int_0^t |t-s|^{p}\,
\mathrm{d}s\int_{\R^3} G(1,x-\mathrm{d}y)\E \bigl( \bigl |} {} -b
\bigl(X_n\bigl(s,(t-s) (y-x)+x \bigr) \bigr) \bigr |^p
1_{L_n(s)} \bigr).
\end{eqnarray*}

Owing to \eqref{s4.18}, the first term on the right hand-side of the
last inequality is bounded up to a constant by $|\bar{t}-t|^p$. For
the second one, we use the Hypothesis~\ref{HypB} along with
\eqref{s3.17} to obtain
\begin{eqnarray*}
&&\int_0^t |t-s|^{p} \,
\mathrm{d}s\int_{\R^3} G(1,x-\mathrm{d}y)
\\
&&\qquad{} \times\E \bigl( \bigl |b \bigl(X_n\bigl(s,(\bar{t}-s) (y-x)+x
\bigr) \bigr)-b \bigl(X_n\bigl(s,(t-s) (y-x)+x \bigr) \bigr)
\bigr |^p 1_{L_n(s)} \bigr)
\\
&&\quad \le C \int_0^t \mathrm{d}s \int
_{\R^3} G(1,x-\mathrm{d}y)
\\
&&\qquad{} \times\E \bigl( \bigl |X_n\bigl(s,(\bar{t}-s) (y-x)+x
\bigr)-X_n\bigl(s,(t-s) (y-x)+x \bigr)\bigr  |^p
1_{L_n(s)} \bigr)
\\
&&\quad \le C |t-\bar t|^{\rho p},
\end{eqnarray*}
with $\rho\in\, ]0,\frac{2-\beta}{2} [$.

H\"older inequality along with \eqref{s4.18} clearly yields
%
%e2.80 #&#
\begin{eqnarray}
\label{s3.56} R_n^{5,2}(t,\bar{t},x)&\le& C|
\bar{t}-t|^{p-1} \int_t^{\bar{t}} \int
_{\R^3} G(\bar{t}-s,x-\mathrm{d}y) \E \bigl(\bigl\llvert b
\bigl(X_n(s,y)\bigr) \bigr\rrvert ^p1_{L_n(s)}
\bigr)\,\mathrm{d}s
\nonumber
\\[-8pt]
\\[-8pt]
&\le& C|\bar{t}-t|^{p}.
\nonumber
\end{eqnarray}

Hence, we have proved that
%
%e2.81 #&#
\begin{equation}
\label{s3.57} R_n^5(t,\bar{t},x)\le C |\bar{t}-t
|^{\rho p},\qquad \rho\in\, \biggl]0,\frac{2-\beta}{2} \biggr[.
\end{equation}

With the inequalities \eqref{s3.63}, \eqref{s3.640}, \eqref{s3.74},
\eqref{s3.75} and \eqref{s3.57}, we have
\[
\E \bigl(\bigl |X_n(\bar{t},x)-X_n(t,x)\bigr |^p1_{L_n(\bar{t})}
\bigr)\le C \bigl[|\bar{t}-t|^{\rho p } + f_n \bigr],
\]
with $\rho\in\, ]0,\frac{2-\beta}{2} [$.

For a given fixed $\bar t\in[t_0,T]$, we introduce the function
\[
\varPsi_{n,x,p}^{\bar{t}}(t):=\E \bigl(\bigl |X_n(\bar
{t},x)-X_n(t,x)\bigr |^p1_{L_n(\bar{t})} \bigr),
\]
for $t_0\le t\le\bar{t}$.\vadjust{\goodbreak}

Notice that $\lim_{n\to\infty} f_n=0$ and thus, $\sup_n f_n \le C$.
Thus, there exists
a constant $0<C_0<\infty$, such that
\[
\sup_n f_n\le C_0 t_0
\le C_0\bar{t}\le C_0 \int_0^{\bar{t}}
\mathrm{d}s \bigl[1+\varPsi_{n,x,p}^{\bar{t}}(s) \bigr].
\]
With a similar argument, there exists $0<C_1<\infty$ such that
\[
1\le C_1 t_0 \le C_1\bar{t}\le
C_1 \int_0^{\bar{t}} \mathrm{d}s
\bigl[1+\varPsi_{n,x,p}^{\bar{t}}(s) \bigr].
\]
Therefore,
\[
1+ \varPsi_{n,x,p}^{\bar{t}}(t) \le C \biggl\{|
\bar{t}-t|^{\rho p} + \int_0^{\bar{t}}
\mathrm{d}s \bigl[1+\varPsi_{n,x,p}^{\bar{t}}(s) \bigr] \biggr\}.
\]
Then, by Gronwall's lemma,
\[
1+ \varPsi_{n,x,p}^{\bar{t}}(t) \le C \bigl(|\bar{t}-t|^{\rho p}
\bigr),
\]
where $\rho\in\, ]0,\frac{2-\beta}{2} [$. This finish the
proof of the proposition.
\end{pf*}

%%%%%%
%%%%%%%POINTWISE CONVERGENCE
%%%%%%%%

%s2.3 #&#
\subsection{Pointwise convergence}
\label{ss3.3}

This section is exclusively devoted to the proof of Theorem~\ref
{ts3.3}. Using equations \eqref{s3.7}, \eqref{s3.6}, we write the
difference $X_n(t,x)-X(t,x)$ grouped into comparable terms
in order to prove their convergence to zero. The main difficulty lies
in the proof of the convergence of $\langle G(t-\cdot,x-\ast)
B(X_n(\cdot,\ast)),w^n\rangle_{\hact}$ to
$\int_0^t\int_{\R^3} B(X(s,y)) M(\mathrm{d}s,\mathrm{d}y)$. We write
\[
X_n(t,x)-X(t,x)=\sum_{i=1}^8
U_n^i(t,x),
\]
where
\begin{eqnarray*}
U_n^1(t,x) &=& \int_0^t
\int_{\R^3} G(t-s,x-y) \bigl[(A+B) \bigl(X_n(s,y)
\bigr)-(A+B) \bigl(X(s,y)\bigr) \bigr] M(\mathrm{d}s,\mathrm{d}y),
\\
U_n^2(t,x) &=& \bigl\langle G(t-\cdot,x-\ast)\bigl[D
\bigl(X_n(\cdot,\ast)\bigr)-D\bigl(X(\cdot ,\ast)\bigr)\bigr], h \bigr
\rangle_{\mathcal{H}_t},
\\
U_n^3(t,x) &=& \int_0^t
\mathrm{d}s\int_{\R^3} G(t-s,x-\mathrm{d}y)\bigl[b
\bigl(X_n(s,y)\bigr)-b\bigl(X(s,y)\bigr)\bigr],
\\
U_n^4(t,x) &=& \bigl\langle G(t-\cdot,x-\ast)\bigl[B
\bigl(X_n(\cdot,\ast )\bigr)-B\bigl(X_n^-(\cdot,\ast)
\bigr)\bigr], w^n \bigr\rangle_{\mathcal{H}_t},
\\
U_n^5(t,x) &=& \bigl\langle G(t-\cdot,x-\ast)\bigl[B
\bigl(X_n^-(\cdot,\ast )\bigr)-B\bigl(X^-(\cdot,\ast)\bigr)\bigr],
w^n \bigr\rangle_{\mathcal{H}_t},
\\
U_n^6(t,x) &=& \bigl\langle G(t-\cdot,x-\ast)B
\bigl(X^-(\cdot,\ast)\bigr), w^n \bigr\rangle_{\mathcal{H}_t}
\\
&&{} - \int_0^t\int_{\R^3}
G(t-s,x-y)B\bigl(X^-(s,y)\bigr) M(\mathrm{d}s,\mathrm{d}y),
\\
U_n^7(t,x) &=& \int_0^t
\int_{\R^3} G(t-s,x-y)\bigl[B\bigl(X^-(s,y)\bigr)-B
\bigl(X_n^-(s,y)\bigr)\bigr] M(\mathrm{d}s,\mathrm{d}y),
\\
U_n^8(t,x) &=& \int_0^t
\int_{\R^3} G(t-s,x-y)\bigl[B\bigl(X_n^-(s,y)
\bigr)-B\bigl(X_n(s,y)\bigr)\bigr] M(\mathrm{d}s,\mathrm{d}y).
\end{eqnarray*}
Here, we have used the abridged notation $X^-(\cdot,\ast)$ for the
stochastic process $X^-(t,x):=X(t,t_n,x)$ defined in \eqref{s3.8.3}.
Notice that, although this is not apparent in the notation $X^-(\cdot
,\ast)$ does depend on $n$.

Fix $p\in[2,\infty[$. Clearly,
\[
\E \bigl(\bigl\llvert X_n(t,x)-X(t,x)\bigr\rrvert
^p1_{L_n(t)} \bigr)\le C\sum_{i=1}^8
\E \bigl(\bigl\llvert U_n^i(t,x)\bigr\rrvert
^p1_{L_n(t)} \bigr).
\]
Next, we analyze the contribution of each term $U_n^i(t,x)$,
$i=1,\ldots,8$.

Burkholder's and H\"older's inequalities yield
%
%e2.82 #&#
\begin{equation}
\label{s3.76} \E \bigl(\bigl\llvert U_n^1(t,x)\bigr
\rrvert ^p1_{L_n(t)} \bigr)\le C \int_0^t
\mathrm{d}s \Bigl[\sup_{y \in K(s)} \E \bigl( \bigl |X_n(s,y)-X(s,y)
\bigr |^{p} 1_{L_n(s)} \bigr) \Bigr].
\end{equation}
Cauchy--Schwarz's inequality implies
\[
\E \bigl(\bigl\llvert U_n^2(t,x)\bigr\rrvert
^p1_{L_n(t)} \bigr) \le \|h\|_{\mathcal{H}_t}^p\E
\bigl(\bigl\llVert G(t-\cdot,x-\ast) \bigl[D\bigl(X_n(\cdot,\ast)
\bigr)-D\bigl(X(\cdot,\ast)\bigr)\bigr]1_{L_n(t)}\bigr\rrVert
_{\mathcal
{H}_t}^2 \bigr)^{{\trup{p}{2}}}.
\]
Then, by using H\"older's inequality we obtain
%
%e2.83 #&#
\begin{equation}
\label{s3.77} \E \bigl(\bigl\llvert U_n^2(t,x)\bigr
\rrvert ^p1_{L_n(t)} \bigr)\le C \int_0^t
\mathrm{d}s \Bigl[\sup_{y \in K(s)} \E \bigl( \bigl |X_n(s,y)-X(s,y)
\bigr |^{p} 1_{L_n(s)} \bigr) \Bigr].
\end{equation}

For $U_n^3(t,x)$, we apply H\"older's inequality. This yields
%
%e2.84 #&#
\begin{equation}
\label{s3.78} \E \bigl(\bigl\llvert U_n^3(t,x)\bigr
\rrvert ^p1_{L_n(t)} \bigr)\le C \int_0^t
\mathrm{d}s \Bigl[\sup_{y \in K(s)} \E \bigl( \bigl |X_n(s,y)-X(s,y)
\bigr |^{p} 1_{L_n(s)} \bigr) \Bigr].
\end{equation}

Let $\tau_n$ and $\pi_n$ be the operators defined in the proof of
Lemma~\ref{lss3.1.2} (see \eqref{s3.34} and lines thereafter).
Let $I_{\mathcal{H}_t}$ be the identity operator on $\mathcal{H}_t$.
$\Upsilon_t:= (\pi_n \circ\tau_n )
-I_{\mathcal{H}_t}$ is a contraction operator on~$\mathcal{H}_t$.

After having applied Burkholder's inequality, we obtain
\begin{eqnarray*}
&&\E \bigl[ \bigl(\bigl\llvert U_n^5(t,x) +
U_n^7(t,x)\bigr\rrvert ^p
\bigr)1_{L_n(t)} \bigr]
\\
&&\quad \le C \E \bigl( \bigl \| \Upsilon_t \bigl[G(t-\cdot,x-\ast)\bigl
\{B\bigl(X_n^-\bigr)-B\bigl(X^-\bigr)\bigr\} \bigr](\cdot,\ast
)1_{L_n(\cdot)}\bigr  \|_{\mathcal{H}_t}^p \bigr)
\\
&&\quad \le C \E \biggl(\int_0^t\mathrm{d}s
\bigl\llVert\bigl [G(t-s,x-\ast)\bigl\{B\bigl(X_n^-\bigr)-B\bigl(X^-\bigr)
\bigr\}\bigr] (s,\ast)1_{L_n(s)} \bigr\rrVert _{\mathcal{H}}^2
\biggr)^{{\trup{p}{2}}}.
\end{eqnarray*}
Similarly as for $U_n^2(t,x)$, we have
%
%e2.85 #&#
\begin{eqnarray}
\label{s3.79}
&&\E \bigl[ \bigl(\bigl\llvert U_n^5(t,x)
+ U_n^7(t,x)\bigr\rrvert ^p
\bigr)1_{L_n(t)} \bigr]
\nonumber
\\[-8pt]
\\[-8pt]
&&\quad \le C \int_0^t
\mathrm{d}s \Bigl[ \sup_{y\in K(s)}\E \bigl( \bigl |X_n^-(s,y)-X^-(s,y)
\bigr |^p 1_{L_n(s)} \bigr) \Bigr].\nonumber
\end{eqnarray}
This clearly implies
\begin{eqnarray*}
\E \bigl[ \bigl(\bigl\llvert U_n^5(t,x) +
U_n^7(t,x)\bigr\rrvert ^p
\bigr)1_{L_n(t)} \bigr] &\le& C \biggl\{\int_0^t
\mathrm{d}s \Bigl[ \sup_{y\in K(s)}\E \bigl( \bigl |X_n^-(s,y)-X_n(s,y)
\bigr |^p 1_{L_n(s)} \bigr) \Bigr]
\\
&&\phantom{C \biggl\{} {}   +\int_0^t
\mathrm{d}s \Bigl[ \sup_{y\in K(s)}\E \bigl( \bigl |X_n(s,y)-X(s,y)
\bigr |^p 1_{L_n(s)} \bigr) \Bigr]
\\
&&\phantom{C \biggl\{} {}   + \int_0^t
\mathrm{d}s \Bigl[ \sup_{y\in K(s)}\E \bigl( \bigl |X(s,y)-X^-(s,y)
\bigr |^p 1_{L_n(s)} \bigr) \Bigr] \biggr\}.
\end{eqnarray*}
Recall that $X^-(s,y)=X(s,s_n,y)$. By applying \eqref{s4.4} and \eqref
{s4.19}, we obtain
%
%e2.86 #&#
\begin{eqnarray}
\label{s3.80} \E \bigl[ \bigl(\bigl\llvert U_n^5(t,x)
+ U_n^7(t,x)\bigr\rrvert ^p
\bigr)1_{L_n(t)} \bigr] &\le& C \int_0^T
\mathrm{d}s \Bigl[ \sup_{y\in K(s)}\E \bigl( \bigl |X_n(s,y)-X(s,y)
\bigr |^p 1_{L_n(s)} \bigr) \Bigr]\qquad
\nonumber
\\[-8pt]
\\[-8pt]
&&{} + Cn^{\trup{3p}{2}}2^{-np{\trup{(3-\beta)}{2}}}.
\nonumber
\end{eqnarray}

Next, we will prove that
%
%e2.87 #&#
\begin{equation}
\label{s3.81} \lim_{n\to\infty} \Bigl(\sup_{t\in[0,T]}
\sup_{x\in K(t)}\E \bigl(\bigl\llvert U_n^i(t,x)
\bigr\rrvert ^p1_{L_n(t)} \bigr) \Bigr)=0,\qquad i=4,6,8.
\end{equation}

Consider $i=4$. Cauchy--Schwarz' inequality along with \eqref{s3.101} implies
\begin{eqnarray*}
&&\E \bigl(\bigl\llvert U_n^4(t,x)\bigr\rrvert
^p1_{L_n(t)} \bigr)
\\
&&\quad\le Cn^{\trup{3p}{2}}2^{n{\trup{p}{2}}} \E
\biggl(\int_0^t \mathrm{d}s\bigl \| G(t-s,x-\ast)
\bigl[B(X_n)-B\bigl(X_n^-\bigr)\bigr](s,
\ast)1_{L_n(s)} \bigr \|_{\mathcal{H}}^2 \biggr)^{{\trup{p}{2}}}.
\end{eqnarray*}
Then, the Lipschitz continuity of $B$ and \eqref{s4.19} yield
\begin{eqnarray*}
\E \bigl(\bigl\llvert U_n^4(t,x)\bigr\rrvert
^p1_{L_n(t)} \bigr) &\le& Cn^{\trup{3p}{2}}2^{n{\trup{p}{2}}}
\int_0^t \mathrm{d}s \Bigl[ \sup
_{y\in\IR^3}\E \bigl( \bigl |X_n(s,y)-X_n^-(s,y)
\bigr |^p 1_{L_n(s)} \bigr) \Bigr]
\\
&\le& C n^{3p}2^{-np [\trup{(3-\beta)}{2}-\trup{1}{2} ]}.
\end{eqnarray*}
Since $\beta\in\,]0,2[$, this implies \eqref{s3.81} for $i=4$.

%%i=8
The arguments based on Burkholder's and H\"older's inequalities,
already applied many times, give\vspace*{-2pt}
\begin{eqnarray*}
\E \bigl(\bigl\llvert U_n^8(t,x)\bigr\rrvert
^p1_{L_n(t)} \bigr) &\le& C\int_0^t
\mathrm{d}s \sup_{y\in\IR^3} \E \bigl( \bigl |X_n^-(s,y)-X_n(s,y)
\bigr |^{p} 1_{L_n(s)} \bigr)
\\
&\le& C n^{\trup{3p}{2}}2^{-np\trup{(3-\beta)}{2}},
\end{eqnarray*}
where, in the last inequality we have used \eqref{s4.19}. Thus, \eqref
{s3.81} holds for $i=8$.

%%i=6

Let us now consider the case $i=6$. Define\vspace*{-2pt}
\begin{eqnarray*}
U_n^{6,1}(t,x) &=&\int_0^t
\int_{\R^3} \bigl\{\pi_n \bigl[
\tau_n \bigl[G(t-\cdot ,x-\ast)B\bigl(X^-(\cdot,\ast)\bigr) \bigr]
\\
&&\phantom{\int_0^t
\int_{\R^3} \bigl\{\pi_n \bigl[}{}  -G(t-\cdot,x-\ast)\tau_n \bigl[B\bigl(X^-(\cdot,
\ast )\bigr) \bigr] \bigr](s,y) \bigr\} M(\mathrm{d}s,\mathrm{d}y),
\\
U_n^{6,2}(t,x) &=& \int_0^t
\int_{\R^3} \pi_n \bigl[G(t-\cdot,x-\ast)
\tau_n \bigl[B\bigl(X^-(\cdot,\ast)\bigr) \bigr]
\\[-1pt]
&&\phantom{\int_0^t
\int_{\R^3} \pi_n \bigl[}{}   - G(t-\cdot,x-\ast)B\bigl(X^-(\cdot,\ast)\bigr) \bigr](s,y)M(
\mathrm{d}s,\mathrm{d}y),
\\[-1pt]
U_n^{6,3}(t,x) & =& \int_0^t
\int_{\R^3} \bigl\{ \pi_n \bigl[G(t-\cdot,x-\ast
)B\bigl(X^-(\cdot,\ast)\bigr) \bigr]
\\[-1pt]
&&\phantom{\int_0^t
\int_{\R^3} \bigl\{}{}   -G(t-s,x-y)B\bigl(X^-(s,y)\bigr) \bigr\} M(\mathrm{d}s,
\mathrm{d}y).
\end{eqnarray*}
Clearly,\vspace*{-2pt}
\[
U_n^6(t,x)=U_n^{6,1}(t,x)+U_n^{6,2}(t,x)+U_n^{6,3}(t,x).
\]

To facilitate the analysis, we write $U_n^{6,1}(t,x)$ more explicitly,
as follows\vspace*{-2pt}
%
%e2.88 #&#
\begin{eqnarray}
\label{s3.82} U_n^{6,1} (t,x)&=& \int
_0^t\int_{\R^3} \big\{
\pi_n \bigl[G(t-\cdot ,x-\ast)B\bigl(X^-(\cdot,\ast)\bigr) \bigr]
\bigl(\bigl(s+2^{-n}\bigr)\wedge t,y \bigr)
\nonumber
\\[-8pt]
\\[-8pt]
&&\phantom{\int_0^t\int_{\R^3}
\bigl\{} {}   -\pi_n \bigl[G(t-\cdot,x-\ast)B \bigl(X^-
\bigl(\bigl(\cdot +2^{-n}\bigr)\wedge t,\ast \bigr) \bigr) \bigr](s,y)
\bigr\} M(\mathrm{d}s,\mathrm{d}y).
\nonumber
\end{eqnarray}
We are assuming that $t\ge t_0>0$. Hence, for $n$ big enough,
$t-2^{-n}>0$. Consider the first integral on the right-hand side of
\eqref{s3.82}. We have\vspace*{-2pt}
%
%e2.89 #&#
\begin{eqnarray}
\label{s3.83} && \int_0^t\int
_{\R^3} \pi_n \bigl[G(t-\cdot,x-\ast)B\bigl(X^-(
\cdot ,\ast)\bigr) \bigr] \bigl(\bigl(s+2^{-n}\bigr)\wedge t,y \bigr)M(
\mathrm{d}s,\mathrm{d}y)
\nonumber
\\[-1pt]
&&\quad =\int_0^{t-2^{-n}}\int_{\R^3}
\pi_n \bigl[G(t-\cdot,x-\ast )B\bigl(X^-(\cdot,\ast)\bigr) \bigr]
\bigl(s+2^{-n},y \bigr)M(\mathrm{d}s,\mathrm{d}y)
\nonumber
\\[-1pt]
&&\qquad{} +\int_{t-2^{-n}}^t\int
_{\R^3} \pi_n \bigl[G(t-\cdot,x-\ast )B\bigl(X^-(
\cdot,\ast)\bigr) \bigr] (t,y )M(\mathrm{d}s,\mathrm{d}y)
\\
&&\quad =\int_0^{t-2^{-n}}\int_{\R^3}
\pi_n \bigl[G(t-\cdot,x-\ast )B\bigl(X^-(\cdot,\ast)\bigr) \bigr]
\bigl(s+2^{-n},y \bigr)M(\mathrm{d}s,\mathrm{d}y)
\nonumber
\\
&&\quad =\int_{2^{-n}}^{t}\int_{\R^3}
\pi_n \bigl[G(t-\cdot,x-\ast )B\bigl(X^-(s,\ast)\bigr)
\bigr](s,y)M(\mathrm{d}s,\mathrm{d}y).
\nonumber
\end{eqnarray}
Indeed, the integral on the domain $[t-2^{-n},t]$ vanishes, and we have
applied the change of variable
$s\mapsto s+2^{-n}$.

For the second integral on the right-hand side of \eqref{s3.82}, we
split the domain of integration of the $s$-variable into three disjoint
sets, as follows:
%
%e2.90 #&#
\begin{eqnarray}
\label{s3.84} &&\int_0^t \int
_{\R^3} \pi_n \bigl[G(t-\cdot,x-\ast)B \bigl(X^-
\bigl(\bigl(\cdot+2^{-n}\bigr)\wedge t,\ast \bigr) \bigr) \bigr](s,y) M(
\mathrm{d}s,\mathrm{d}y)
\nonumber
\\
&&\quad= \int_0^{2^{-n}}\int_{\R^3}
\pi_n \bigl[G(t-\cdot,x-\ast)B \bigl(X^- \bigl(
\cdot+2^{-n},\ast \bigr) \bigr) \bigr](s,y) M(\mathrm{d}s,\mathrm{d}y)
\nonumber
\\[-8pt]
\\[-8pt]
&&\qquad{} + \int_{2^{-n}}^{t-2^{-n}}\int
_{\R^3} \pi_n \bigl[G(t-\cdot,x-\ast )B \bigl(X^-
\bigl(\cdot+2^{-n},\ast \bigr) \bigr) \bigr](s,y) M(\mathrm{d}s,
\mathrm{d}y)
\nonumber
\\
&&\qquad{} + \int_{t-2^{-n}}^t\int
_{\R^3} \pi_n \bigl[G(t-\cdot,x-\ast)B \bigl(X^-
(t,\ast ) \bigr) \bigr](s,y) M(\mathrm{d}s,\mathrm{d}y).
\nonumber
\end{eqnarray}
Then \eqref{s3.83}, \eqref{s3.84} yield
\begin{eqnarray*}
U_n^{6,1}(t,x)&=&-\int_0^{2^{-n}}
\int_{\R^3} \pi_n \bigl[G(t-\cdot ,x-\ast)B
\bigl(X^- \bigl(\cdot+2^{-n},\ast \bigr) \bigr) \bigr](s,y) M(
\mathrm{d}s,\mathrm{d}y)
\\
&&{} + \int_{2^{-n}}^{t-2^{-n}} \int_{\R^3}
\pi_n \bigl[G(t-\cdot ,x-\ast) \bigl(B\bigl(X^-(\cdot,\ast)\bigr)
\\
&&\phantom{+ \int_{2^{-n}}^{t-2^{-n}} \int_{\R^3}
\pi_n \bigl[G(t-\cdot ,x-\ast) \bigl(}{}-B
\bigl(X^-\bigl(\cdot+2^{-n},\ast\bigr)\bigr) \bigr) \bigr] (s,y)M(
\mathrm{d}s,\mathrm{d}y)
\\
&&{} -\int_{t-2^{-n}}^t\int_{\R^3}
\pi_n \bigl[G(t-\cdot,x-\ast)B \bigl(X^- (t,\ast ) \bigr)
\bigr](s,y)M(\mathrm{d}s,\mathrm{d}y)
\\
&&{} +\int_{t-2^{-n}}^t \int_{\R^3}
\pi_n \bigl[G(t-\cdot,x-\ast )B \bigl(X^- (\cdot,\ast ) \bigr)
\bigr](s,y)M(\mathrm{d}s,\mathrm{d}y).
\end{eqnarray*}
From this, we see that
$\E (\llvert  U_n^{6,1}(t,x)\rrvert ^p 1_{L_n(t)}
)\le C \sum_{i=1}^4V_n^{6,i}(t,x)$,
where
\begin{eqnarray*}
V_n^{6,1}(t,x)&=&\E \biggl(\biggl\llvert \int
_0^{t-2^{-n}}\int_{\IR
^3}
\pi_n \bigl[G(t-\cdot,x-\ast)
\\
&&\phantom{\E\bigl(\vert\int
_0^{t-2^{-n}}\int_{\IR
^3}} {}   \times \bigl\{B\bigl(X^-(
\cdot,\ast)\bigr)-B\bigl(X^-\bigl(\cdot +2^{-n},\ast\bigr)\bigr) \bigr\}
\bigr](s,y)M(\mathrm{d}s,\mathrm{d}y)\biggr\rrvert 1_{L_n(t)}
\biggr)^p,\\
V_n^{6,2}(t,x)&=&\E \biggl(\biggl\llvert \int
_0^{2^{-n}}\int_{\IR^3}\pi_n \bigl[G(t-\cdot,x-\ast)B\bigl(X^-(\cdot,\ast)\bigr) \bigr](s,y) M(
\mathrm{d}s,\mathrm{d}y)\biggr\rrvert \biggr)^p,
\\
V_n^{6,3}(t,x)&=&\E \biggl(\biggl\llvert \int
_{t-2^{-n}}^t\pi_n \bigl[G(t-\cdot,x-
\ast)B\bigl(X^-(t,\ast)\bigr) \bigr](s,y) M(\mathrm{d}s,\mathrm{d}y)\biggr\rrvert
\biggr)^p,
\\
V_n^{6,4}(t,x)&=&\E \biggl(\biggl\llvert \int
_{t-2^{-n}}^t\pi_n \bigl[G(t-\cdot,x-
\ast)B\bigl(X^-(\cdot,\ast)\bigr) \bigr](s,y) M(\mathrm{d}s,\mathrm{d}y)\biggr
\rrvert \biggr)^p.
\end{eqnarray*}
By Burkholder's and H\"older's inequalities, we have
\begin{eqnarray*}
V_n^{6,1}(t,x)&\le& C \int_0^t
\mathrm{d}s \sup_{y\in\IR^3} \E \bigl(\bigl\llvert X^-(s,y)-X^-
\bigl(s+2^{-n},y\bigr)\bigr\rrvert ^p1_{L_n(t)}
\bigr)
\\
&=&C \int_0^t \mathrm{d}s \sup
_{y\in\IR^3} \E \bigl(\bigl\llvert X(s,s_n,y)-X^-
\bigl(s+2^{-n},s_n+2^{-n},y\bigr)\bigr\rrvert
^p1_{L_n(t)} \bigr) \le C 2^{-np\rho},
\end{eqnarray*}
with $\rho\in\, ]0,\frac{2-\beta}{2} [$. Indeed, the last
inequality is obtained by using the triangular inequality along with
\eqref{s4.4} and \eqref{s3.43}.

For $s\in[0,2^{-n}]$, $X^-(s,y)=X(s,s_n,y)=0$. Therefore,
\[
V_n^{6,2}(t,x)\le C \biggl(\int_0^{2^{-n}}
\mathrm{d}s \int_{\R^3} \mu(\mathrm{d}\xi ) \bigl |\mathcal{F}G
(t-s,\ast) (\xi)\bigr |^2 \biggr)^{\trup{p}{2}} \le C 2^{-\trup{np}{2}},
\]
where in the last inequality we have used the property
\[
\int_{\IR^3}\mu(\mathrm{d}\xi)\bigl \vert\mathcal{F}\bigl(G(r,\ast)
\bigr) (\xi)\bigr \vert ^2=Cr^{2-\beta}.
\]
In a rather similar way,
\begin{eqnarray*}
&& V_n^{6,3}(t,x)+V_n^{6,4}(t,x)
\\
&&\quad \le C \Bigl(1+\sup_{(t,x)\in[0,T]\times\IR^3}\E\bigl\llvert
X(t,t_n,x)\bigr\rrvert ^p \Bigr) \biggl(\int
_0^{2^{-n}}\mathrm{d}s \int_{\IR^3}
\mu(\mathrm{d}\xi)\bigl \vert\mathcal {F}\bigl(G(s,\ast)\bigr) (\xi)\bigr \vert^2
\biggr)^{\trup{p}{2}}
\\
&&\quad \le C 2^{-np\trup{(3-\beta)}{2}}.
\end{eqnarray*}

Thus, we have established the convergence
%
%e2.91 #&#
\begin{equation}
\label{s3.85} \lim_{n\to\infty} \sup_{(t,x)\in[0,T]\times K(t)} \E
\bigl(\bigl\llvert U_n^{6,1}(t,x)\bigr\rrvert
^p1_{L_n(t)} \bigr)=0.
\end{equation}
Next, we consider the term $U_n^{6,2}(t,x)$. As usually for these type
of terms, we apply Burkholder's and then H\"older's inequalities, along
with the contraction property of the projection\vadjust{\goodbreak} $\pi_n$. This yields,
\begin{eqnarray*}
&& \E \bigl(\bigl\llvert U_n^{6,2}(t,x)\bigr\rrvert
^p 1_{L_n(t)} \bigr)
\\
&&\quad =\E \biggl(\biggl\llvert \int_0^t\int
_{\R^3}\pi_n \bigl[G(t-\cdot ,x-\ast) \bigl\{B
\bigl(X^-\bigl( \bigl(\cdot+2^{-n}\bigr)\wedge t,\ast\bigr)\bigr) - B
\bigl(X^-(\cdot,\ast )\bigr) \bigr\} \bigr](s,y)
\\
&&\phantom{\quad =\E \biggl(\biggl\llvert}{}   \times M(\mathrm{d}s,\mathrm{d}y)\biggr\rrvert
^p1_{L_n(t)} \biggr)
\\
&&\quad \le C \int_0^t \sup
_{x\in\IR^3}\E \bigl(\bigl\llvert X\bigl(\bigl(s+2^{-n}
\bigr)\wedge t,\bigl(s_n+2^{-n}\bigr)\wedge t,x
\bigr)-X(s,s_n,x)\bigr\rrvert \bigr)^p.
\end{eqnarray*}
Equation \eqref{s3.7} is a particular case of equation \eqref{s3.6}.
Therefore, Proposition~\ref{pss3.2.1} also holds with $X_n$ replaced
by $X$. Then,
by virtue of \eqref{s4.4} and \eqref{s3.43}, this is bounded up to a
constant by $2^{-np\trup{(3-\beta)}{2}}+2^{-np\rho}$, with $\rho\in\,
]0,\frac{2-\beta}{2} [$.
Consequently,
%
%e2.92 #&#
\begin{equation}
\label{s3.86} \lim_{n\to\infty} \sup_{(t,x)\in[0,T]\times\R^3} \E
\bigl(\bigl\llvert U_n^{6,2}(t,x)\bigr\rrvert
^p1_{L_n(t)} \bigr)=0.
\end{equation}

%%%%%%%%
%%%%%%%% U_n^{6,3}
%%%%%%%%

For $U_n^{6,3}(t,x)$, after having applied Burkholder's inequatily we have
\[
\E \bigl( \bigl | U_n^{6,3}(t,x) \bigr |^p
1_{L_n(t)} \bigr)\le C\E \bigl(\bigl\llVert (\pi_n-I_{\mathcal{H}_t}
) \bigl[G(t-\cdot,x-\ast)B\bigl(X^-(\cdot,\ast)\bigr) \bigr] 1_{L_n(\cdot)}
\bigr\rrVert _{\mathcal{H}_t}^p \bigr).
\]
We want to prove that the right-hand side of this inequality tends to
zero as $n\to\infty$, uniformly in $(t,x)\in[t_0,T]\times K(t)$. For
this, we will use a similar approach as in \cite{milletss2}, pages 906--909.

Set
\[
\tilde Z_n(t,x)= \bigl\llVert (\pi_n-I_{\mathcal{H}_t}
) \bigl[G(t-\cdot,x-\ast)B\bigl(X^-(\cdot,\ast)\bigr) \bigr] 1_{L_n(\cdot)}
\bigr\rrVert _{\mathcal{H}_t}.
\]
Since $\pi_n$ is a projection on the Hilbert space $\mathcal{H}_t$,
the sequence $\{\tilde Z_n(t,x), n\ge1\}$ decreases to zero as $n\to
\infty$.
Assume that
%
%e2.93 #&#
\begin{equation}
\label{s3.87} \E \Bigl(\sup_n\bigl\llVert G(t-\cdot,x-
\ast) B\bigl(X^-(\cdot,\ast)\bigr) 1_{L_n(\cdot)} \bigr\rrVert
_{\mathcal{H}_t}^p \Bigr) <\infty.
\end{equation}
Remember that $X^-(s,y)$ stands for $X(s,s_n,y)$, defined in \eqref
{s3.8.3}, and therefore it depends on~$n$.
Then, by bounded convergence, this would imply $\lim_{n\to\infty}\E
(\tilde Z_n(t,x) )^p=0$.
Set $Z_n(t,x)=\E (\tilde Z_n(t,x) )^p$. Proceeding as in
the proof of Lemmas~\ref{lss3.1.1}, \ref{lss3.2.1}, we can check that\vspace*{-1pt}
\[
\bigl\llvert \bigl(Z_n(t,x)\bigr)^{\trup{1}{p}}-
\bigl(Z_n(\bar t,\bar x)\bigr)^{\trup
{1}{p}}\bigr\rrvert \le C \bigl(
\vert t-\bar t\vert+\vert x-\bar x\vert \bigr)^\rho,
\]
with $\rho\in\, ]0,\frac{2-\beta}{2} [$.

Hence, $(Z_n)_n$ is a sequence of monotonically decreasing continuous
functions defined on
$[0,T]\times\IR^3$ which converges pointwise to zero. Appealing to
Dini's theorem, we obtain
%
%e2.94 #&#
\begin{equation}
\label{s3.88} \lim_{n\to\infty}\sup_{(t,x)\in[t_0,T]\times K(t)}\E
\bigl(\tilde Z_n(t,x) \bigr)^p=0.
\end{equation}
This yields the expected result on $U_n^{6,3}$.\vadjust{\goodbreak}

It remains to prove \eqref{s3.87}. We will sketch the arguments,
leaving the details to the reader. As usually, we write
$\Vert G(t-\cdot,x-\ast) B(X^-(\cdot,\ast)) 1_{L_n(\cdot)} \Vert
_{\mathcal{H}_t}$ using the identity \eqref{fundamental}. By applying
H\"older's inequality with respect to the
measure on $[0,t]\times\R^3\times\R^3$ with density
$G(t-s,x-y)G(t-s,x-z)|y-z|^{-\beta} \,\mathrm{d}s \,\mathrm{d}y\,\mathrm{d}z$, and using the linear
growth of the function $B$, we obtain as upper bound for the left-hand
side of \eqref{s3.87}
%
%e2.95 #&#
\begin{equation}
\label{s3.880} C \Bigl[1+ \sup_{t,x\in[0,T]\times\R^3} E \Bigl(\sup
_{n}\bigl\llvert X(t,t_n,x)\bigr\rrvert
^p \Bigr) \Bigr].
\end{equation}
Looking back to the definition of $X(t,t_n,x)$, we see that for the
second and third terms in \eqref{s3.8.3}, the supremum in $n$ can be
easily handled, since they are defined pathwise.
For the stochastic integral term, we consider the discrete martingale
\[
\biggl\{\int_0^{t_n} \int_{\R^3}
G(s_0-s,x-y) (A+B) \bigl(X(s,y)\bigr) M(\mathrm{d}s,\mathrm{d}y),
\tf_{t_n}, n\in\mathbb{N} \biggr\},
\]
where $s_0\in\,]0,T]$ is fixed. By applying first Doob's maximal
inequality and then Burkholder's inequality, we have
\begin{eqnarray*}
&& E \biggl(\sup_{n}\biggl\llvert \int_0^{t_n}
\int_{\R^3} G(s_0-s,x-y) (A+B) \bigl(X(s,y)\bigr)
M(\mathrm{d}s,\mathrm{d}y)\biggr\rrvert ^p \biggr)
\\
&&\quad \le C E \biggl(\biggl\llvert \int_0^t
\int_{\R^3} G(s_0-s,x-y) (A+B) \bigl(X(s,y)\bigr)
M(\mathrm{d}s,\mathrm{d}y)\biggr\rrvert ^p \biggr)
\\
&&\quad \le C E \bigl(\bigl \Vert G(s_0-\cdot, x-\ast) (A+B) \bigl(X(
\cdot,\ast)\bigr)\bigr \Vert ^{\trup{p}{2}}_{\hac_t} \bigr).
\end{eqnarray*}
Finally, we take $s_0:=s$. Using the property $\sup_{(t,x)\in
[0,T]\times\R^3}E (\vert X(t,x)\vert^p )$, we obtain that
the expression \eqref{s3.880} is finite.

Owing to \eqref{s3.85}, \eqref{s3.86} and \eqref{s3.88}, we have
%
%e2.96 #&#
\begin{equation}
\label{s3.90} \lim_{n\to\infty}\sup_{(t,x)\in[0,T]\times K(t)} \E
\bigl(\bigl |U_n^6(t,x)\bigr |^p1_{L_n(t)}
\bigr)=0.
\end{equation}

%%%%%%%Conclusion

In order to conclude the proof, let us consider the estimates \eqref
{s3.76}, \eqref{s3.77}, \eqref{s3.78}, \eqref{s3.80}, along with
\eqref{s3.81}. We see that
\[
\E \bigl(\bigl\llvert X_n(t,x)-X(t,x)\bigr\rrvert
^p1_{L_n(t)} \bigr)\le C_1 \theta_n +
C_2 \int_0^t \mathrm{d}s \Bigl[
\sup_{x\in K(s) }\E \bigl(\bigl\llvert X_n(s,x)-X(s,x)
\bigr\rrvert ^p 1_{L_n(s)} \bigr) \Bigr],
\]
where $(\theta_n,n\ge1)$ is a sequence of real numbers which
converges to zero as $n\to\infty$. Applying Gronwall's lemma, we
finish the proof of the theorem.
\qed

%%%%%%%
%%%%%%%
%%%%%%% Subsection 3.4 Proof of the Theorem
%%%%%%%

%s2.4 #&#
\subsection{Proof of Theorem \texorpdfstring{\protect\ref{ts3.1}}{2.2}}
\label{ss3.4}

Fix $t_0>0$ and a compact set $K\subset\IR^3$.
Let $Y_n(t,x):= X_n(t,x)-X(t,x)$ and $B_n(t):= L_n(t)$, $n\ge1$,
$(t,x)\in[t_0,T]\times K$, $p\in[1,\infty[$. From Theorems~\ref
{ts3.2} and~\ref{ts3.3}, we see that the conditions (P1) and (P2) of
Lemma~\ref{ls5.2} are satisfied with $\delta=p\rho-4$, for any $\rho
\in\, ]0,\frac{2-\beta}{2} [$.
We infer that
%
%e2.97 #&#
\begin{equation}
\label{s5.2} \lim_{n\to\infty}\E \bigl(\llVert X_n-X
\rrVert ^p_{\rho
,t_0,K}1_{L_n(t)} \bigr)=0,
\end{equation}
for any $p\in[1,\infty[$ and $\rho\in\, ]0,\frac{2-\beta
}{2} [$.

Fix $\epsilon>0$. Since $\lim_{n\to\infty}\Pb(L_n(t)^c)=0$, there
exists $N_0\in\IN$ such that for all $n\ge N_0$,
$\Pb(L_n(t)^c)<\epsilon$.  Then,
for any $\lambda>0$ and $n\ge N_0$,
\begin{eqnarray*}
\Pb \bigl(\|X_n-X\|_{\rho,t_0,K}>\lambda \bigr)&\le&\epsilon+ \Pb \bigl( \bigl(
\|X_n-X\|_{\rho,t_0K}>\lambda \bigr)\cap L_n(t) \bigr)
\\
&\le& \epsilon+ \lambda^{-p} \E \bigl(\|X_n-X
\|_{\rho
,t_0K}^p1_{L_n(t)} \bigr).
\end{eqnarray*}
Since $\epsilon>0$ is arbitrary, this finishes the proof of the theorem.
\qed

%%%%%%%%%
%%%%%%%%%
%%%%%%%%% Support Theorem
%%%%%%%%%
%%%%%%%%%

%s3 #&#
\section{Support theorem}
\label{sm}

This section is devoted to the characterization of the topological
support of the law of the random field solution to the stochastic wave
equation \eqref{s1.6}. As has been explained in the \hyperref[s1]{Introduction}, this
is a corollary of Theorem~\ref{ts3.1}.
%
%th3.1 #&#
\begin{theorem}
\label{tsm.1}
Assume that the functions $\sigma$ and $b$ are Lipschitz continuous.
Fix $t_0\in\,]0,T[$ and a compact set $K\subset\IR^3$. Let $u=\{
u(t,x), (t,x)\in[t_0,T]\times K\}$ be the random field solution to
\eqref{s1.6}. Fix
$\rho\in\, ]0,\frac{2-\beta}{2} [$. Then the topological
support of the law of $u$ in the space $\mathcal{C}^\rho
([t_0,T]\times K)$ is the closure in
$\mathcal{C}^\rho([t_0,T]\times K)$ of the set of functions $\{\Phi
^h, h\in\mathcal{H}_T\}$, where $\{\Phi^h(t,x), (t,x)\in
[t_0,T]\times K\}$ is the solution of \eqref{sm.h}.
\end{theorem}

Let $\{w^n, n\ge1\}$ be the sequence of $\mathcal{H}_T$-valued random
variables defined in \eqref{s3.3}. For any $h\in\mathcal{H}_T$, we
consider the sequence of transformations of $\Omega$ defined in \eqref
{sm.1}. As has been pointed out in Section~\ref{s1},
$P\circ(T_n^h)^{-1}\ll P$.

Notice also that the process $v_n(t,x):=(u\circ T_n^h)(t,x)$, $(t,x)\in
[t_0,T]\times\IR^3$, satisfies the equation
%
%e3.1 #&#
\begin{eqnarray}
\label{sm.2} v_n(t,x)&=&\int_0^t
\int_{\IR^3}G(t-s,x-y)\sigma \bigl(v_n(s,y)\bigr)M(
\mathrm{d}s,\mathrm{d}y)
\nonumber
\\[-8pt]
\\[-8pt]
&&{}+ \bigl\langle G(t-\cdot,x-\ast)\sigma\bigl(v_n(\cdot,\ast )
\bigr),h-w^n \bigr\rangle_{\mathcal{H}_t} +\int_0^t
\mathrm{d}s \bigl[G(t-s,\cdot)\star b\bigl(v_n(s,\cdot)\bigr)
\bigr](x).\quad\ \
\nonumber
\end{eqnarray}

\begin{pf*}{Proof of Theorem~\ref{tsm.1}}
According to the method developed in \cite{millet-ss94a} (see also
\cite{Ba-Mi-SS} and Section~\ref{s1} for a summary), the theorem will
be a consequence of the following convergences:
%
%e3.2 #&#
%e3.3 #&#
\begin{eqnarray}
\label{sm.3}\lim_{n\to\infty}\Pb \bigl\{\bigl\llVert u-
\Phi^{w^n}\bigr\rrVert _{\rho,t_0,K}>\eta \bigr\}&=&0,
\\
\label{sm.4}\lim_{n\to\infty}\Pb \bigl\{\bigl\llVert u\circ
T_n^h-\Phi^h\bigr\rrVert _{\rho,t_0,K}>
\eta \bigr\}&=&0,
\end{eqnarray}
where $\eta$ is an arbitrary positive real number.\vadjust{\goodbreak}

This follows from the general approximation result developed in Section~\ref{s3}.
Indeed, consider equations \eqref{s3.7} and \eqref{s3.6}
with the choice of coefficients $A=D=0$, $B=\sigma$. Then the
processes $X$ and $X_n$ coincide with $u$ and $\Phi^{w^n}$, respectively.
Hence, the convergence \eqref{sm.3} follows from Theorem~\ref{ts3.1}.
Next, we consider again equations \eqref{s3.7} and \eqref{s3.6} with
a new choice of coefficients: $A=D=\sigma$, $B=-\sigma$. In this
case, the processes $X$ and $X_n$ are equal to $\Phi^h$ and
$v_n:=u\circ T_n^h$, respectively. Thus, Theorem~\ref{ts3.1} yields
\eqref{sm.4}.\vspace*{1pt}
\end{pf*}

%%%%%%%%%%%
%%%%%%%%%% Resultats auxiliars
%%%%%%%%%%%
%%%%%%%%%%

%s4 #&#
\section{Auxiliary results}
\label{s4}

The most difficult part in the proof of Theorem~\ref{ts3.1} consists
of establishing \eqref{s3.15}. In particular, handling the
contribution of the pathwise integral (with respect to $w^n$) requires
a careful analysis of the discrepancy between this integral and the
stochastic integral with respect to $M$. This section gathers several
technical results that have been applied in the analysis of such
questions in the preceding Section~\ref{s3}.

The first statement in the next lemma provides a measure of the
discrepancy between the processes $X(t,x)$ and $X(t,t_n,x)$ defined in
\eqref{s3.7}, \eqref{s3.8.3}, respectively.\vspace*{1pt}
%
%le4.1 #&#
\begin{lemma}
\label{ls4.1}
Suppose that Hypothesis~\textup{\ref{HypB}} is satisfied. Then for any $p\in[1,\infty
)$ and every integer $n\ge1$,\vspace*{1pt}
%
%e4.1 #&#
\begin{equation}
\label{s4.4} \sup_{(t,x)\in[0,T]\times\R^3} \bigl \|X(t,x)-X(t,t_n,x)
\bigr \|_p\le C2^{-n{\trup{(3-\beta)}{2}}}
\end{equation}
and\vspace*{1pt}
%
%e4.2 #&#
\begin{equation}
\label{s4.5} \sup_{n\ge1}\sup_{(t,x)\in[0,T]\times\R^3}
\bigl \|X(t,t_n,x)\bigr \|_p\le C<\infty,
\end{equation}
where $C$ is a positive constant not depending on $n$.
\end{lemma}

\begin{pf} Fix $p\in[2,\infty[$. From equations \eqref{s3.7},
\eqref{s3.8.3}, we obtain\vspace*{1pt}
\[
\bigl \|X(t,x)-X(t,t_n,x)\bigr \|_p^p\le C
\bigl(V_1(t,x)+ V_2(t,x)+V_3(t,x) \bigr),
\]
where\vspace*{1pt}
\begin{eqnarray*}
V_1(t,x) &:=& \biggl\llVert \int_{t_n}^t
\int_{\R^3} G(t-s,x-y) (A+B) \bigl(X(s,y)\bigr) M(\mathrm{d}s,
\mathrm{d}y)\biggr\rrVert _p^p,
\\
V_2(t,x) &:=&\bigl\llVert G(t-\cdot,x-\ast)D\bigl(X(\cdot,\ast )
\bigr)1_{[t_n,t]}(\cdot),h\rangle_{\mathcal{H}_t}\bigr\rrVert
_p^p,
\\
V_3(t,x) &:=&\biggl\llVert \int_{t_n}^t
G(t-s,\cdot)\star b\bigl(X(s,\cdot)\bigr) (x) \,\mathrm{d}s\biggr\rrVert
_p^p.
\end{eqnarray*}
Applying first Burholder's and then H\"older's inequalities, we obtain
\begin{eqnarray*}
V_1(t,x) &\le& C \biggl(\int_{t_n}^t
\mathrm{d}s \int_{\R^3}\mu(\mathrm{d}\xi) \bigl |\mathcal {F}G(t-s)
(\xi)\bigr |^2 \biggr)^{\trup{p}{2}} \sup_{(t,x)\in[0,T]\times\R^3}\E
\bigl(\bigl |(A+B) \bigl(X(t,x)\bigr)\bigr  |^{p} \bigr)
\\
&\le& C \biggl(\int_{t_n}^t \mathrm{d}s \int
_{\R^3}\mu(\mathrm{d}\xi) \bigl |\mathcal {F}G(t-s) (
\xi)\bigr |^2 \biggr)^{\trup{p}{2}} \Bigl(1 + \sup_{(t,x)\in[0,T]\times\R^3}
\E \bigl( \bigl|X\bigl((t,x)\bigr|^{p}\bigr) \Bigr).
\end{eqnarray*}
Applying the inequality \eqref{s3.59} along with \eqref{s5.21}, imply
\[
V_1(t,x)\le C 2^{-np{\trup{(3-\beta)}{2}}}.
\]

For the study of $V_2$, we apply first Cauchy--Schwarz inequality and
then H\"older's inequality. We obtain
\[
V_2(t,x) \le\bigl \|h1_{[t_n,t]}(\cdot)\bigr \|_{\mathcal{H}_t}^{{\trup{p}{2}}}
\E \biggl(\int_0^t \mathrm{d}s \bigl \|G(t-s,x-
\ast)D\bigl(X(s,\ast)\bigr) 1_{[t_n,t]}(s)\bigr \|_{\mathcal{H}}^2
\biggr)^{{\trup{p}{2}}}.
\]
Hence, similarly as for $V_1$ we have
\[
V_2(t,x)\le C 2^{-np{\trup{(3-\beta)}{2}}}.
\]

By applying H\"older's inequality, we get
\begin{eqnarray*}
V_3(t,x) &\le& \biggl(\int_{t_n}^t
\mathrm{d}s \int_{\R^3} G(t-s,x-\mathrm{d}y)
\biggr)^{p-1} \int_{t_n}^t \mathrm{d}s
\int_{\R^3} G(t-s,x-\mathrm{d}y) \E\bigl(\bigl |b\bigl(X(s,y)
\bigr)\bigr |^p\bigr)
\\
&\le& C \biggl(\int_{t_n}^t \mathrm{d}s \int
_{\R^3} G(t-s,x-\mathrm{d}y) \biggr)^p \Bigl(1+
\sup_{(t,x)\in[0,T]\times\R^3} \E\bigl(\bigl |X(s,y)\bigr |^p\bigr) \Bigr)
\\
&\le& C 2^{-2np}.
\end{eqnarray*}
The condition $\beta\in\,]0,2[$ implies $2^{-2np}< 2^{-np{\trup
{(3-\beta)}{2}}}$. Thus from the estimates on $V_i(t,x)$, $i=1,2,3$
(which hold uniformly on $(t,x)\in[0,T]\times\IR^3$) we obtain
\eqref{s4.4}.

Finally, \eqref{s4.5} is a consequence of the triangular inequality,
\eqref{s4.4} and \eqref{s5.21}.
\end{pf}

%%%%%End lemma{l4.1}

The next result states an analogue of Lemma~\ref{ls4.1} for the
stochastic processes $X_n$, $X_n^-$ defined in \eqref{s3.6}, \eqref
{s3.8.2}, respectively, this time including a localization by $L_n$.

%le4.2 #&#
\begin{lemma}
\label{ls4.2}
We assume Hypothesis~\textup{\ref{HypB}}.
Then for any $p\in[2,\infty)$ and $t\in[0,T]$,
%
%e4.3 #&#
\begin{eqnarray}
\label{s4.11} &&\sup_{(s,y)\in[0,t]\times\R^3}\E \bigl(\bigl |X_n(s,y)-X_n^-(s,y)\bigr |^p1_{L_n(s)}
\bigr)
\nonumber
\\[-8pt]
\\[-8pt]
&&\quad \le Cn^{\trup{3p}{2}} 2^{-np{\trup{(3-\beta)}{2}}} \Bigl[1+ \sup
_{(s,y)\in[0,t]\times\R^3} \E\bigl(\bigl |X_n(s,y)\bigr |^p1_{L_n(s)}
\bigr) \Bigr].
\nonumber
\end{eqnarray}
\end{lemma}

\begin{pf}
Fix $p\in[2,\infty[$ and consider the decomposition
%
%e4.4 #&#
\begin{equation}
\label{s4.12} \E \bigl(\bigl |X_n(t,x) - X_n^-(t,x)\bigr |^p
1_{L_n(t)} \bigr)\le C\sum_{i=1}^4
T_{n,i}^k (t,x),
\end{equation}
where
\begin{eqnarray*}
T_{n,1}(t,x) &=&\E \biggl(\biggl\llvert \int_{t_n}^t
\int_{\R^3} G(t-s,x-y) A\bigl( X_n(s,y)\bigr)M(
\mathrm{d}s,\mathrm{d}y) \biggr\rrvert ^p 1_{L_n(t)} \biggr),
\\
T_{n,2}(t,x) &=&\E \bigl(\bigl\llvert \bigl\langle G(t-\cdot,x-\ast) B
\bigl( X_n(\cdot,\ast)\bigr)1_{[t_n,t]}(\cdot),w^n
\bigr\rangle_{\mathcal{H}_t} \bigr\rrvert ^p 1_{L_n(t)} \bigr),
\\
T_{n,3}(t,x) &=&\E \bigl(\bigl\llvert \bigl\langle G(t-\cdot,x-\ast) D
\bigl( X_n(\cdot,\ast)\bigr)1_{[t_n,t]}(\cdot),h\bigr
\rangle_{\mathcal{H}_t} \bigr\rrvert ^p 1_{L_n(t)} \bigr),
\\
T_{n,4}(t,x) &=&\E \biggl(\biggl\llvert \int_{t_n}^t
G(t-s,\cdot)\star b\bigl(X_n(s,\cdot)\bigr) (x)\, \mathrm{d}s \biggr
\rrvert ^p 1_{L_n(t)} \biggr).
\end{eqnarray*}

By the same arguments used for the analysis of $V_1(t,x)$ in the
preceding lemma, we obtain
%
%e4.5 #&#
\begin{equation}
\label{s4.13} T_{n,1}(t,x)\le C 2^{-np{\trup{(3-\beta)}{2}}}\times \Bigl[1+\sup
_{(s,y)\in[0,t]\times\R^3}\E \bigl(\bigl\llvert X_n(s,y) \bigr\rrvert
^p 1_{L_n(s)} \bigr) \Bigr].
\end{equation}

%T_{n,2}^k(t,x)

For $T_{n,2}(t,x)$, we first use Cauchy--Schwarz' inequality to obtain
\begin{eqnarray*}
&&T_{n,2}(t,x)
\\
&&\quad\le\E \bigl(\bigl\llvert \bigl \|w^n
1_{[t_n,t]} 1_{L_n(t)}\bigr \| _{\mathcal{H}_t} \bigl \|G(t-\cdot,x-\ast) B
\bigl( X_n(\cdot,\ast)\bigr)1_{[t_n,t]}(\cdot )1_{L_n(t)}
\bigr \|_{\mathcal{H}_t}\bigr\rrvert ^p \bigr).
\end{eqnarray*}
Appealing to \eqref{s3.14}, this yields
\[
T_{n,2}(t,x)\le C n^{\trup{3p}{2}} \E \biggl(\biggl\llvert \int
_{t_n}^t \mathrm{d}s \bigl \|G(t-s,x-\ast) B\bigl(
X_n(s,\ast)\bigr) (s)1_{L_n(s)}\bigr \|_{\mathcal{H}}^2
\biggr\rrvert ^{{\trup{p}{2}}} \biggr).
\]
We can now proceed as for the term $V_2((t,x)$ in the proof of Lemma~\ref{ls4.1}. We obtain
%
%e4.6 #&#
\begin{equation}
\label{s4.14} T_{n,2}(t,x)\le C n^{\trup{3p}{2}}2^{-np{\trup{(3-\beta)}{2}}}
\Bigl[1+\sup_{(s,y)\in[0,t]\times\R^3}\E \bigl(\bigl\llvert X_n(s,y)
\bigr\rrvert ^p 1_{L_n(s)} \bigr) \Bigr].
\end{equation}

%T_{n,3}^k(t,x)
The difference between the terms $T_{n,3}(t,x)$ and $T_{n,2}(t,x)$ is
that $w^n$ in the latter is replaced by $h$ in the former. Hence,
following similar arguments as for the study of $T_{n,2}(t,x)$, and
using that $\Vert h 1_{[t_n,t]} 1_{L_n(t)}\Vert_{\mathcal
{H}_T}<\infty$, we prove
%
%e4.7 #&#
\begin{equation}
\label{s4.15} T_{n,3}(t,x)\le C 2^{-np{\trup{(3-\beta)}{2}}}\times \Bigl[1+\sup
_{(s,y)\in[0,t]\times\R^3}\E \bigl(\bigl\llvert X_n(s,y) \bigr\rrvert
^p 1_{L_n(s)} \bigr) \Bigr].
\end{equation}

%T_{n,4}^k(t,x)
Finally, we notice the similitude between $T_{n,4}(t,x)$ and $V_3(t,x)$
in Lemma~\ref{ls4.1}. Proceeding as for the study of this term, we obtain
%
%e4.8 #&#
\begin{eqnarray}
\label{s4.16} T_{n,4}(t,x)&\le& C \biggl(\int_{t_n}^t
\mathrm{d}s \int_{\IR^3} G(t-s,x-\mathrm{d}y)
\biggr)^p \Bigl[1+\sup_{(s,y)\in[0,t]\times\R^3} \E \bigl(\bigl\llvert
X_n(s,y)\bigr\rrvert ^p 1_{L_n(s)} \bigr) \Bigr]
\nonumber
\\[-8pt]
\\[-8pt]
&\le& C 2^{-np{\trup{(3-\beta)}{2}}} \Bigl[1+\sup_{(s,y)\in[0,t]\times
\R^3}\E \bigl(\bigl
\llvert X_n(s,y)\bigr\rrvert ^p 1_{L_n(s)} \bigr)
\Bigr].
\nonumber
\end{eqnarray}
From \eqref{s4.12}--\eqref{s4.16} we obtain \eqref{s4.11}.
\end{pf}

%%%Lemma \ls4.3
%
%le4.3 #&#
\begin{lemma}
\label{ls4.3}
We assume Hypothesis~\textup{\ref{HypB}}. Then, for any $p\in[1,\infty)$, there
exists a finite constant $C$ such that
%
%e4.9 #&#
\begin{equation}
\label{s4.18} \sup_{n\ge1}\sup_{(t,x)\in[0,T]\times\R^3}\E
\bigl[\bigl(\bigl |X_n(t,x)\bigr |^p+\bigl |X_n^-(t,x)\bigr |^p
\bigr)1_{L_n(t)} \bigr]\le C.
\end{equation}
Moreover,
%
%e4.10 #&#
\begin{equation}
\label{s4.19} \sup_{(t,x)\in[0,T]\times\R^3} \bigl \|\bigl(X_n(t,x)-X_n^-(t,x)
\bigr)1_{L_n(t)}\bigr \| _p \le C n^{\trup{3}{2}} 2^{-n{\trup{(3-\beta)}{2}}}.
\end{equation}
\end{lemma}
\begin{pf}
For $0\le r\le t$, define
\begin{eqnarray*}
X_n(t,r;x) &=&\int_0^r
\int_{\R^3} G(t-s,x-y) A\bigl(X_n(s,y)\bigr) M(
\mathrm{d}s,\mathrm{d}y)
\\
&&{} +\bigl\langle G(t-\cdot,x-\ast)B\bigl(X_n(\cdot,\ast)
\bigr)1_{[0,r]}(\cdot ),w^n\bigr\rangle_{\mathcal{H}_t}
\\
&&{} +\bigl\langle G(t-\cdot,x-\ast)D\bigl(X_n(\cdot,\ast)
\bigr)1_{[0,r]}(\cdot ),h\bigr\rangle_{\mathcal{H}_t} +\int
_0^r G(t-s,\cdot)\star b\bigl(X_n(s,
\cdot)\bigr) (x) \,\mathrm{d}s.
\end{eqnarray*}

%%%%%
%%%%%
%%%%%
Fix $p\in[2,\infty[$ and consider the decomposition
\[
\E\bigl(\bigl |X_n(t,r;x)\bigr |^p1_{L_n(t)}\bigr)\le C \sum
_{i=1}^5 T_{n,i}(t,r;x),
\]
where
\begin{eqnarray*}
T_{n,1}(t,r;x) &=& \E \biggl(\biggl\llvert \int_0^r
\int_{\R^3} G(t-s,x-y)A\bigl(X_n(s,y)\bigr) M(
\mathrm{d}s,\mathrm{d}y)\biggr\rrvert ^p 1_{L_n(t)} \biggr),
\\
T_{n,2}(t,r;x) &=& \E \bigl(\bigl\llvert \bigl\langle G(t-\cdot,x-
\ast) B\bigl(X_n^-(\cdot,\ast)\bigr) 1_{[0,r]}(
\cdot),w^n\bigr\rangle_{\mathcal{H}_t} \bigr\rrvert ^p
1_{L_n(t)} \bigr),
\\
T_{n,3}(t,r;x) &=& \E \bigl(\bigl\llvert \bigl\langle G(t-\cdot,x-
\ast) \bigl[B\bigl(X_n(\cdot,\ast)\bigr)-B\bigl(X_n^-(
\cdot,\ast)\bigr)\bigr] 1_{[0,r]}(\cdot ),w^n\bigr
\rangle_{\mathcal{H}_t}\bigr\rrvert ^p1_{L_n(t)} \bigr),
\\
T_{n,4}(t,r;x) &=& \E \bigl(\bigl\llvert \bigl\langle G(t-\cdot,x-
\ast) D\bigl(X_n(\cdot,\ast)\bigr) 1_{[0,r]}(\cdot),h\bigr
\rangle_{\mathcal{H}_t}\bigr\rrvert ^p1_{L_n(t)} \bigr),
\\
T_{n,5}^k (t,r;x) &=& \E \biggl(\biggl\llvert \int
_0^rG(t-s,\cdot)\star b\bigl(X_n(s,
\cdot)\bigr) (x) \,\mathrm{d}s \biggr\rrvert ^p 1_{L_n(t)}
\biggr).
\end{eqnarray*}

%T_{n,1}

Similarly as for the term $V_1(t,x)$ in Lemma~\ref{ls4.1}, we have
%
%e4.11 #&#
\begin{eqnarray}
\label{s4.22} T_{n,1}(t,r;x) &\le& C \biggl(\int
_0^r \mathrm{d}s \int_{\IR^3}
\mu(\mathrm{d}\xi)\bigl \vert\mathcal{F} G(t-s) (\xi)\bigr \vert^2
\biggr)^{\trup{p}{2}-1}
\nonumber
\\
&&{} \times\int_0^r \mathrm{d}s \Bigl[1+
\sup_{(\hat{s},y)\in[0,s]\times\R
^3}\E \bigl(\bigl\llvert X_n(\hat{s},y)
\bigr\rrvert ^{p}1_{L_n(\hat
{s})} \bigr) \Bigr]
\nonumber
\\[-8pt]
\\[-8pt]
&&{}\times\biggl(\int
_{\IR^3}\mu(\mathrm{d}\xi)\bigl \vert\mathcal{F} G(t-s) (\xi)\bigr \vert
^2 \biggr)\nonumber
\\
& \le& C \int_0^r \mathrm{d}s \Bigl[1+ \sup
_{(\hat{s},y)\in[0,s]\times\R
^3}\E \bigl(\bigl\llvert X_n(\hat{s},y) \bigr
\rrvert ^{p}1_{L_n(\hat{s})} \bigr) \Bigr].
\nonumber
\end{eqnarray}

%T_{n,2}
Let $\tau_n$ and $\pi_n$ be as in the proof of Lemma~\ref{lss3.1.2}
(see \eqref{s3.34} and the successive lines). Since $X_n^-(s,y)$ is
$\mathcal{F}_{s_n}$-measurable, the definition of $w^n$ implies
\begin{eqnarray*}
&& T_{n,2}(t,r;x)
\\
&&\quad =\E \biggl(\biggl\llvert \int_0^t\int
_{\R^3} (\pi_n \circ\tau_n)
\bigl[G(t-\cdot,x-\ast)B\bigl(X_n^-(\cdot,\ast)\bigr)
\bigr](s,y)1_{L_n(t)} 1_{[0,r]}(\cdot)M(\mathrm{d}s,\mathrm{d}y)
\biggr\rrvert ^p \biggr).
\end{eqnarray*}
Then, applying Burkholder's inequality, using the boundedness of the
operator $\pi_n \circ\tau_n$, and similar arguments as for the term
$T_{n,1}(t,r;x)$ we obtain
%
%e4.12 #&#
\begin{equation}
\label{s4.24} T_{n,2}(t,r;x) \le C \int_0^r
\mathrm{d}s \Bigl[1+ \sup_{(\hat{s},y)\in[0,s]\times\R
^3}\E \bigl(\bigl\llvert
X_n^-(\hat{s},y) \bigr\rrvert ^{p}1_{L_n(\hat{s})}
\bigr) \Bigr].
\end{equation}

%T_{n,3}$
To study $T_{n,3}(t,r;x)$, we apply Cauchy--Schwarz and then H\"older's
inequality. This yields
\begin{eqnarray*}
&& T_{n,3}(t,r;x)
\\
&&\quad \le\E \bigl(\bigl\llvert \bigl\llVert w^n 1_{[0,r]}
1_{L_n(t)}\bigr\rrVert _{\mathcal{H}_t}^2\bigl\llVert G(t-
\cdot,x-\ast) \bigl[B( X_n)-B\bigl( X_n^-\bigr) \bigr](
\cdot,\ast) 1_{[0,r]}(\cdot) 1_{L_n(t)}\bigr\rrVert
_{\mathcal{H}_t}^2 \bigr\rrvert ^{{\trup{p}{2}}} \bigr)
\\
&&\quad \le C n^{\trup{3p}{2}}2^{n{\trup{p}{2}}} \E \biggl(\int
_0^t \mathrm{d}s \bigl\llVert G(t-s,x-\ast)
\bigl[B( X_n)-B\bigl(X_n^-\bigr) \bigr](s,\ast)
1_{[0,r]}(s)1_{L_n(s)} \bigr\rrVert _{\mathcal{H}}^2
\biggr)^{{\trup{p}{2}}}
\\
&&\quad \le C n^{\trup{3p}{2}} 2^{n{\trup{p}{2}}} \biggl(\int_0^r
\mathrm{d}s\int_{\R^3} \mu(\mathrm{d}\xi)\bigl \vert\mathcal
{F}G(t-s)\bigr \vert^2(\xi) \biggr)^{\trup{p}{2}-1}
\\
&&\qquad{} \times\int_0^r \mathrm{d}s \sup
_{(\hat{s},y)\in[0,s]\times\R^3}\E \bigl(\bigl\llvert X_n(
\hat{s},y)-X_n^-(\hat{s},u) \bigr\rrvert ^p
1_{L_n(\hat{s})} \bigr)
\\
&&\qquad{}\times \biggl(\int_{\R^3} \mu(\mathrm{d}\xi)
\bigl \vert\mathcal{F}G(t-s)\bigr \vert^2(\xi ) \biggr),
\end{eqnarray*}
where we have used \eqref{s3.101} and the Lipschitz continuity of the
function $B$. By applying \eqref{s4.11}, we obtain
\begin{eqnarray*}
T_{n,3}(t,r;x) &\le& C n^{3p} 2^{-np[{\trup{(3-\beta)}{2}}-{\trup
{1}{2}}]}\int
_0^r \mathrm{d}s \Bigl[ 1 + \sup
_{(\hat{s},y)\in[0,s]\times\R
^3}\E \bigl(\bigl\llvert X_n(\hat{s},y)\bigr
\rrvert ^p 1_{L_n(\hat{s})} \bigr) \Bigr]
\\
&\le& C \int_0^r \mathrm{d}s \Bigl[ 1 + \sup
_{(\hat{s},y)\in[0,s]\times\R
^3}\E \bigl(\bigl\llvert X_n(\hat{s},y)\bigr
\rrvert ^p 1_{L_n(\hat{s})} \bigr) \Bigr],
\end{eqnarray*}
where in the last inequality we have used that
$\sup_n  \{n^{3p} 2^{-np [{\trup{(3-\beta)}{2}}-{\trup
{1}{2}} ]} \}<\infty$.

%T_{n,4}
We now consider $T_{n,4}(t,r;x)$. With similar arguments as those used
in the analysis of $T_{n,3}(t,x)$ in Lemma~\ref{ls4.2}, we prove
%
%e4.13 #&#
\begin{equation}
\label{s4.26} T_{n,4}(t,r;x)\le C \int_0^r
\mathrm{d}s \Bigl[1+ \sup_{(\hat{s},y)\in
[0,s]\times\R^3}\E \bigl(\bigl\llvert
X_n(\hat{s},y) \bigr\rrvert ^{p}1_{L_n(\hat{s})} \bigr)
\Bigr].
\end{equation}
Finally, we notice that $T_{n,5}(t,r;x)$ is very similar to
$T_{n,4}(t,x)$ in Lemma~\ref{ls4.2}. With similar arguments as those
used in the analysis of this term, we have
%
%e4.14 #&#
\begin{equation}
\label{s4.27} T_{n,5}(t,r;x) \le C \int_0^r
\mathrm{d}s \Bigl[ 1 + \sup_{(\hat{s},y)\in
[0,s]\times\R^3} \E \bigl(\bigl\llvert
X_n(\hat{s},y) \bigr\rrvert ^p 1_{L_n(\hat{s})} \bigr)
\Bigr].
\end{equation}

Bringing together
\eqref{s4.22}, \eqref{s4.24}--\eqref{s4.27} yields
%
%e4.15 #&#
\begin{eqnarray}
\label{s4.28}
&&\E\bigl(\bigl |X_n(t,r;x)\bigr |^p1_{L_n(t)}
\bigr)
\nonumber
\\[-8pt]
\\[-8pt]
&&\quad\le C \biggl\{1+ \int_0^r \sup
_{(\hat
{s},y)\in[0,s]\times\R^3} \E \bigl( \bigl\{\bigl\llvert X_n(\hat{s},y)
\bigr\rrvert ^p+ \bigl\llvert X_n^-(\hat {s},y) \bigr
\rrvert ^p \bigr\}1_{L_n(\hat{s})} \bigr)\,\mathrm{d}s \biggr\}.
\nonumber
\end{eqnarray}
Notice that $X_n(t,t;x)=X_n(t,x)$. Hence, for $r:=t$, \eqref{s4.28}
tells us
%
%e4.16 #&#
\begin{eqnarray}
\label{s4.29}
&& \E\bigl(\bigl |X_n(t,x)\bigr |^p1_{L_n(t)}
\bigr)
\nonumber
\\[-8pt]
\\[-8pt]
&&\quad\le C \biggl\{1+ \int_0^t \sup
_{(\hat
{s},y)\in[0,s]\times\R^3} \E \bigl( \bigl\{\bigl\llvert X_n(\hat{s},y)
\bigr\rrvert ^p +\bigl\llvert X_n^-(\hat{s},y) \bigr
\rrvert ^p \bigr\}1_{L_n(\hat{s})} \bigr) \,\mathrm{d}s \biggr\}.\nonumber
\end{eqnarray}
Next, take $r:=t_n$ and remember that $X_n(t,t_n;x)=X_n^-(t,x)$. From
\eqref{s4.28}, and since $t_n\le t$, we obtain
%
%e4.17 #&#
\begin{eqnarray}
\label{s4.30}
&&\E\bigl(\bigl |X_n^-(t,x)\bigr |^p1_{L_n(t)}
\bigr)
\nonumber
\\[-8pt]
\\[-8pt]
&&\quad\le C \biggl\{1+ \int_0^t \sup
_{(\hat{s},y)\in[0,s]\times\R^3} \E \bigl( \bigl\{\bigl\llvert X_n(\hat{s},y)
\bigr\rrvert ^p+\bigl\llvert X_n^-(\hat {s},y) \bigr
\rrvert ^p \bigr\}1_{L_n(\hat{s})} \bigr) \,\mathrm{d}s \biggr\}.\nonumber
\end{eqnarray}
For $t\in[0,T]$, set
\[
\varphi_n(t)= \sup_{(s,y)\in[0,t]\times\R^3} \E \bigl[
\bigl(\bigl |X_n(s,y)\bigr |^p+\bigl |X_n^-(s,y)\bigr |^p
\bigr)1_{L_n(s)} \bigr].
\]
The inequalities \eqref{s4.29}, \eqref{s4.30} imply
$\varphi_n(t)\le C \{1+ \int_0^t \varphi_n(s)\,\mathrm{d}s \}$.
By Gronwall's lemma, this implies \eqref{s4.18}. Finally, the
inequality \eqref{s4.19} is a consequence of \eqref{s4.11} and \eqref{s4.18}
\end{pf}

\begin{appendix}\label{s5}
%%%%%%
%%%%%% appendix
%%%%%%
%s5 #&#
\section*{Appendix}

We start this section with a theorem on existence and uniqueness of
solution to a class of equations which in particular applies to \eqref
{s3.7}, and therefore also to \eqref{s1.6}, and to
\eqref{s3.6}. For related results, we refer the reader to \cite
{dalang}, Theorem~13, \cite{dalang-quer}, Theorem~4.3 and \cite
{ortiz2011}, Proposition~4.0.4.
In comparison with these references, here we state the theorem in
spatial dimension $d=3$, and we assume that $G$ is the fundamental
solution of the wave equation in dimension three.
\setcounter{remark}{0}
%th5.1 #&#
\begin{theorem}
\label{ts5.1}
Let $G$ denote the fundamental solution to the wave equation in
dimension three and $M$ a Gaussian process as given in the
\hyperref[s1]{Introduction}. Consider the stochastic evolution equation defined by
%
%e5.1 #&#
\setcounter{equation}{0}
\begin{eqnarray}
\label{s5.20} Z(t,x)&=&\int_0^t\int
_{\IR^3}G(t-s,x-y) \sigma \bigl(Z(s,y)\bigr)M(\mathrm{d}s,
\mathrm{d}y)
\nonumber
\\
&&{}+\bigl\langle G(t-\cdot,x-\ast)g\bigl(Z(\cdot,\ast)\bigr),H\bigr
\rangle_{\mathcal
{H}_T}
\\
&&{}+\int_0^t \bigl[G(t-s,
\cdot)\star b\bigl(Z(s,\cdot)\bigr)\bigr](x),
\nonumber
\end{eqnarray}
where the functions $\sigma, g, b\dvtx \IR\rightarrow\IR$ are Lipschitz
continuous.
\begin{enumerate}[(ii)]
\item[(i)] Assume that $H=\{H_t, t\in[0,T]\}$ is an $\hac$-valued
predictable stochastic process such that
$C_0:=\sup_\omega\Vert H(\omega)\Vert_{\mathcal{H}_T} <\infty$.

Then, there exists a unique real-valued adapted stochastic process $Z=\{
Z(t,x),\allowbreak   (t,x)\in[0,T]\times\IR^3\}$ satisfying \textup{\eqref{s5.20}}, a.s.,
for all $(t,x)\in[0,T]\times\R^3$. Moreover, the process $Z$ is
continuous in $L^2$ and satisfies
\[
\sup_{(t,x)\in[0,T]\times\IR^3}\E \bigl(\bigl \vert Z(t,x)\bigr \vert^p \bigr) \le
C<\infty,
\]
for any $p\in[1,\infty[$, where the constant $C$ depends among others
on $C_0$.
\item[(ii)] Assume that there exist an increasing sequence of events
$\{\Omega_n, n\ge1\}$ such that $\lim_{n\to\infty}\mathbb
{P}(\Omega_n)=1$, and that $H_n=\{H_n(t), t\in[0,T]\}$ is a sequence
of $\hac$-valued predictable stochastic processes such that
$C_n:=\sup_\omega\Vert H(\omega)1_{\Omega_n}(\omega)\Vert
_{\mathcal{H}_T} <\infty$. Then, the conclusion on existence and
uniqueness of solution to \textup{\eqref{s5.20}} stated in part (\textup{i}) also holds.
\end{enumerate}
The process $Z$ is termed a \textit{random field solution} to \textup{\eqref{s5.20}}.
\end{theorem}

\begin{pf*}{Sketch of the proof}
We start with part (i). Consider the
Picard iteration scheme
\[
Z^0(t,x) =0,\vadjust{\goodbreak}
\]
\begin{eqnarray*}
Z^{(k+1)}(t,x) &=&\int_0^t\int
_{\IR^3}G(t-s,x-y) \sigma \bigl(Z^{(k)}(s,y)\bigr)M(
\mathrm{d}s,\mathrm{d}y)
\\
&&{}+\bigl\langle G(t-\cdot,x-\ast)g\bigl(Z^{(k)} (\cdot,\ast)\bigr),H
\bigr\rangle _{\mathcal{H}_T}
 +\int_0^t
\bigl[G(t-s,\cdot)\star b\bigl(Z^{(k)} (s,\cdot)\bigr)\bigr](x),
\end{eqnarray*}
$k\ge0$.

Fix $p\in[2,\infty[$. First, we prove by induction on $k\ge0$ that
%
%e5.2 #&#
\begin{equation}
\sup_{(t,x)\in[0,T]\times\IR^3}\E \bigl(\bigl \vert Z^{k}(t,x)\bigr \vert
^p \bigr)\le C<\infty,
\end{equation}
with a constant $C$ independent of $k$.
Second, we prove that
\begin{eqnarray*}
&&\sup_{x\in\IR^3}\E \bigl(\bigl \vert Z^{(k+1)}(t,x)-Z^{(k)}(t,x)
\bigr \vert ^p \bigr)
\\
&&\quad \le C(1+C_0) \biggl[\int_0^t
\mathrm{d}s \sup_{y\in\IR^3}\E \bigl(\bigl \vert Z^{(k)}(s,y)-Z^{(k-1)}(s,y)
\bigr \vert^p \bigr) \biggr].
\end{eqnarray*}
With this, we conclude that the sequence of processes $\{Z^{(k)}(t,x),
(t,x)\in[0,T]\times\IR^3\}$, $k\ge0$ converges in $L^p(\Omega)$ as
$k\to\infty$, uniformly
in $(t,x)\in[0,T]\times\IR^3$. The limit is a random field that
satisfies the properties of the statement. We refer the reader to
\cite{dalang,dalang-quer,ortiz2011}, for more
details on the proof.

The proof of part (ii) is done by localizing the preceding Picard
scheme using the sequence $\{\Omega_n, n\ge1\}$.
\end{pf*}

In comparison with the equation considered in \cite{dalang-quer}, Theorem~4.3, \eqref{s5.20} has null initial conditions, and the
extra term
$\langle G(t-\cdot,x-\ast)g(Z(\cdot,\ast)),H\rangle_{\mathcal{H}_T}$.

Part (i) of Theorem~\ref{ts5.1} can be applied to \eqref{s1.6},
\eqref{s3.7}. Therefore, we have
%
%e5.3 #&#
\begin{equation}
\label{s5.21} \sup_{(t,x)\in[0,T]\times\IR^3}\E \bigl(\bigl \vert X(t,x)
\bigr \vert^p \bigr) <\infty.
\end{equation}
Let $\Omega_n=L_n(t)$ as given in \eqref{localization}. The sequence
$H_n:=w^n$ defined in \eqref{s3.3} satisfies the assumptions of part
(ii) of Theorem~\ref{ts5.1} (see \eqref{s3.101}). Therefore the
conclusion applies to the stochastic process solution of \eqref{s3.6}.

%%%%%%%%S-PROPERTY
%
%re5.2 #&#
\begin{remark}
\label{r5.1}
Set $Z^{(z)}(s,x)=Z(s,x+z)$, $z\in\R^3$. Similarly as in \cite
{dalang}, we can argue that the finite dimensional distributions of the
process $\{Z^{(z)}(s,x), (s,x)\in[0,T]\times\R^3\}$ do not depend on
$z$. This is a consequence from the fact that the martingale measure
$M$ has a spatial stationary covariance, and that the initial condition
of the SPDE vanishes.
\end{remark}

%%%%%%%%%%%%OTHER RESULTS
At several points, we have applied the following version of Gronwall's
lemma whose proof can be found in \cite{Ba-Si}, Theorem~4.9.
%
%le5.3 #&#
\begin{lemma}
\label{ls5.1}
Let $u$, $b$ and $k$ be nonnegative continuous functions defined on the
interval $J=[\alpha,\beta]$. Let $\bar p\ge0$, $\bar p\ne1$ and
$a>0$ be constants.
Suppose that
\[
u(t)\le a +\int_\alpha^t b(s) u(s) \,\mathrm{d}s+
\int_\alpha^t k(s) u^{\bar
p}(s) \,
\mathrm{d}s,\quad t\in J.
\]
Then
%
%e5.4 #&#
\begin{equation}
\label{s5.1} u(t)\le\exp \biggl(\int_\alpha^\beta
b(s) \,\mathrm{d}s \biggr) \biggl[a^{\bar
q} + \bar q \int
_\alpha^\beta k(s) \exp \biggl(-\bar q\int
_\alpha^s b(\tau) \,\mathrm{d}\tau \biggr) \,\mathrm{d}s
\biggr]^{\trup{1}{\bar q}},
\end{equation}
for every $t\in[\alpha,\beta_1)$, where $\bar q=1-\bar p$ and $\beta
_1$ is choosen so that the expression between $[\cdots]$ is positive in the
subinterval $[\alpha,\beta_1)$ ($\beta_1=\beta$ if $\bar q>0$).
\end{lemma}

In the proof of Theorem~\ref{ts3.1}, we have used the lemma below. For
its proof, we refer the reader to Lemma A.2 in \cite{milletss2},
with a trivial change on the spacial dimension
($d=3$ in \cite{milletss2}, while $d=4$ in Lemma~\ref{ls5.2}).

%le5.4 #&#
\begin{lemma}
\label{ls5.2}
Fix $[t_0,T]$ with $t_0\ge0$ and a compact set $K\subset\IR^3$.
Let $\{Y_n(t,x),(t,x)\in[t_0,T]\times K, n\ge1\}$ be a sequence of
processes and
$\{B_n(t), t\in[t_0,T]\}\subset\mathcal{F}$ be a sequence of adapted
events which, for every n, decreases in $t$. Assume that for every
$p\in\,]1,\infty[$ the following conditions hold:
\begin{enumerate}[(P2)]
\item[(P1)] There exists $\delta>0$ and $C>0$ such that, for any
$t_0\le t\le\bar{t}\le T$, $x,\bar{x}\in K$,
\[
\sup_n \E \bigl(\bigl |Y_n(t,x)-Y_n(
\bar{t},\bar{x})\bigr |^p 1_{B_n(\bar{t})} \bigr)\le C \bigl(|t-\bar{t}|+|x-
\bar{x}| \bigr)^{4+\delta}.
\]
\item[(P2)] For every $(t,x)\in[t_0,T]\times K$,
\[
\lim_{n\to\infty} \E \bigl(\bigl |Y_n(t,x)\bigr |^p
1_{B_n(t)} \bigr)=0.
\]
\end{enumerate}

Then, for any $\eta\in\,]0,\delta/p[$ and any $r\in[1,p[$,
\[
\lim_{n\to\infty} \E\bigl(\|Y_n\|_{\eta,t_0,K}^r
1_{B_n(T)} \bigr) =0.
\]
\end{lemma}
\end{appendix}

\section*{Acknowledgements}

The authors would like to thank Robert Dalang for bringing our
attention to \eqref{fundamental}.

This work was supported by the grant MICINN-FEDER MTM 2009-07203 from
the \textit{Direcci\'on General de
Investigaci\'on, Ministerio de Educaci\'on y Ciencia, Spain}.

%suskaldyti doi

% imsref loaded by audrone.aklyte, 2014-01-09 14:08:50
%

\printhistory

\end{document}